\documentclass[10pt,a4paper]{amsart}
\usepackage{amsmath, amssymb, amsfonts, amsthm, float, stmaryrd, epsfig, adjustbox}
\usepackage[abbrev,alphabetic]{amsrefs}
\usepackage[
  unicode=true,
  bookmarksdepth=3,
  bookmarksopen=true,
  pdfborder={0 0 0 [0 0]},
  colorlinks,
  breaklinks
]{hyperref}
\hypersetup{
  linkcolor=black,
  citecolor=black,
  urlcolor=blue
}
\usepackage[latin1,utf8]{inputenc}
\usepackage{textalpha}
\usepackage{color}
\usepackage[ps,all,arc,rotate]{xy}
\usepackage{bbm} 
\usepackage{todonotes}
\usepackage{amsrefs}
\usepackage{mathtools}
\usepackage{tikz-cd}
\usepackage{pdfsync}  
\usepackage{fancyhdr}
\usepackage{comment}

\usepackage{ifthen}
\newcounter{thmcount}
\newcommand*{\numberedtheorem}[3]{\theoremstyle{plain}\newtheorem*{makethm\thethmcount}{#1}
\ifthenelse{\equal{#2}{}}{\begin{makethm\thethmcount}#3\end{makethm\thethmcount}\stepcounter{thmcount}}
{\begin{makethm\thethmcount}[#2]#3\end{makethm\thethmcount}\stepcounter{thmcount}}}

\newtheorem{theorem}{Theorem}[section]
\newtheorem{corollary}[theorem]{Corollary}
\newtheorem{lemma}[theorem]{Lemma}

\newtheorem{proposition}[theorem]{Proposition}
\theoremstyle{definition}
\newtheorem{definition}[theorem]{Definition}
\newtheorem{eg}[theorem]{Example}
\newtheorem{convention}[theorem]{Convention}
\newtheorem{remark}[theorem]{Remark}
\theoremstyle{theorem}
\newtheorem{thmx}{Theorem}

\newcommand{\R}{\mathbb{R}}
\newcommand{\C}{\mathbb{C}}
\newcommand{\Z}{\mathbb{Z}}
\newcommand{\N}{\mathbb{N}}

\newcommand{\HH}{\mathbb{H}}
\newcommand{\M}{\textbf{M}}
\newcommand{\F}{\mathcal{F}}
\newcommand{\Homology}{{H}}

\newcommand{\inv}[1]{#1^{-1}}

\newcommand{\supp}{\mathrm{supp}}
\newcommand{\Hom}{\mathrm{Hom}}
\newcommand{\id}{\mathrm{id}}
\newcommand{\im}{\mathrm{im}}
\newcommand{\qker}{\mathcal{Q}ker}
\newcommand{\fd}{\mathrm{fd}}
\newcommand{\coarse}{\mathrm{coarse}}
\newcommand{\cd}{\mathrm{cd}}

\newcommand{\ccd}{\mathrm{ccd}}

\newcommand{\diam}{\mathrm{diam}}
\newcommand{\FP}{\mathrm{FP}}

\newcommand{\PD}{\mathrm{PD}}
\newcommand{\Hd}{\mathrm{Hd}}

\makeatletter
\newcommand{\myitem}[2]{% #1 = texto visible del label, #2 = clave de \label
  \item[#1]%
  \phantomsection% ancla para hyperref
  \protected@edef\@currentlabel{#1}% forzar lo que \ref devolverá
  \label{#2}%
}
\makeatother
\def\MR#1{}
\title{Quasimorphisms and Poincar\'e duality in dimension $3$}
\author[Paula Heim]{Paula Heim}
\address{Max Planck Institute for Mathematics in the Sciences \& Mathematical Institute, University of Oxford}
\email{paula.heim@maths.ox.ac.uk}
\author[William Thomas]{William Thomas}
\address{Mathematical Institute, University of Oxford}

\email{william.thomas@maths.ox.ac.uk}

\begin{document}

\begin{abstract}

We study $\PD^3$ groups which admit an unbounded quasimorphism to $\R$ with coarsely connected quasikernel. We show that such a group $G$ must either arise as the fundamental group of a torus or Klein-bottle bundle over~$S^1$, or be quasiisometric to a Riemannian manifold~$(\R^3,g)$ of bounded geometry, with the quasikernel being coarsely equivalent to~$\HH^2$. If~$G$ is moreover hyperbolic, it admits a faithful action on $S^1$ by quasisymmetric homeomorphisms. Our approach features a coarse generalisation of Shapiro's lemma, and a new definition of homological isoperimetric inequalities for metric spaces; these tools make use of Margolis's framework for coarse homological algebra \cite{MR123}.

\end{abstract}

\maketitle

\section{Introduction}

A satisfying example of the interplay between group theory and the topology of $3$-manifolds is the main result of Stallings's seminal paper \cite{MR158375}, which offers the following group theoretic characterisation of $3$-manifolds fibring over $S^1$.

\begin{theorem}[Stallings's Fibring Theorem]

Let $M$ be a compact irreducible $3$-manifold and $\psi\colon \pi_1(M)\to\Z$ be a nontrivial homomorphism. If $\ker(\psi)\ne\Z/2$ is finitely generated, then $\psi$ is induced by a fibring of $M$ over $S^1$ and $\ker(\psi)\cong \pi_1(\Sigma)$ is the fundamental group of a compact surface. 
    
\end{theorem}

One can ask what happens in the setting of Stallings's theorem when homomorphisms $\psi\colon \pi_1(M)\to\Z$ are replaced with quasimorphisms $\phi\colon \pi_1(M)\to \R$. Roughly speaking, these are functions that are at bounded distance from being a group homomorphism, and the correct analogue of finite generation of $\ker(\psi)$ is \textit{coarse connectedness} of the quasikernel $\qker(\phi)$. The motivation for considering such a generalisation is twofold. 

\begin{enumerate}
   
    \item By the work of Epstein--Fujiwara \cite{MR1452851},  quasimorphisms are in abundance in the world of hyperbolic groups. It is thus not unreasonable to expect that many such groups (at least virtually) admit those with coarsely connected quasikernel.
    
    \item  Quasimorphisms with coarsely connected quasikernel arise naturally on the fundamental groups of $3$-manifolds that \textit{slither over $S^1$}, a structure introduced by Thurston as a generalisation of fibring \cite{MR473827927}. In that paper, examples were given of closed hyperbolic $3$-manifolds that cannot fibre over~$S^1$, but do still admit slitherings.
\end{enumerate}

In view of this, we prove the following coarse analogue of Stallings's Fibring Theorem in the setting of 3-dimensional Poincaré-duality groups.

\begin{thmx}\label{thm:manifold}
    Let $G$ be a $\PD^3$ group and $\phi\colon G\to\R$ be a nontrivial homogeneous quasimorphism. If $\qker(\phi)$ is coarsely connected, then either $G$ is the fundamental group of a torus or Klein-bottle bundle over $S^1$, or $\qker(\phi)$ is coarsely equivalent to $\HH^2$. In the latter case, $G$ is quasiisometric to a Riemannian manifold $(\R^3,g)$ of bounded geometry, and is hence finitely presented.
\end{thmx}
Here, we understand a complete Riemannian manifold to have  \textit{bounded geometry} if there exists a uniform lower bound on the injectivity radius as well as uniform upper and lower bounds on sectional curvature.  

Our result take inspiration from an unpublished article of Danny Calegari \cite{calegari2010boundedcochains3manifolds}, who investigated closed $3$-manifolds admitting quasimorphisms of a similar nature. His approach is continued in collaboration with Ino Loukidou through the theory of \textit{zippers} \cite{calegari2024zippers}. In the case that $\phi$ is a homomorphism, our theorem follows from work of Hillman \cite[Theorem $1.19$]{MR1943724} combined with the classification of $\PD^2$ groups obtained by Eckmann et al. \cite{MR604709}   \cite{MR699010} (one also requires the Dehn--Nielsen--Baer Theorem \cite{MR1555345} and its nonorientable analogue due to Mangler \cite{MR1545786}). This has been written down more explicitly and in greater generality by Bridson--Kielak--Kudlinska \cite{MR4861506}.

An old question of Wall \cite[Question G$2$, p.$391$]{wall1979homological} asks whether $\PD^3$ groups ought actually to be fundamental groups of closed aspherical $3$-manifolds. Theorem~\ref{thm:manifold} can be seen as evidence in favour of an affirmative answer to this question. Restricting to the class of torsion-free hyperbolic groups, work of Bestvina \cite{MR1381603} (building on Bestvina--Mess \cite{MR1096169}) shows that Wall's question becomes equivalent to the Cannon conjecture \cite{MR1130181}, which says that a hyperbolic group $G$ with $\partial G\cong S^2$ should be virtually Kleinian. The following is then an immediate consequence of Theorem \ref{thm:manifold}.

\begin{corollary}\label{cannon}
    Let $G$ be a torsion-free hyperbolic group for which~$\partial G\cong S^2$. If $G$ admits a nontrivial homogeneous quasimorphism to $\R$ with coarsely connected quasikernel, then it is quasiisometric to a Riemannian manifold~$(\R^3,g)$ of bounded geometry.
\end{corollary}

We remark here that if the Riemannian manifold obtained from Corollary \ref{cannon} is in fact quasiisometric to $\HH^3$, then Sullivan's Theorem \cite{MR624833} implies that $G$ is virtually Kleinian; this is particularly relevant given the ubiquity of quasimorphisms on hyperbolic groups.

To prove Theorem \ref{thm:manifold}, we invoke Margolis's recently introduced framework of \textit{coarse cohomology}  \cite{MR123}, which builds upon previous work of Roe \cite{MR1147350} and Kapovich--Kleiner \cite{MR2168506}, to associate to metric spaces a cohomology theory that is in many ways very similar to ordinary group cohomology. A key starting point to our approach is the following coarse analogue of Shapiro's lemma, which crucially relates coarse cohomology and honest group cohomology.

\begin{thmx}\label{Intro coarse Shapiro}
    Given a coarse embedding $\imath\colon X\to G$ of a metric space into a finitely generated group, there is an isomorphism of cohomology groups \[\Homology^\ast_{\coarse}(X;\textbf{M})\xrightarrow{\sim} \Homology^\ast(G; \Hom_{\fd}( R[G]|_X, \textbf{M}))\]for every $R$-module $\M$ over the metric space $X$.
\end{thmx}

     The $R[G]$-module $\Hom_{\fd}(R[G]|_X,\M)$ appearing in the statement of Theorem \ref{Intro coarse Shapiro} is a coarse analogue of the coinduced module in group cohomology. We delegate the task of defining this and the other objects in the statement to Section \ref{coarse cohomology intro}, but note that they all live inside the theory of modules over metric spaces introduced in \cite{MR123}, and are naturally parallel to standard constructions in group cohomology. Theorem~\ref{Intro coarse Shapiro} is preceded by similar Shapiro-type lemmata in group cohomology due to Sauer~\cite{MR2231471}, and Li~\cite{MR3868227}, who additionally assume that the coarsely embedded space is itself a finitely generated group.

      If $\phi\colon G\to\R$ is a nontrivial homogeneous quasimorphism, and Theorem \ref{Intro coarse Shapiro} is applied with $X=\qker(\phi)$, we show that the right hand side is closely related to the group cohomology of $G$ with coefficients in the Novikov ring $\widehat{R[G]}^\phi$; this builds on work in the case that $\phi$ is a homomorphism due to Hillman--Kochloukova \cite{MR2282258} and  Fisher \cite{MapstoZ}. We then extend a result of Heuer--Kielak \cite{QBNS} to higher dimensions in proving Theorem \ref{Novikov homology new}, a Sikorav-style theorem relating the coarse finiteness properties of $\qker(\phi)$ with group homology over $\widehat{R[G]}^\phi$. Exploiting Poincar\'e duality, we are able to combine Theorem \ref{Intro coarse Shapiro} and \cite[Theorem $B$]{MR123} to show that if $G$ is a $\PD^3$ group as in Theorem \ref{thm:manifold}, then~$\qker(\phi)$ is a \textit{coarse $\PD^2$ space} in the sense of Kapovich--Kleiner \cite{MR2168506} and Margolis \cite{MR123}. We provide the following partial classification of such spaces, where amenability here is understood to mean the existence of a F\o lner sequence.

\begin{thmx}\label{intro-characterising PD^2}
    If $X$ is a discrete quasigeodesic bounded-geometry coarse $\PD^2(R)$ space over a commutative ring $R$, then it is either amenable or quasiisometric~to~$\HH^2$.
\end{thmx}

A similar characterisation is obtained in work of Kapovich--Kleiner \cite{Quasiplanes} (see also \cite{MR49321}) under the assumption that $X$ is coarsely simply-connected; we however require such a result without these finiteness hypotheses. Our approach introduces a theory of homological isoperimetric inequalities over the projective resolutions over metric spaces defined in \cite{MR123}, thereby providing a natural setting for the development of homological methods in the spirit of Kielak--Kropholler~ \cite{kielak2021isoperimetric}. Complementing Theorem~\ref{intro-characterising PD^2}, we prove in Proposition~\ref{amenable quasikernel implies amenable group} that if~$\qker(\phi)$ is amenable, then so is the ambient group~$G$, from which it follows that such a quasimorphism is an honest homomorphism.

We consider now the case that $\qker(\phi)$ is coarsely equivalent to $\HH^2$. There is then a natural action of $G$ on $\partial\HH^2\cong S^1$, first investigated in \cite{calegari2010boundedcochains3manifolds}, which we prove to be faithful when $G$ is hyperbolic.

\begin{thmx}\label{introthm:circle_action}
    Let $G$ be a hyperbolic $\PD^3$ group and $\phi\colon G\to\R$ be a nontrivial homogeneous quasimorphism. If $\qker(\phi)$ is coarsely connected, then $G$ acts faithfully on $S^1$ by quasisymmetric homeomorphisms.
\end{thmx}

 A closely related result has recently been obtained in \cite{calegari2024zippers}. For an arbitrary hyperbolic group admitting such a quasimorphism they obtain a structure on the Gromov boundary called a \textit{zipper}, and when this group is the fundamental group of a closed hyperbolic $3$-manifolds this gives rise to an action on a universal circle.
 
To prove Theorem \ref{thm:manifold}, it remains to construct, under the assumption that the quasikernel is coarsely equivalent to $\HH^2$, a Riemannian manifold of bounded geometry quasiisometric to $G$. We do this by encoding the coarse geometry of the group in a certain self-quasiisometry of $\HH^2$, which should be thought of as the analogue of the monodromy of a closed fibred $3$-manifold with fibre a hyperbolic surface. The celebrated Douady--Earle extension allows one to straighten such a self-quasiisometry of $\HH^2$ to a bi-Lipschitz diffeomorphism, from which the desired manifold can roughly speaking be built as an iterated mapping cylinder.

The paper is structured as follows. Section~\ref{section:preliminaries} contains preliminary
material, where in particular we give an introduction to Margolis's theory of coarse cohomology in Subsection~\ref{coarse cohomology intro}. 
We then prove Theorem~\ref{Intro coarse Shapiro}, a coarse version of Shapiro's lemma, in Section~\ref{section:shapiros_lemma}, which we link to Novikov cohomology in Section~\ref{section:novikov_cohomology}. Section~\ref{novikov homology section} relates Novikov homology to the coarse finiteness properties of the quasikernel. Homological isoperimetric inequalities are introduced in Section~\ref{section:corase_pr2_spaces}, which are then used to classify coarse $\PD^2$ spaces, proving Theorem~\ref{intro-characterising PD^2}.
Section~\ref{section:group_theory_and_the_quasikernel} generalises some group theoretic results to the setting of quasimorphisms.
In Section~\ref{section:quasimorphisms_on_pd3_groups}, we combine all of the previous sections to study the quasikernels of quasimorphisms on $\PD^3$ groups. Section~\ref{section:circle_action} uses this structure to obtain a faithful circle action for Theorem~\ref{introthm:circle_action}. Finally, we construct Riemannian manifolds quasiisometric to $\PD^3$ groups in Section~\ref{section:constructing_Riemannian_manifolds} and prove that they are of bounded geometry in Section~\ref{section:bounded_geometry}, completing the proof of Theorem~\ref{thm:manifold}. 

Shortly before the completion of this article, Calegari--Loukidou released a preprint \cite{calegari2026catherinewheels} exhibiting new connections between quasimorphisms with coarsely connected quasikernel on the fundamental group of a closed hyperbolic $3$-manifold, there called \textit{uniform} quasimorphisms, and certain dynamical structures on the manifold itself. Combined with recent work of Calegari--Zung \cite{2648284748}, it is for instance shown that the existence of such a quasimorphism is equivalent to having a pseudo-Anosov flow with no perfect fits. They also obtain that each connected component of the space of uniform quasimorphisms inside the space of all quasimorphisms is an open convex cone, yielding a natural analogue to the cones on the (open) fibred faces of the Thurston norm ball in the classical setting of fibring over $S^1$.

\subsection*{Acknowledgements}  
The second author would like to thank Martin Bridson and Dawid Kielak for helpful conversations and advice on this work. 

We would also like to thank Dawid again for carefully reading and commenting on previous versions of this article, Anna Wienhard for clarifying the proof of Lemma~\ref{lem:boundedness_of_derivatives_of_f}, Alex Margolis for insightful comments and explanations, Shaked Bader, Sam Ketchell, Harry Petyt and Ralfs Pundurs for fruitful discussions, and Sam Fisher for sharing ideas at the conception of the project. The second author was supported by an Engineering and Physical Sciences Research Council studentship (Project Reference 2928333).

\section{Preliminaries}\label{section:preliminaries}

\subsection{Conventions}
All rings considered shall be unital and associative and $R$ will always denote a commutative ring. Modules will be acted on on the left unless explicitly stated otherwise. finitely generated groups $G$ will be equipped with a left-invariant word metric with respect to some finite generating set, the choice of which is unique up to quasiisometry.
Whenever there is no risk of ambiguity, $d$ will denote a metric of the metric space under consideration.
Graphs are metrised with the usual graph metric that assigns length~1 to edges unless stated otherwise, and subspaces of metric spaces are equipped with the subspace metric.

\subsection{Coarse geometry} 
We first introduce some notation. Let $Y$ be a subspace of a metric space $(X,d)$. For $r\in\R_{\geq 0}$, we define the \textit{metric $r$-neighbourhood of $Y$ in $X$} by $N_r(Y):=\{y\in X:d(y,Y)\leq i\}$; the \textit{metric ball of radius $r$ around $x\in X$} is $B_r(x)\coloneq N_r(\{x\})$. A subspace $Y\subseteq X$ is \textit{coarsely dense} if there is $r\geq 0$ such that $N_r(Y)=X$.

Given metric spaces $(X,d_X)$, $(Y,d_Y)$ and maps $f,g\colon X\to Y$, we say that $f$ and~$g$ are \textit{close} or \textit{at bounded distance} if there is $C\geq 0$ such that $d_Y(f(x),g(x))\leq C$ for all $x\in X$; we denote this as
\begin{equation*}
    d_\infty(f,g)\leq C.
\end{equation*}

     A map $f\colon X\to Y$ is called a \textit{coarse embedding} if there is a pair of proper functions $\rho_{\pm}\colon \R_{\geq0}\to \R$ such that for all $x_1, x_2\in X$ 
    \begin{equation*}
        \rho_- (d_X(x_1,x_2))\leq d_Y(f(x_1),f(x_2))\leq\rho_+ (d_X(x_1,x_2)).
    \end{equation*}
    We call $\rho_{\pm}:\R_{\geq 0}\to\R_{\geq 0}$ the \textit{control functions} of $f:X\to Y$.
    If in addition there exists a coarse embedding $g\colon Y\to X$ such that $d_\infty(f\circ g,\id_Y)<\infty$ and $d_\infty(g\circ f,\id_X)<\infty$, we say that $f\colon X\to Y$ is a \textit{coarse equivalence}.

     A coarse embedding between metric spaces $f\colon X\to Y$ is a \textit{quasiisometric embedding} if the control functions~$\rho_{\pm}$ are affine linear; we will call the coefficients of such an affine linear function \textit{quasiisometry constants}. If~$f$ is also a coarse equivalence, it is called a \textit{quasiisometry}.
     It is well-known that a quasiisometric embedding is a quasiisometry if and only if its image is coarsely dense. A metric space is called \textit{quasigeodesic} (resp. \textit{coarse geodesic}) if it is quasiisometric (resp. coarsely equivalent) to a geodesic metric space.

A prominent example of a coarse embedding that is not necessarily a quasiisometric embedding is given by a subgroup inclusion between finitely generated groups, with both equipped with word metrics. Contrasting this, a coarse equivalence between quasigeodesic metric spaces is a quasiisometry, so all coarsely equivalent such groups are quasiisometric.

Let $Y$ be a discrete bounded-geometry metric space and $\imath \colon X\to Y$ be a coarse embedding. A  \textit{closest-point projection map} $\pi\colon Y\to X$ is a map sending every $y\in Y$ to some $\pi(y)\in X$ such that \[d_Y(\imath(\pi(y)), y)=\min_{x\in X}\{d_Y(x,y)\}.\] 

For a metric space $X$ and subsets $A,B\subseteq X$, we define the \textit{Hausdorff distance} \[d_{\Hd}(A,B)\coloneq \max\{\sup_{a\in A} d(a,B),\sup_{b\in B}d(A,b)\}.\]
The following two lemmata are straightforward, so we choose to omit their proofs. Note that the assumptions of discreteness and bounded geometry are not strictly necessary, but they are made in order to avoid a discussion of closest-point projections in a more general setting, which will not be required in this paper. 
\begin{lemma}\label{Projecting is a coarse equivalence}
    Let~\(X, Y\subseteq Z\) be subspaces of a discrete bounded-geometry metric space~\(Z\) at finite Hausdorff distance $d_{\Hd}(X, Y)<\infty$ from one another. Then any closest-point projection map \(\pi\colon X\to Y\) induces a quasiisometry between the metric spaces~$X$ and~$Y$, whose quasiisometry constants depend only on $d_{\Hd}(X,Y)$.
\end{lemma}

\begin{lemma}\label{lem:composition_of_nearest_point_projections_is_close_to_nearest_point_projection}
    Let $X, Y, Z$ be subspaces of a discrete bounded-geometry metric space $W$ such that there exists a constant~$C>0$ with 
    \begin{equation*}
        \max\{d_{\Hd}(X, Y), d_{\Hd}(X, Z), d_{\Hd}(Y, Z)\}\leq C.
    \end{equation*}
    If $\pi_{X, Y}\colon X\to Y$, $\pi_{Y, Z}\colon Y\to Z$ and $\pi_{X, Z}\colon X\to Z$ are closest-point projections, then we have 
    \begin{equation*}
        d_\infty(\pi_{Y,Z}\circ \pi_{X,Y})\leq 3C.
    \end{equation*}

\end{lemma}

    We now consider properties of metric spaces. A discrete metric space $X$ has \emph{bounded geometry} if for all $r\geq 0$ there exists $C(r)>0$ such that
    \begin{equation*}
        |B_r(x)|\leq C(r)
    \end{equation*}
    for all points $x\in X$. A space $X$ is \textit{coarsely homogeneous} if for every $x,y\in X$ there is a coarse equivalence $\phi\colon X\to X$ such that $\phi(x)=y$, and the control functions of $\phi$ can be chosen independently of $x$ and $y$.

   The \textit{Rips complex} $P_r(X)$ of a metric space $X$ is the flag simplicial complex with $0$-simplices given by the points of $X$, and $n$-simplices spanning any subset $\{x_0,\dots,x_n\}\subset X$ with $d(x_i,x_j)\leq r$ for all $i,j\in\{0,\dots,n\}$. For $X\subseteq Y$ and $r\leq s$ there are natural inclusions $P_r(X)\subseteq P_s(Y)$. We define the \textit{Rips graph} $P_r^1(X)$ to be the $1$-skeleton of the Rips complex $P_r(X)$, which will be metrised with the usual graph metric.

\subsection{Coarse finiteness properties}

In this subsection we introduce the well-known theory of coarse homological finiteness properties of metric spaces. 

A metric space is \textit{coarsely connected} if $P_r(X)$ is connected for some $r\geq 0$. It is an observation, often attributed to Gromov, that a subgroup of a finitely generated group is itself finitely generated if and only if it is 
coarsely connected as a subspace. This can be generalised to define homological finiteness of metric spaces as follows.

\begin{definition}\label{Coarse acyclicity}
     A metric space $X$ is \textit{coarsely $(n-1)$-acyclic over $R$} if for all $r\in\R_{\geq 0}$ there is $s\geq r$ such that the map on reduced simplicial homology \[\tilde{H}_i(P_r(X),R)\to \tilde{H}_i(P_s(X),R)\]is trivial for all $i\leq n-1$.
\end{definition}

The work of Alonso \cite{MR1293049} implies that this definition is invariant under coarse equivalence. The following is implicit in Brown's criterion \cite[Theorem $2.2$]{MR885095}. 

\begin{proposition}\label{FP_n and coarse acyclicity}
    If $G$ is a finitely generated group equipped with a word metric, then it is of type $\FP_n(R)$ if and only if it is coarsely $(n-1)$-acyclic over $R$.
\end{proposition}

Following Gromov \cite{MR1253544}, Margolis uses in \cite{MR123} the more restrictive notion of coarse \textit{uniform} acyclicity, see Definition \ref{coarse uniform acylicity}, which turns out to have an equivalent definition in terms of \textit{projective resolutions over metric spaces} analogous to that of type $\FP_n(R)$ for groups. It is noted in Lemma \ref{Uniform acyclicity vs acyclicity} that for discrete coarsely homogeneous bounded-geometry metric spaces, which are those we shall mainly be concerned with in this article, coarse acyclicity and coarse uniform acylicity agree, so we will often not require the more technical latter definition.

We now observe that for sufficiently nice coarsely connected metric spaces, there is an analogue of the Cayley graph associated to a finitely generated group. The following is a combination of \cite[Propositions $2.5$ and $2.18$]{MR3816385} and Lemma~\ref{Uniform acyclicity vs acyclicity}.

\begin{proposition}\label{X and Rips graph}
    Let $X$ be a discrete coarsely homogeneous and coarsely connected bounded-geometry metric space. There is $r\geq 0$ such that the Rips graph $P_r^1(X)$ is connected and the inclusion $X\to P_r^1(X)$ is a coarse equivalence.
\end{proposition}

\subsection{Quasimorphisms and quasikernels}

We now introduce homogeneous quasimorphisms and their quasikernels, which will be some of the main objects of study in this article.

\begin{definition}
    A map $\phi\colon G\to\R$ is called a \textit{quasimorphism} if there exists $D\geq0$ such that for all $g,h\in G$\[|\phi(gh)-\phi(g)-\phi(h)|\leq D.\]We call the minimal such $D$ the \textit{defect of $\phi$}, denoted $D(\phi)$. If $\phi$ is a homomorphism on cyclic subgroups, we say that it is 
    \textit{homogeneous}. Such a quasimorphism will be called \textit{nontrivial} if it is nonzero.
\end{definition}

    \begin{remark}
       By \cite[Lemma $2.21$]{calegari2009scl}, every quasimorphism is at bounded pointwise distance from a unique homogeneous one, so there is no loss of generality in considering only the latter, as we shall do for the remainder of this article. The definition of the quasikernel for a general quasimorphism can be pulled back from its unique homogenisation, from which one can easily see the equivalence of the statements in the abstract of this article and those in the main theorems.
    \end{remark}

    \begin{definition}
        For $\phi\colon G\to\R$ a homogeneous quasimorphism, we define the \textit{quasikernel} of $\phi$ by \[\mathcal{Q}ker(\phi)\coloneq \phi^{-1}([-2D(\phi),2D(\phi)]).\] We will always equip this with the subspace metric inherited from $G$.
    \end{definition}

\begin{lemma}\label{Good element for quasimorphisms}
    Given a group $G$ and a nontrivial homogeneous quasimorphism $\phi$, there is an element $a\in G$ such that $-3D(\phi)\leq\phi(a)\leq-2D(\phi)$. 
\end{lemma}

\begin{proof}
    Recall that 
    \(D(\phi)=\underset{g, h\in G}{\sup}\{|{\phi([g,h])}|\}\) by~\cite[Lemma~2.24]{calegari2009scl} and choose a commutator $b\in G$ such that $2/3D(\phi)\leq |\phi(b)|\leq D(\phi)$; after possibly replacing $b$ with its inverse, we may assume that $-D(\phi)\leq \phi(b)\leq -2/3D(\phi)$. We conclude by setting $a\coloneq b^3$ so that $-3D(\phi)\leq \phi(a)\leq -2D(\phi)$.
\end{proof}

We now establish a key relationship between the distance of a group element~$g$ to the quasikernel and the value of $\phi(g)$. 
\begin{proposition}\label{filtration vs quasimorphism-moved}
Let $G$ be a finitely generated group and $\phi\colon G\to\R$ be a nontrivial homogeneous quasimorphism.
Then there exist proper increasing functions $\lambda, \mu\colon\N\to\N$ such that 
\begin{enumerate}
    \item 
    for every $g\in N_i(\qker(\phi))$ we have $|\phi(g)|\leq \lambda(i)$, and
    \item 
    for $g\in G$ with $|\phi(g)|\leq iD(\phi)$ we have $g\in N_{\mu(i)}(X)$.
\end{enumerate}
\end{proposition}

\begin{proof}
   Let $X=\qker(\phi)$. For any $g\in N_i(X)$, there is by definition a sequence of group elements $x,  xs_{n_1}, xs_{n_1}s_{n_2},\dots,xs_{n_1}s_{n_2}\dots s_{n_i}=g$ where $s_{n_j}\in S$ is the symmetric generating set associated to our choice of left-invariant metric. As $\phi\colon G\to \R$ is a quasimorphism, we have \[|\phi(g)-(\phi(x)+\phi(s_{n_1})+\dots+\phi(s_{n_i}))|\leq (i+1)D(\phi).\]Note that as $S$ is a finite set, we can bound $|\phi(s_{n_j})|$ uniformly from above by some $K\in \R_{\geq 0}$. It follows that \[|\phi(g)|\leq iK+(i+3)D(\phi),\] and we conclude (1) by setting $\lambda(i)\coloneq iK+(i+3)D(\phi)$.
   
   To prove (2), we construct the function $\mu:\N\to\N$ by induction on $i$. The case when $i\leqslant 2$ is immediate, setting $\mu(0)=0$. 
   Let $i\geq 3$  and choose $g\in G$ such that $(i-1)D(\phi)<\phi(g)\leq iD(\phi)$ for some $x\in X$ (the case of $\phi(g)$ negative is analogous).
 
   Choose $a\in G$ as in Lemma \ref{Good element for quasimorphisms}, then $(i-5)D(\phi)<\phi(ga)\leq (i-1)D(\phi)$, so by the induction hypothesis we have $ga\in N_{\mu(i-1)}(X)$ for some function $\mu\colon \{1,\dots,i-1\}\to\N$. But then \[d(g, X)\leq d(g,ga)+d(ga, X)\leq d(a,1) +\mu(i-1)\eqcolon \mu(i),\] from which we conclude.
\end{proof}

\begin{corollary}\label{cor:filtration_vs_quasimorphism-quantitative}
    There exist proper increasing functions $\Lambda_{\pm}\colon \R_{\geq 0}\to \R_{\geq 0}$ with \[
        \Lambda_-(|\phi(g)-\phi(h)|)\leq d(g,h.\qker(\phi))\leq \Lambda_+(|\phi(g)-\phi(h)|)
  \]
    for every $g,h\in G$.
\end{corollary}

The following basic observation allows us to view the quasikernel as coarsely being a normal subgroup.

 \begin{lemma}\label{Normal subgroupness}
  Any nontrivial homogeneous quasimorphism $\phi\colon G\to\R$ is constant on conjugacy classes. In particular, $g\qker(\phi)g^{-1}=\qker(\phi)$ for every $g\in G$.
   \end{lemma}

   \begin{proof}
       It suffices to note that for every $g,h\in G$, \[|\phi(ghg^{-1})-\phi(h)|=\frac{1}{n}|\phi(gh^ng^{-1})-\phi(h^n)|\leq \frac{2D(\phi)}{n}\] for every $n\in\N$. 
   \end{proof}

Our next lemma will later be used to understand certain groups as having the coarse geometry of an iterated mapping cylinder over the quasikernel of a quasimorphism. 

\begin{lemma}\label{G as a disjoint union of level sets}
    Let $G$ be a finitely generated group and $\phi\colon G\to\R$ be a nontrivial homogeneous quasimorphism. There is a constant $K>0$ depending only on $D(\phi)$ such that if $c\in G$ has $\phi(c)>K$, the following hold. 

    \begin{enumerate}

        \item The translates $\{c^i.\qker(\phi)\}_{i\in\Z}$ are pairwise disjoint.
        \item The subspace $\bigsqcup_{i\in\Z}c^i.\qker(\phi)$ is coarsely dense in~$G$.

    \end{enumerate}
\end{lemma}

\begin{proof}

For any $g\in c^i.\qker(\phi)$ we have the inequality $|\phi(g)-i\phi(c)|\leq 3D(\phi)$. In particular, if $\phi(c)>0$ is chosen to be large, it is not possible that $\phi(c^ix)=\phi(c^jy)$ for some $x,y\in \qker(\phi)$ and $i\ne j$, so the translates $\{c^i.\qker(\phi)\}_{i\in\Z}$ are pairwise disjoint. Part $(2)$ follows directly from Corollary \ref{cor:filtration_vs_quasimorphism-quantitative}.
\end{proof}

We conclude with a basic observation.

\begin{lemma}\label{coarselyhomogeneous}

 If $G$ is a finitely generated group and $\phi\colon G\to\R$ is a nontrivial homogeneous quasimorphism, then $\qker(\phi)$ is a coarsely homogeneous metric space.
    
\end{lemma}

\begin{proof}
    First note that for every $p,q\in \qker(\phi)$ there is $g\in G$ such that $g.p=q$. As $\phi:G\to\R$ is a quasimorphism, we see that $|\phi(g)|\leq 5D(\phi)$ and further that $|\phi(gx)|\leq 8D(\phi)$ for every $x\in \qker(\phi)$. But now by Corollary \ref{cor:filtration_vs_quasimorphism-quantitative} there is a function $\eta\colon \N\to\N$ such that \[d_{\Hd}(g.\qker(\phi),\qker(\phi))\leq\eta(D(\phi)),\] where $\eta$ is independent of $g,x\in G$. We consider the following composition map $f_{g}\colon \qker(\phi)\to \qker(\phi)$ given by $f_g(x)\coloneq \pi_g(g.x)$ where $\pi_g\colon g.\qker(\phi)\to \qker(\phi)$ is the closest-point projection map. Amongst group elements $g\in G$ such that $d_{\Hd}(g.\qker(\phi),\qker(\phi))\leq\eta(D(\phi))$, the map $f_g\colon \qker(\phi)\to \qker(\phi)$ is readily a coarse equivalence as it is the composition of an isometry and a closest-point projection between subspaces at uniformly bounded Hausdorff distance from one another. By Lemma~\ref{Projecting is a coarse equivalence} the control functions can be chosen uniformly, from which we conclude.
\end{proof}

\subsection{Group cohomology}

    Let~$M$ be an $R[G]$-module. A \textit{projective resolution} of~$M$ over $R[G]$ is an exact sequence \[
        \cdots \to P_1\to P_0\to M\to 0
    \] with $P_i$ a projective $R[G]$-module for all~$i$.

    A group $G$ is of type $\FP_n(R)$ if there is a projective resolution of $R$ over $R[G]$ such that $P_i$ is a finitely generated $R[G]$-module for $i\leq n$.

\begin{definition}
    Let~$G$ be a group, $R$ a ring and $P_\bullet\to R\to 0$ be a projective resolution of $R$ over $R[G]$. If $M$ is a left $R[G]$-module and $N$ a right $R[G]$-module, we define the \textit{group cohomology of $G$ with coefficients in $M$} by \[\Homology^\ast (G;M)\coloneq  \Homology^\ast(\Hom_{R[G]}(P_\bullet,M))\]and the \textit{group homology of $G$ with coefficients in $N$} by\[\Homology_\ast(G;N)\coloneq  \Homology_\ast(N\otimes_{R[G]}P_\bullet).\]
    
\end{definition}

\begin{definition}\label{def: PD group}
    A group $G$ is a \emph{Poincar\'e-duality group of dimension $n$ over $R$} ($\PD^n(R)$ group) if it is of type $\FP_n(R)$ and there are $R$-module isomorphisms \[H^i(G;R[G])\cong \begin{cases}
    R&\text{  if } i=n\\
    0  &\text{  if } i\leq n-1.
\end{cases}
\]
If $H^n(G;R[G])$ has the trivial $R[G]$-module structure then $G$ is said to be \textit{orientable}.
\end{definition}

Given a subgroup $H\leqslant G$, there is a well known-relation between the (co)homology of $H$ and that of $G$, usually referred to as \textit{Shapiro's lemma}. We state only the cohomological version, as this is what we shall generalise.

\begin{lemma}[Shapiro's lemma]\label{thm: Shapiros_lemm}
Given \(H\leqslant G\) a subgroup and \(M\) a left \(R[H]\)-module, there is an isomorphism\[
\Homology ^\ast(H;M)\xrightarrow{\sim} \Homology ^\ast(G;\Hom_{R[H]}(R[G],M)),
\]where the left action of $R[G]$ on $\Hom_{R[H]}(R[G],M)$ is given by $(g.f)(x)\coloneq f(xg)$.
\end{lemma}

\subsection{Coarse cohomology}\label{coarse cohomology intro} 

In this section we introduce coarse cohomology in the sense of Margolis \cite{MR123}, who builds upon the work of Roe \cite{MR1147350} and Kapovich--Kleiner \cite{MR2168506} to give a theory sharing many similarities with ordinary group cohomology.
We give several examples that serve to illustrate the definitions as well as to introduce and standardise specific objects that are referred to throughout the remainder of the article. Apart from some examples and basic observations, all material is lifted directly from the first source.

\begin{definition}    Given a metric space $X$, an \textit{$R$-module over $X$} consists of the tuple $\textbf{M}=(M,B,p, \delta,\{B(i)\}_{i\in\N})$ where

    \begin{enumerate}
        \item $M$ is an $R$-module, called the \textit{underlying $R$-module}.
        \item $B$ is an indexing set (often a basis for $M$ as an $R$-module).
        \item $\delta\colon B\to \Hom_R(M,R)$ is a function, where we denote by $\delta_b$ the image of $b$ under $\delta$. We call $\delta_b$ a \textit{coordinate} (this often behaves like a dual basis to $B$). 
        \item $\delta_b(m)=0$ for all $b\in B$ if and only if $m=0$.
        \item $p\colon B\to X$ is a function called the \textit{control function}.
        \item $\{B(i)\}_{i\in\N}$ is a filtration of $B$ such that for all $m\in M$ there is $i\in\N$ such that $\delta_b(m)=0$ for all $b\in B\setminus B(i)$. 
    \end{enumerate}

    For every $m\in M$ we define the support \[\supp_{\textbf{M}}(m)\coloneq \{p(b): b\in B,\delta_b(m)\ne 0\}.\]We also define a filtration of $\textbf{M}$ corresponding to that of the indexing set $B$ \[\textbf{M}(i)\coloneq \{m\in M:  \delta_b(m)=0 \text{ for all } b\in B\setminus B(i)\},\] which are $R$-modules over $X$ also. The support function and filtration of $\M$ are the most important features of a module over a metric space, as in practice this is almost always the data one is required to work with.
\end{definition}

\begin{remark}
    A module $\M$ over a metric space $X$ does not itself interact with the geometry of $X$, but the supports of elements of $\M$ are subsets of $X$. In particular our morphisms of modules, called finite displacement maps, places restrictions on the distortion of the supports of module elements as subsets of the metric space under such a map of $R$-modules. 
\end{remark}

We will repeatedly invoke the following lemma, whose proof is elementary.
\begin{lemma}\cite[Lemma~3.4]{MR123}\label{lem:Margolis-Lemma3_4}
Let $\M$ be an $R$-module over $X$, and let $m_1, m_2 \in \M$ with $r_1, r_2 \in R$. Then
\[
\operatorname{supp}_\M(r_1 m_1 + r_2 m_2) \subseteq \operatorname{supp}_\M(m_1) \cup \operatorname{supp}_\M(m_2).
\]
\end{lemma}

\begin{definition}
    An $R$-module $\M$ over a metric space $X$ is of \textit{finite height} if there exists $i_0\in\N$ such that $\M=\M(i_0)$. The module $\M$ is said to be \textit{proper} if for every bounded subset $C\subseteq X$, the set $p^{-1}(C)\cap B(i)$ is finite for every $i\in\N$.
\end{definition}

\begin{remark}\label{finite height remark}
    finite-height $R$-modules over a metric space $X$ are designed to exactly be the analogue of finitely generated $R[G]$-modules, see Proposition \ref{uniform acyclicity and resolutions}.
\end{remark}

The following will be our notion of morphism of $R$-modules over a metric space, and as such should be thought of as a generalisation of $R[G]$-linear maps between $R[G]$-modules, see Example \ref{Maps between R[G] modules}.

\begin{definition}
    Let $\textbf{M}$ and $\textbf{N}$ be $R$-modules over a metric space $X$. We say that an $R$-module homomorphism $f\colon M\to N$ has \textit{finite displacement}, denoted $f\colon \textbf{M}\to\textbf{N}$, if there exists an increasing function $\Phi\colon \N\to\N$ such that for all $i\in \N$ the following hold:

    \begin{enumerate}
        \myitem{(FD1)}{finite_displacement_1} $f(\M(i))\subset \textbf{N}(\Phi(i))$,
        \myitem{(FD2)}{finite_displacement_2} for all $m\in \M(i)$, we have $\supp_{\textbf{N}}(f(m))\subseteq N_{\Phi(i)}(\supp_{\textbf{M}}(m))$.
    \end{enumerate}

    In this case we say that $f\colon \M\to \textbf{N}$ has displacement $\Phi$. A finite displacement map $f\colon\M\to\textbf{N}$ is called a \textit{geometric isomorphism} if there is a finite displacement map $g\colon\textbf{N}\to\M$ such that $g\circ f=\id$ and $f\circ g=\id$. A geometric isomorphism $f$ is moreover called a \textit{canonical isomorphism} if it is induced by the identity map on the underlying $R$-modules.
\end{definition}

 As one would expect, a composition of finite displacement maps is also of finite displacement, see \cite[Lemma $3.11$]{MR123}. It follows that we have a category of $R$-modules over metric spaces, whose morphisms are given by finite displacement maps.

\begin{remark}\label{staggering}
     By \cite[Remark $3.22$ and Proposition $3.23$]{MR123} we may always assume, up to canonical isomorphisms of domains and codomains, that a countable family of finite displacement maps of $R$-modules $\{f_i\colon\M_i\to \textbf{N}_i\}_{i\in \N}$ over a metric space $X$ have a uniform displacement function $\Phi:\N\to\N$.
\end{remark}

We illustrate these definitions with some examples.

\begin{eg}
    When $M=R$ and $X$ is any metric space, we define the \textit{constant} $R$-module structure of $R$ over $X$ by $M=R$, $B=\{1\}$, $\delta_1(r)=r$, $B(i)=B=\{1\}$ and $p\colon B\to X$ by $1\mapsto x_0$ for some fixed $x_0\in X$; in particular $\supp_R(r)=x_0$ for every $r\in R$. This should be thought of as the analogue of the trivial $R[G]$-module~$R$.
\end{eg}

\begin{eg}\label{free R[G] module}
    Let $G$ be a finitely generated group. We will equip the free $R[G]$-module $R[G]$ with the structure of a module over the metric space $G$. Let $B=G$, which is the $R$-module basis for $R[G]$, and $\{\delta_b\}_{b\in B}$ be the dual $R$-basis to $B$. It is then natural to set $p=\id_G$ and $B(i)=B$ for every $i\in\N$. In particular we have $$\supp_{R[G]}\left(\sum_{g\in G}\lambda_gg\right)=\{g\in G:\lambda_g\ne 0\}.$$ 
    \end{eg}

 A similar structure can be put on any countable direct sum $\oplus_IR[G]$, see Definition~\ref{direct sum def}. From now on, whenever~$\oplus_I R[G]$ is viewed as an $R$-modules over $G$, it will be equipped with this structure.

\begin{eg}\label{free module over X}
 Let $X$ be a metric space and for $n\in \N\cup\{0\}$ define the set of formal tuples \[B_n\coloneq \{[x_0,\dots,x_n]\colon x_j\in X \text{ for all } j\} \] with a filtration 
 \[B_n(i)\coloneq \{[x_0,\dots,x_n]\colon d(x_j,x_k)\leq i \text { for all } j,k\}.\]
 We equip the free $R$-module on the set $B_n$ with the structure of an $R$-module over~$X$, which we will denote by  $\textbf{C}_n(X)$. The maps $\{\delta_b\}_{b\in B}$ will be the dual $R$-basis to $B$. We choose the control function $p\colon B\to X$ to be given by $[x_0,\dots,x_n]\mapsto x_0 $, and the filtration by $\{B_n(i)\}_{i\in\N}$. Note then that $\supp_{\textbf{C}_n(X)}([x_0,\dots ,x_n])=x_0$, and the submodules $\textbf{C}_n(X)(i)$ are given by the free $R$-module on $B_n(i)$ for each $i\in\N$.
 
 The reader might be reminded of the definition of the Rips complex of $X$, and in fact, if one equips $X$ with an arbitrary total order and considers the subset of~$B_n$ consisting of ordered tuples, the free $R$-module on this set can be identified with the $n$th chain module in the simplicial $R$-chain complex of the infinite Rips complex $P_\infty(X)$. If one considers the filtration of each chain module coming from the Rips filtration $\{P_i(X)\}_{i\in\N}$, it is not hard to see that there is a structure of an $R$-module over $X$ arising on these modules in a natural way, see \cite[Lemma $6.25$]{MR123}.
\end{eg}

\begin{eg}\label{Maps between R[G] modules}
    Given a finitely generated group $G$, consider a left $R[G]$-module homomorphism $f \colon  R[G]\to  R[G]$ given by right multiplication by some $\alpha\in R[G]$. We show that this is a finite displacement map when $R[G]$ is equipped with the structure of an $R$-module over~$G$ from Example~\ref{free R[G] module}. Condition~$(\mathrm{FD}1)$ is clear as the filtration of $R[G]$ is trivial. To see $(\mathrm{FD}2)$, note that for every $g\in G$ we have
    \[\supp_{R[G]}(f(g))=\supp_{R[G]}(g\alpha)=g.\supp_{R[G]}(\alpha)\subset B_K(g)\] for $K>0$ such that $\supp(\alpha)\subseteq B_{K}(1)\subseteq G$. A similar argument shows that any $R[G]$-linear map between countably generated free $R[G]$-modules has finite displacement.
\end{eg}

We now introduce a construction due to Margolis which will be employed crucially throughout this paper:  the pullback of modules. We adopt here a different notation, as often we wish to emphasise the spaces involved when the map is implicit.

\begin{definition}\label{pullback def}
    Let $\imath\colon X\to Y$ be a coarse embedding, and $\M$ be an $R$-module over $Y$. We define the \textit{pullback of $\M$ along $\imath$} to be the $R$-module over $X$ given by \[\M|_X=(M,B, \delta,\pi\circ p, \{B|_X(i)\})\] where $\pi\colon Y\to X$ is a choice of closest-point projection and \[B|_X(i):=\{b\in B:d(p(b),\imath(X))\leq i.\] By \cite[Lemma $3.18$]{MR123} this definition does not depend on the choice of $\pi$, 
    up to canonical isomorphism.
\end{definition}

\begin{eg}\label{pullback example}

    Let $H\leq G$ be a subgroup of a finitely generated group. Recall that for any right transversal $T$ of $H$ in $G$, the restriction of scalars $R[G]|_{R[H]}$ is isomorphic as an $R[H]$-module to $\bigoplus_{t\in T}R[H]t$. It is not hard to see that there is a choice of right transversal $T$ of $H$ in $G$ and a closest-point projection $\pi\colon G\to H$ such that for every $g\in G$, $g=\pi(g)t$ for some $t\in T$.
    
    But now it follows that if Definition \ref{pullback def} is applied with this choice of $\pi$, with $H$ equipped with the subspace metric, there is a natural identification of $R$-modules $(R[G]|_H)(i)\cong \oplus_{t\in T_i}R[H]t$ for some finite subset $T_i\subseteq T$. This filtration of $R[G]|_H$ then corresponds to a filtration by $R[H]$-submodules of the restriction of scalars $R[G]|_{R[H]}\cong\oplus_{t\in T}R[H]t$, a situation supporting the perspective of Remark \ref{finite height remark}.
\end{eg}

 We now consider the basic construction of direct sums in the category of modules over metric spaces, the following is \cite[Proposition $3.25$]{MR123} stated for a countable indexing set.

\begin{proposition}\label{direct sum def}
   Let $J$ be a countable indexing set, and for every $j\in J$ let $\M_j=(M_j,B_j,\delta_j,p_j,\{B_j(i)\}_{i\in\N})$ be an $R$-module over a metric space $X$. 
    We define the \textit{direct sum} of $\{\M_j\}_{ J}$ by \[\oplus_{ J}\M_j:=(\oplus_{ J}M_j, \sqcup_{ J}B_j,\sqcup_{ J}\delta_j, \sqcup_{ J}p_j,\{\sqcup_{ J}B_j(i)\}_{i\in\N}) \]where $\delta_j:B\to \Hom_R(M_j,R)\to \Hom_R(\oplus_{ J}M_j, R)$. This equips $\oplus_J \M_j$ with the structure of an $R$-module over $X$ satisfying the following:
    \begin{enumerate}
        \item $(\oplus_{ J}\M_j)(i)=\oplus_{ J}(\M_j(i))$,
        \item $\supp_{\oplus_{ J}\M_j}(\sum_{ J}m_j)=\cup_{ J}\supp_{\M_j}(m_j)$ for all $m_j\in M_j$.
\end{enumerate}
    
\end{proposition}

The operation of taking the direct sum interacts well with pullbacks.
\begin{lemma}\label{direct sum and pullback}
     If $\imath\colon X\to Y$ is a coarse embedding, then there is a canonical isomorphism of $R$-modules over $X$, $(\oplus_{ J}\M_j)|_X\cong \oplus_{ J}(\M_j|_X)$.
    
\end{lemma}

\begin{proof}
   First we observe that the pullback module $(\oplus_J \M_j)\vert_X$ is given by
\[
(\oplus_J M_j,\ \sqcup_J B_j,\ \sqcup_J \delta,\ \pi\circ(\sqcup_J p_j),\ \{\sqcup_J B_j\vert_X(i)\}_{i\in\mathbb{N}}),
\]
where $\pi\colon Y\to X$ is a closest-point projection and
\[
\sqcup_J B_j\vert_X(i)=\{b\in \sqcup_J B_j : d(p(b),\imath(X))\le i\}.
\]

We claim that this module is canonically isomorphic to $\oplus_J(\M_j\vert_X)$, and that the elements of the defining tuple agree set-theoretically. Indeed, it is easily verified that
\[
\pi\circ(\sqcup_J p_j)=\sqcup_J(\pi\circ p_j),
\]
and
\[
\{b\in \sqcup_J B_j : d((\sqcup_J p_j)(b),\imath(X))\le i\}
= \sqcup_J \{b\in B_j : d(p_j(b),\imath(X))\le i\}.\qedhere
\]
\end{proof}

We next introduce free and projective $R$-modules over metric spaces. Recall that for an $R$-module $M$, a \textit{projective basis} for $M$ is a set of elements $\{m_b\}_{b\in B}\subset M$ and maps $\delta_b:B\to \Hom_R(M,R)$ such that for any $m\in M$, $m=\sum_{b\in B}\delta_b(m)m_b$. An $R$-module is \emph{projective} if and only if it has a projective basis, see \cite[Proposition $\mathrm{I}.8.2$]{Brown}, and for a free $R$-module this can be chosen to be a free basis.

\begin{definition}\label{def free module over X}
    Let $X$ be a metric space and $\M=(M,B,\delta,p,\{B(i)\})$ be an $R$-module over $X$. We say that $\M$ is a \emph{projective} (resp. \emph{free}) $R$-module over $X$ if the underlying $R$-module $M$ is a projective (resp. free) $R$-module with projective (resp. free) basis $\{m_b\}_{b\in B}$, $\{\delta_b\}_{b\in B}$ indexed by $B$ such that there is an increasing function $\Phi\colon \N\to \N$ with

    \begin{enumerate}
        \item $m_b\in M(\Phi(i))$ for all $b\in B(i)$,
        \item $\supp_{\M}(m_b)\subset N_{\Phi(i)}(p(b))$ for all $b\in B(i)$.
    \end{enumerate}
\end{definition}

    It is not hard to verify that Examples \ref{free R[G] module} and \ref{free module over X} yield free $R$-modules over their respective metric spaces.

\begin{definition}
    A non-negative chain complex of $R$-modules $\textbf{C}_\bullet$ is called an \textit{$R$-chain complex over $X$} if 

    \begin{enumerate}
        \item each $\textbf{C}_i$ is an $R$-module over $X$,
        \item each boundary map $\partial:\textbf{C}_i\to \textbf{C}_{i-1}$ is of finite displacement.
    \end{enumerate}

    An $R$-chain complex is \textit{augmented} if there is a map of $R$-modules $\epsilon:\textbf{C}_0\to R$ such that the composition $\textbf{C}_1\to \textbf{C}_0\to R$ is the trivial map. An $R$-chain complex over $X$ is \textit{proper} (resp. \textit{finite height}) if all chain modules are so. An $R$-linear chain map is a \textit{finite displacement chain map} if all maps on $R$-modules over $X$ are finite displacement; we define \textit{finite displacement chain homotopies (resp. homotopy equivalences, nullhomotopies)} likewise.
\end{definition}

The following is a technical condition we impose on the boundary maps of a chain complex as part of the definition of a projective $R$-resolution over a metric space. Whilst finite displacement maps are designed to model $R[G]$-linearity, uniform preimages is what will distinguish between coarse acyclicity and coarse \textit{uniform} acyclicity.

\begin{definition}\label{uniform preimages}
    An $R$-module homomorphism $f:\M\to \textbf{N}$ between $R$-modules over $X$ has \textit{uniform preimages} if there is an increasing function $\Omega:\N\times\N\to\N$ such that for all $i\in\N$ and all $n\in \textbf{N}(i)\cap \im(f)$, there is some $m\in f^{-1}(n)$ such that 

    \begin{enumerate}
        \item $m\in\M(\Omega(i,D))$,
        \item $\supp_{\M}(m)\subset N_{\Omega(i,D)}(\supp_{\textbf{N}}(n))$,

    \end{enumerate} where $D=\lceil \diam(\supp_{\textbf{N}}(n))\rceil$.
\end{definition}
\begin{definition}\label{projective res over $X$}
   Let $X$ be a metric space. A \emph{projective} (resp. \emph{free}) \emph{$R$-resolution over $X$} is an augmented $R$-chain complex $\textbf{C}_\bullet$ over $X$ with augmentation $\epsilon$ such that

    \begin{enumerate}
        \item each $\textbf{C}_\bullet $ is a projective (resp. free) $R$-module over $X$,
        \item the augmented complex $\textbf{C}_\bullet\xrightarrow{\epsilon}R\to 0$ is exact,

    \item all boundary maps are finite displacement and have uniform preimages,

    \item there is $D\geq 0$ such that for every $x\in X$, there is $m_x\in \textbf{C}_0(D)$ such that $\epsilon(m_x)=1$ and $\supp_{\textbf{C}_0}(m_x)\subset N_D(x)$.
    \end{enumerate}

    We will denote this by $\textbf{C}_\bullet$ when the augmentation is implicit.
\end{definition}

   Most of our use of projective resolutions over metric spaces is restricted to the following key examples.

   \begin{eg}\label{examples of free resolutions}
       For any metric space $X$, there is a free $R$-resolution over $X$ called the \textit{standard resolution}. Recall from Example \ref{free module over X} the free $R$-modules over $X$ \[\textbf{C}_n(X)=R\{[x_0,\dots,x_n]: x_j\in X\text{ for all } j\leq n\}\] for any $n\in\N$. There is then a free $R$-resolution over $X$ given by \[\dots\to\textbf{C}_n(X)\to \textbf{C}_{n-1}(X)\to\dots\to \textbf{C}_2(X)\to\textbf{C}_1(X)\to R[X]\xrightarrow{\epsilon} R\to 0,\] where the boundary maps and augmentation are the obvious ones, see \cite[Definition $6.18$]{MR123}. Likewise, the simplicial $R$-chain complex of the Rips complex $P_\infty(X)$ naturally gives a free $R$-resolution over $X$, see \cite[Lemma $6.25$]{MR123}.
   \end{eg}

\begin{eg}\label{Pullback resolution example} 

Given a coarse embedding $\imath\colon X\to Y$ and $\textbf{C}_\bullet$ a projective $R$-resolution over $Y$, the pullback augmented chain complex $\textbf{C}_\bullet|_X$ is a free $R$-resolution over $X$, see \cite[Proposition $6.16$]{MR123}. This is the analogue of how, given a group cohomological projective resolution of $R$ over $R[G]$ and a subgroup $H\leqslant G$, one can restrict scalars to obtain a projective resolution over $R[H]$.

\end{eg}

\begin{definition}
    Let $\M$, $\textbf{N}$ be modules over a metric spaces $X$, with underlying $R$-modules $M$, $N$. We define $\Hom_{\fd}(\M,\textbf{N})$ to be the $R$-submodule of $\Hom_R(M,N)$ consisting of finite displacement maps. 
\end{definition}

Having defined projective resolutions and $\Hom$-sets in the category of $R$-modules over metric spaces, we are ready to introduce coarse cohomology groups.

\begin{definition}
    Let $X$ be a metric space and $\mathbf{C}_\bullet$ be a free $R$-resolution over $X$, then we define the \emph{coarse cohomology of $X$ with coefficients in $\M$} by \[\Homology^\ast_{\coarse}(X;\M)\coloneq \Homology^\ast(\Hom_{\fd}(\mathbf{C}_\bullet,\M)).\]
\end{definition}

    One of the key features of coarse cohomology, going back to the work of \cite{MR1147350}, is that it recovers the cohomology at infinity \[\Homology^\ast_{\coarse}(G;R)\cong \Homology^\ast(G;R[G])\] when $G$ is a finitely generated group,  see \cite[Corollary $8.9$]{MR123}. We can now also define the natural notion of cohomological dimension in this setting, which turns out to be an improved version of the ordinary virtual cohomological  dimension for groups, for instance being finite for all CAT$(0)$ and hyperbolic groups \cite[Theorem $E$]{MR123}.

\begin{definition}
    Given a metric space $X$ and a ring $R$, we define the \textit{coarse cohomological dimension of $X$ over $R$}, denoted $\ccd_R(X)$, to be the minimal $n$ such that $\Homology^k_{coarse}(X;\M)=0$ for all $k>n$ and all $R$-modules $\M$ over $X$.
\end{definition}

\begin{proposition}{\cite[Proposition $8.2$]{MR123}}
    The coarse cohomological dimension of a metric space $X$ over $R$ is the minimal length of a projective $R$-resolution over~$X$.
\end{proposition}

We now introduce uniform versions of the coarse homological finiteness properties given in Definition \ref{Coarse acyclicity}, which we then characterise in terms of projective resolutions over metric spaces.

\begin{definition}\label{coarse uniform acylicity}
    A metric space $X$ is \textit{coarsely uniformly $(n-1)$-acyclic} if for all $ r\in\R_{\geq  0}$, there exists $ s\geq r$ such that for all $a\in \R_{\geq 0}$ there exists $ b\geq a$  such that for every $x\in X$, the map on reduced simplicial homology \[\tilde{H}_i(P_r(N_a(x)),R)\to \tilde{H}_i(P_s(N_b(x)),R)\] is trivial for all $i\leq n$.
\end{definition}

\begin{proposition}{\cite[Proposition $9.3$]{MR123}}\label{uniform acyclicity and resolutions}
    A metric space $X$ is \textit{coarsely uniformly $(n-1)$-acyclic over $R$} if and only if there is a projective $R$-resolution $\textbf{C}_\bullet$ over $X$ such that $\textbf{C}_i$ is of finite height for all $i\leq n$.  
\end{proposition}

In this article we will mainly concern ourselves with the study of discrete coarsely homogeneous bounded-geometry metric spaces. The following result allows us to use the less involved Definition \ref{Coarse acyclicity} instead of Definition \ref{coarse uniform acylicity}.

\begin{lemma}\label{Uniform acyclicity vs acyclicity}
    A discrete coarsely homogeneous bounded-geometry metric space $X$ is coarsely uniformly $(n-1)$-acyclic if and only if it is coarsely $(n-1)$-acyclic.
\end{lemma}

\begin{proof}
    This is a standard proof, so we give only a sketch. As $X$ is coarsely homogeneous, any two neighbourhoods $N_a(x)$ and $N_a(y)$ are coarsely equivalent, and the control functions do not depend on $x,y\in X$ or $a\in\R_{\geq 0}$. Combining this with the analogue of \cite[Corollary $5$]{MR1293049} for a coarse equivalence, we reduce to the case of a single point $x\in X$. 
    
    Now, $P_r(N_a(x))$ is compact as~$X$ is discrete and of bounded geometry, hence it has finitely generated reduced homology in all degrees. In particular, by applying coarse acyclicity of $X$, every cycle of the requisite dimension in $P_s(N_a(x))$ can be filled in $P_s(N_b(x))$ for some $s\geq r$ and $b\geq a$, from which we conclude.
\end{proof}

Finally, we give the following 
definition of a coarse Poincar\'{e}-duality space over a commutative ring $R$ due to Margolis \cite[Definition $12.1$]{MR123}, inspired by that of Kapovich--Kleiner \cite{MR2168506}.

\begin{definition}
    A proper metric space $X$ is a \textit{coarse Poincar\'{e}-duality space over~$R$} (coarse $\PD^n(R)$ space) if there exists a proper finite-height projective $R$-resolution $\textbf{C}_\bullet$ over $X$, satisfying the following.  There exist finite displacement chain maps $f_\bullet \colon \textbf{C}_\bullet\to \textbf{C}_c^{n-\bullet}$, $\bar{f}_\bullet\colon \textbf{C}_c^{n-\bullet}\to \textbf{C}_{\bullet}$ such that $f_\bullet\circ \bar{f}_\bullet\simeq\id$, $\bar{f}_\bullet\circ f_\bullet\simeq \id$ are finite displacement chain nullhomotopic, where $\textbf{C}_c^{\bullet}:=\Hom_{\fd}(\textbf{C}_\bullet,R)$ (this is a chain complex of modules over metric spaces by \cite[Proposition $5.20$]{MR123}). 
\end{definition}

\begin{remark}\label{Remark about coarse Poincare duality}
It is worth noting that Margolis's definition is slightly more general than that of \cite{MR2168506} and the readily equivalent \cite{HyperbolicPD^n}. The only place in which we refer to the latter definition will be to invoke \cite[Corollary $3$]{HyperbolicPD^n} in the proof of Theorem~\ref{intro-characterising PD^2}. 
\end{remark}

Apart from the proof of Proposition \ref{amenable or linear iso inequality}, we will not need to work with this definition explicitly due to the following coarse cohomological characterisation of coarse $\PD^n(R)$ spaces, which is a consequence of \cite[Theorem B]{MR123}. This is the coarse analogue of Definition \ref{def: PD group}, which deals with Poincar\'e-duality groups.

\begin{theorem}\label{charactersing PD^n}
    Let $R$ be a PID and $X$ be a coarsely homogeneous proper metric space that is coarsely uniformly $(n-1)$-acyclic over $R$ for some $n\geq 1$. If \[\Homology_{\coarse}^i(X;R)\cong \begin{cases}
    R&\text{  if } i=n\\
    0  &\text{  if } i\leq n-1.
\end{cases}
\] then $X$ is a coarse $\PD^n(R)$ space.
\end{theorem}

We conclude with noting the following basic lemma in coarse homological algebra.

\begin{lemma}\label{hom sets and products}
    For any metric space $X$ and a countable family of $R$-modules $\{\textbf{P}_i\}_{i\in I}$ over $X$, there is an isomorphism of $R$-modules \[\Hom_{\fd}(\bigoplus_{i\in I} \textbf{P}_i,\textbf{M})\cong \prod_{i\in I}\Hom_{\fd}(\textbf{P}_i,\textbf{M})\] for any $R$-module $\textbf{M}$ over $X$.
\end{lemma}

\begin{proof}
    We define a map 
    \begin{align*}
        F\colon \Hom_{\fd}(\bigoplus_{i\in I} \textbf{P}_i,\textbf{M})&\to \prod_{i\in I}\Hom_{\fd}(\textbf{P}_i,\textbf{M})\\
        f'&\mapsto (f'|_{\textbf{P}_i})_{i\in I}. 
    \end{align*}
    It is clear that the map $F$ is a homomorphism of $R$-modules, where well-definedness follows from the fact that each $f'|_{\textbf{P}_i}\colon \textbf{P}_i\to \textbf{M}$ is a finite displacement map of $R$-modules over $X$.
    We obtain an inverse to~$F$ by
    \begin{equation*}
        (f_i)_{i\in I}\mapsto f:=\sum_{i\in I}\tilde{f}_i,    
    \end{equation*}
    %$(f_i)_{i\in I}\in \prod_{i\in I}\Hom_{\fd}(\textbf{P}_i,\M)$ to the sum $f=\sum_{i\in I}f_i$,
    where we extend each $f_i:\textbf{P}_i\to\M $ trivially to a map $\tilde{f}_i:\bigoplus_{i\in I}\textbf{P}_i\to \M$. 
    Every  $x\in\bigoplus_{i\in I} \textbf{P}_i$ has only finitely many nonzero summands, so that $f(x)$ is a finite sum. It follows that the map~$f$ is well-defined. 

    We have to prove that~$f$ has finite displacement.
    By Remark \ref{staggering}, we may assume that there is a uniform choice of displacement function $\Psi\colon \N\to\N$ for all $f_i\colon \textbf{P}_i\to\textbf{M}$. 
    Recall that the direct sum becomes a module over a metric  space by choosing $ B(j)=\bigsqcup_{i\in I} B_i(j)$,  
    so it is clear that for each $m\in (\bigoplus_{i\in I} \textbf{P}_i) (j)=\bigoplus_{i\in I} \textbf{P}_i(j)$, its image $f(m)=\sum_{i\in I}\tilde{f}_i(m_i)$ is contained in $\textbf{M}(\Psi(j))$, because this holds for every summand as each $f_i$ has displacement~$\Psi$, verifying~\ref{finite_displacement_1}.

    It remains to show that for $m\in (\bigoplus_{i\in I} \textbf{P}_i) (j) $ and $y\in \supp_{\textbf{M}}(f(m))$, we have that $y\in N_{\Psi(j)}(\supp_{\bigoplus_{i\in I} \textbf{P}_i}(m)$. But Lemma~\ref{lem:Margolis-Lemma3_4} implies that $y\in \supp_{\textbf{M}}(f(m_i))$ for some $m_i\in \textbf{P}_i$, and $f(m_i)=f_i(m_i)$ by definition. As $f_i$ has displacement $\Psi$ over $g$, we have\[\supp_{\textbf{M}}(f_i(m_i))\subset N_{\Psi(j)}(\supp_{\textbf{P}_i}(m_i)),\] and $y\in \supp_{\M}(f(m_i))=\supp_{\M}(f_i(m_i))$, which confirms~\ref{finite_displacement_2}.
\end{proof}

\section{Coarse coinduced modules and Shapiro's Lemma}\label{section:shapiros_lemma}

In this section we define a coarse coinduced $R[G]$-module associated to the data of a coarse embedding of a metric space $X$ into a group $G$ along with an $R$-module over $X$. This module should be thought of as an analagoue of the coinduced module $\Hom_{R[H]}(R[G],M)$ attached to a subgroup inclusion $H\leqslant G$ and left $R[H]$-module~$M$. We use this to state and prove an analogue of Shapiro's lemma in this context, which will relate the coarse cohomology of the coarsely embedded space to the group cohomology of the ambient group; we are particularly interested in the case of the embedded space being the quasikernel of a quasimorphism. 

Recall from Definition \ref{pullback def} that for a coarse embedding $\imath\colon X\to G$ of a  metric space~$X$ into a finitely generated group~$G$, one can construct the pullback module $R[G]|_X$, which is an $R$-module over~$X$. The support of an element $m\in R[G]\vert_X$ is described by the image of its support in~$R[G]$ under a closest-point projection onto~$X$, and the filtration of $B=G$ is given by subsets $B(i)\coloneq N_i(\imath(X))$.

\begin{definition}\label{coinduced action}
    Given a coarse embedding $\imath\colon X\to G$ of a metric space into a finitely generated group, and an $R$-module $\textbf{M}$ over $X$, we define the \textit{coarse coinduced module} to be the right $R$-module \[\Hom_{\fd}( R[G]|_X, \textbf{M}).\] The $R$-action is given by $(f.r)(x):=f(xr)=rf(x)$ for any $r\in R$ and $x\in R[G]|_X$, where $f\colon R[G]|_X\to\M$ is a finite displacement map.
\end{definition}
We now show that the coarse coinduced module is a right $R[G]$-module as follows. 
\begin{proposition}\label{action on coarse coinduced module}
    If $\imath\colon X\to G$ is a coarse embedding of a metric space into a finitely generated group and $\M$ is an $R$-module over $X$, then
    $\Hom_{\fd}(R[G]|_X,\textbf{M})$ is a right $R[G]$-module with action given by 
    \begin{equation*}
        (f.h) (x):=f(xh^{-1}), 
    \end{equation*}
    where $f\in\Hom_{\fd}(R[G]|_X,\textbf{M})$,  $h\in G$ and $x \in R[G]|_X$.
   
\end{proposition}

\begin{proof}
    It suffices to show that $f.h:R[G]|_X\to\M$ is a finite displacement map for every $f\in \Hom_{\fd}(R[G]\vert_X, \M)$ and $h\in G$.
    For $g\in G$ we observe that
    \[g\in (R[G]|_X)(i)\quad\text{implies} \quad g\in N_i(\imath(X)),\] so in particular we can bound  
    \begin{equation}\label{eq:distance_upper_bound_ghinverse_to_X}
        d(gh^{-1}, \imath(X))\leq d(g,\imath(X))+d(g,gh^{-1}) \leq i+|h|. 
    \end{equation}
    As $f$ is a finite displacement map with displacement function~$\Phi$, we see that
    \begin{equation*}
        (f.h)(g)=f(gh^{-1})\in \M(\Phi(i+|h|)).
    \end{equation*}
    
     This is independent of the choice of~$g\in (R[G]|_X)(i)$, from which we conclude that \ref{finite_displacement_1} holds as we may extend to $R$-linear combinations of group elements $g\in G$ to see that $f((R[G]|_X)(i))\subseteq \M(\Phi(i))$. 
    
    To verify condition~\ref{finite_displacement_2}, note that for every $g\in N_i(\imath(X))$ we have by~\eqref{eq:distance_upper_bound_ghinverse_to_X}  
    and the triangle inequality that  \[d(\imath(\pi(g)),\imath(\pi(gh^{-1})))\leq 
    2i+2|h|,\] where $\pi\colon G\to X$ is a closest-point projection. As $\imath\colon X\to G$ is a coarse embedding, there exists a proper function $\rho_+$ such that 
    \[d(\pi(g), \pi(gh^{-1}))\leq \rho_+(2i+2|h|).\] 
But now this implies that there is an inclusion 
    \[N_{\Phi(i)}(\pi(gh^{-1}))\subseteq N_{\rho_+(2i+2|h|)+\Phi(i)}(\pi(g)).\] 
    By the finite displacement condition on~$f$, for every $g\in (R[G]|_X)(i)$ we have that 
    \[\supp_{\M}(f(gh^{-1}))\subseteq N_{\Phi(i)}(\pi(gh^{-1})),\] and by Lemma~\ref{lem:Margolis-Lemma3_4} the same containment holds for $R$-linear combinations of group elements $g\in G$. We conclude that  \[\supp_{\M}((f.h)(g))=\supp_{\M}(f(gh^{-1}))\subseteq N_{\rho_+(2i+2|h|)+\Phi(i)}(\pi(g)),\] so that ~\ref{finite_displacement_2} follows.
\end{proof}
\begin{remark}\label{actions on coinduced module}
    
  Note that we may view $\Hom_{\fd}(R[G]|_X,\textbf{M})$ as a left $R[G]$-module in the standard way by setting $h\ast f(x)=f(xh)$.
\end{remark}

It will be useful to establish an identification of the coarse coinduced module when $\M=R$ with an $R[G]$-submodule of $R^G$ consisting of functions whose support grows in a controlled way in increasingly large neighbourhoods of~$\imath(X)\subseteq G$.
\begin{lemma}\label{Coarse coinduced module when $M=R$}
    If $\imath\colon X\to G$ is a coarse embedding, then we have the following isomorphism of right $R[G]$-modules
    \[\Hom_{\fd}(R[G]|_X,R)\cong\{f\colon G\to R \mid  \supp(f)\cap N_i(\imath(X))\text{ is finite for all }i\in\R \},\] where the right hand side is an $R[G]$-submodule of $R^G$.

\end{lemma}

\begin{proof}
    Observe that $R[G]|_X$ is proper \cite[Lemma $3.27$, $(3)$]{MR123}, and that one can identify $R$-linear maps $f\colon R[G]\to R$ with functions $f\colon G\to R$.
    Then by \cite[Lemmata~$5.6$ and~$5.9$]{MR123}
     there exists an isomorphism of $R$-modules
    \[\Hom_{\fd}(R[G]|_X,R)\cong\{f\colon G\to R\colon  \pi(\supp(f)\cap N_i(\imath(X))) \text{ is finite for all } i\in\N \},\] 
    where $\pi\colon G\to \imath(X)$ is a closest-point projection.
    We claim that $\supp(f)\cap N_i(\imath(X))$ is finite if $\pi(\supp(f)\cap N_i(\imath(X)))$ is so. As~$G$ is a discrete bounded-geometry metric space, 
    it is sufficient to note that
    \[\supp(f)\cap N_i(\imath(X))\subset N_i(\pi(\supp(f)\cap N_i(\imath(X)))),\] 
    where the latter is finite, from which we conclude the existence of an $R$-module isomorphism between our modules of interest.
    
The fact that the $R$-module isomorphism extends to an isomorphism of $R[G]$-modules follows directly from the natural right action of $G$ on $R^G$ given by $(f.g)(x)=f(xg^{-1})$, which coincides with the action defined in Proposition~\ref{action on coarse coinduced module} upon identifying functions $f:G\to R$ with $R$-linear maps $f:R[G]\to R$.
\end{proof}

We are now ready to prove a version of the Shapiro lemma in coarse cohomology.

\numberedtheorem{Theorem~\ref{Intro coarse Shapiro}}{}{Given a coarse embedding $\imath\colon X\to G$ of a metric space into a finitely generated group, there is an isomorphism of cohomology groups \[\Homology^\ast_{\coarse}(X;\textbf{M})\xrightarrow{\sim} \Homology^\ast(G; \Hom_{\fd}( R[G]|_X, \textbf{M}))\]for every $R$-module $\M$ over the metric space $X$.}

\begin{proof}
Let us equip the simplicial chain complex of $P_\infty(G)$ with the structure of a free $R$-resolution over $G$ outlined in Example \ref{examples of free resolutions}. In particular, all chain modules are countably generated free $R[G]$-modules and all boundary maps are $R[G]$-linear. By \cite[Proposition $6.16$]{MR123},
the pullback resolution $ P_\bullet|_X$ is a projective $R$-resolution over $X$, so it can be used to compute the coarse cohomology of~$X$. Note also that, by Lemma \ref{direct sum and pullback}, we may assume up to canonical isomorphism that the chain modules of $P_{\bullet}|_X$ are of the form $\oplus_{ I}(R[G]|_X)$, where the direct sum is taken in the category of modules over metric spaces.

We now view $\Hom_{\fd}(R[G]|_X,\textbf{M})$ as a left $R[G]$-module as in Remark \ref{actions on coinduced module} in order to make sense of it as a coefficient module for group cohomology over a projective resolution of left $R[G]$-modules.

     We construct an isomorphism between  $\Hom_{R[G]}(\oplus_IR[G], \Hom_{\fd}(R[G]|_X, \textbf{M})$ and $ \Hom_{\fd}(\oplus_I(R[G]|_X),\textbf{M})$ as $R$-modules that is natural in the following sense: for any map of left $R[G]$-modules $\partial\colon \oplus_IR[G]\to \oplus_JR[G]$, we want an induced commutative square of $R$-modules 
\[\begin{tikzcd}[column sep=1em]
{\Hom_{R[G]}(\oplus_J R[G],\Hom_{\fd}(R[G]|_X,\textbf{M}))} \arrow[r]   \arrow[d] & {\Hom_{R[G]}(\oplus_IR[G],\Hom_{\fd}(R[G]|_X,\textbf{M}))} \arrow[d] \\
{\Hom_{\fd}(\oplus_J(R[G]|_X),\textbf{M})} \arrow[r]                        & {\Hom_{\fd}(\oplus_I(R[G]|_X),\textbf{M}),}                       
\end{tikzcd}\]
where the horizontal maps are given by $f\mapsto f\circ\partial$.    Define 
    \begin{align*}
        F\colon  \Hom_{R[G]}(\oplus_IR[G], \Hom_{\fd}(R[G]|_X, \textbf{M}))&\to \Hom_{\fd}(\oplus_I(R[G]|_X),\textbf{M})\\
        f & \mapsto \sum_{i\in I}f(e_i)\circ \pi_i, 
    \end{align*}
    
    where $\{e_i\}_{i\in I}$ is a basis of $\oplus_IR[G]$ and $\pi_i\colon \oplus_{i\in I}R[G]\to R[G]$ denotes the projection onto the $i$th coordinate. 
 The map $F$ is well defined and an isomorphism of $R$-modules by Lemma~\ref{hom sets and products}. 
  
  It now only remains to prove that the square commutes, so in particular that \[(\sum_{j\in J} f(e_j)\circ \pi_j)(\partial (ge_i))=(\sum_{i\in I}(f(\partial(e_i))\circ \pi_i)(ge_i) \]
 for every $f\in \Hom_{R[G]}(\oplus_JR[G], \Hom_{\fd}(R[G]|_X, \textbf{M}))$,  $i\in I$ and $g\in G$. Let therefore $i\in I$ and $g\in G$.
 We choose a basis $\{e_j\}_{j\in J}$ of $\oplus_J R[G]$ and write $\partial(e_i)=\sum_{j\in J}\alpha_je_j$ for $\alpha_j\in R[G]$. Note first that there is an equality \[(\sum_{i\in I}(f(\partial(e_i))\circ \pi_i)(ge_i)=f(\partial(e_i))(g).\] But then by definition of the left $R[G]$-action on $\Hom_{\fd}( R[G]|_X, \textbf{M})$, we may write this as
 \[f(\partial(e_i))(g)=\sum_{j\in J}(\alpha_jf(e_j))(g)=\sum_{j\in J}f(e_j)(g\alpha_j).\]
To conclude, it suffices to note that
 \[(\sum_{j\in J}f(e_j)\circ \pi_j)(\partial(ge_i))=\sum_{j\in J}f(e_j)(g\alpha_j),\] as desired. 
\end{proof}

\section{Novikov cohomology and coarse cohomology}\label{section:novikov_cohomology}

In this section we study the coarse coinduced module $\Hom_\fd(R[G]\vert_X, R)$  
when $X$ is the quasikernel of a quasimorphism. A strong link between this $R[G]$-module and the Novikov ring associated to the quasimorphism emerges, which we use to give cohomological applications. 

\begin{definition}
    Given a ring $R$, a group $G$ and a homogeneous quasimorphism $\phi\colon G\to\R$, we define the \textit{Novikov ring}\[\widehat{R[G]}^\phi\coloneq \{x=\sum_{g\in G}\lambda_g g: \supp(x)\cap \phi^{-1}((-\infty, \alpha])\text{ is finite for all } \alpha\in\R\}.\]
\end{definition}

\begin{proposition}\label{Novikov vs coinduced}

Let $G$ be a finitely generated group and $\phi\colon G\to\R$ be a nontrivial homogeneous quasimorphism. We have the following isomorphism of right $R[G]$-modules,
 \[\Hom_{\fd}(R[G]|_{\mathcal{Q}ker(\phi)},R)\cong\{f\colon G\to R: \supp(f)\cap\phi^{-1}([-\alpha,\alpha])\text{ is finite }\forall\alpha\in\R \},\] where the right hand side is an $R[G]$-submodule of $R^G$. 
\end{proposition}

\begin{proof} 
By Proposition \ref{filtration vs quasimorphism-moved}, $\supp(f)\cap \phi^{-1}([-\alpha,\alpha])  $ is finite for all $\alpha\in\R$ if and only if $\supp(f)\cap N_i(\qker(\phi))$ is finite for all $i\in \N$, so that the isomorphism of $R[G]$-modules follows from 
Lemma \ref{Coarse coinduced module when $M=R$}.
\end{proof}

\begin{corollary}\label{Novikov ring and coinduced} Let $G$ be a finitely generated group and $\phi\colon G\to\R$ be a nontrivial homogeneous quasimorphism. There is an isomorphism of rings \[\widehat{R[G]}^\phi\cong\{f\in\mathrm{\Hom}_{\fd}(R[G]|_{\mathcal{Q}ker(\phi)},R):  \exists\alpha\in \R \text{ with } f(g)=0 \text{ for all } \phi(g)<\alpha\},\] and so an inclusion of $R[G]$-modules \[\widehat{R[G]}^\phi\hookrightarrow \Hom_{\fd}(R[G]|_{\qker(\phi)},R).\]
\end{corollary}

With this understanding of the coarse coinduced module in hand, we can relate the coarse cohomology of quasikernels to the vanishing of Novikov cohomology. For the remainder of this article, we will abuse notation by taking (co)homology with coefficients in $\widehat{R[G]}^{\pm\phi}$, by which we mean that with coefficients in both $\widehat{R[G]}^{\phi}$ and $\widehat{R[G]}^{-\phi}$ separately, when both (co)homology groups are isomorphic.

\begin{proposition}\label{cohomology at infinity}
    Let $G$ be a finitely generated group and $\phi\colon G\to\R$ be a nontrivial homogeneous quasimorphism such that \[\Homology^{i}(G;\widehat{R[G]}^{\pm\phi})\cong\Homology^{i+1}(G;\widehat{R[G]}^{\pm\phi})=0\] for some $i\geq 0$. Then \[\Homology_{\coarse}^i(\mathcal{Q}ker(\phi);R)\cong \Homology^{i+1}(G;R[G]).\]
\end{proposition}

\begin{proof}
    Consider the short exact sequence of left $R[G]$-modules \[0\to N\to \widehat{R[G]}^\phi\oplus \widehat{R[G]}^{-\phi}\to \Hom_{\fd}(R[G]|_{\mathcal{Q}ker(\phi)},R)\to 0\] induced by the inclusions 
    from Corollary~\ref{Novikov ring and coinduced}, where the kernel $N$ consists of precisely those pairs $(x,-x)$ such that  $\supp(x)$ is finite. 
  But then then $N\cong R[G]$ and we conclude the existence of a short exact sequence of left $R[G]$-modules \[0\to R[G]\to \widehat{R[G]}^\phi\oplus \widehat{R[G]}^{-\phi}\to \Hom_{\fd}(R[G]|_{\mathcal{Q}ker(\phi)},R)\to 0.\] By combining the long exact sequence in cohomology, the fact that \[\Homology^{\ast}(G;\widehat{R[G]}^{\phi}\oplus \widehat{R[G]}^{-\phi})\cong \Homology^{\ast}(G;\widehat{R[G]}^{\phi})\oplus \Homology^{\ast}(G;\widehat{R[G]}^{-\phi}) \] 
    and Theorem \ref{Intro coarse Shapiro},
    we conclude that\[\Homology^{i+1}(G;R[G])\cong \Homology^{i}(G;\Hom_{\fd}(R[G]|_{\mathcal{Q}ker(\phi)},R))\cong \Homology^{i}_{coarse}(\mathcal{Q}ker(\phi),R).\qedhere\]
   
\end{proof}

\begin{remark}
    It is likely that under some more assumptions, an alternative proof of Proposition \ref{cohomology at infinity} could be given using the K\"unneth theorem for coarse bundles \cite[Theorem $4.5$]{MR3281809382}, which builds on ideas of Kapovich--Kleiner \cite{MR2168506}. It is not clear to the authors how one might verify those assumptions in the setting that we apply this result, but we believe that it should be possible.
\end{remark}

The following is a coarse analogue of a result of Fisher~\cite[Theorem $D$]{MapstoZ}, that we believe may be of independent interest.

\begin{theorem}\label{ccd drop}
    Let $G$ be a group of type $\FP_n(R)$ with $\cd_R(G)=n$. If $\phi\colon G\to \R$ is a nontrivial homogeneous quasimorphism such that \[\Homology^n(G; \widehat{R[G]}^{\pm\phi} )=0,\] then $\ccd_R(\qker(\phi))\leq n-1$.
\end{theorem}

\begin{proof}[Sketch proof]
    Denote $X=\mathcal{Q}ker(\phi)$. By Theorem \ref{Intro coarse Shapiro} there is an isomorphism of cohomology groups  \[\Homology_{\coarse}^n(X;\textbf{M})\cong \Homology^n(G; \Hom_{\fd}( R[G]|_X,\textbf{M}))\] for any $R$-module $\textbf{M}$ over  $X$. Fix~$\M$ and let $L\coloneq \Hom_{\fd}( R[G]|_X,\textbf{M})$ and define
 \begin{align*}L^\phi & \coloneq \{f\in L: \text{ there exists } \alpha\in\R \text{ s.t } f(g)=0 \text{ for all } \phi(g)<\alpha\}, \\
L^{-\phi} & \coloneq \{f\in L: \text{ there exists } \alpha\in\R \text{ s.t } f(g)=0 \text{ for all } \phi(g)>\alpha\}.
\end{align*} 
It is routine to check that the natural left $R[G]$-action on $L^{\phi}$ which is given by $(g.f)(x)\coloneq f(xg)$ extends to give the structure of a left $\widehat{R[G]}^{-\phi}$-module, and analogously for~$L^{-\phi}$. From here there is then a short exact sequence of $R[G]$-modules  \[0\to N\to L^\phi\oplus L^{-\phi}\to L\to 0\] to which the long exact sequence in cohomology may be applied. An approach identical to that of the proof of \cite[Theorem $3.5$]{MapstoZ} readily yields that \[H^n_{\coarse}(X;\M)=0.\] As this holds for $\M$ being any $R$-module over $X$, we conclude by \cite[Proposition~$8.2$]{MR123} that $\ccd_R(X)\leq n-1$.
\end{proof}

\section{Novikov homology and coarse homological finiteness properties}\label{novikov homology section} 

In this section, we prove a theorem in the style of Sikorav \cite{Sikorav}  relating the vanishing of Novikov homology to the higher dimensional coarse homological finiteness properties of $\qker(\phi)$. Our approach is of a topological nature and takes inspiration from some of the basic methods of Bestvina--Brady Morse theory \cite{MR1465330}. The dimension $1$ case of our result was obtained in \cite{QBNS}, but we do not rely on this. In the case that $\phi$ is a group homomorphism, our proof is substantially different to those in the literature, see for example \cite{MR2283437} and \cite{MR4797112}.

\begin{theorem}\label{Novikov homology new}
    Let $G$ be a group of type $\FP_{n+1}(R)$ for $R$ a commutative ring and $\phi\colon G\to\R$ be a nontrivial homogeneous quasimorphism. If for all $i\leq n$
    we have
    \[\Homology_i(G;\widehat{R[G]}^{\pm\phi})=0,\] then $\qker(\phi)$ is coarsely $(n-1)$-acyclic over $R$.
\end{theorem}

The remainder of this section is dedicated to proving results which will combine to give a proof of Theorem~\ref{Novikov homology new}; in fact, it will be a straightforward consequence of Proposition~\ref{Novikov homology and sublevel sets} and Theorem~\ref{BNSR invariant and finiteness properties}.
\begin{remark}
    We note here that our result is not quite optimal -- it is almost surely the case that Theorem~\ref{Novikov homology new} holds when~$G$ is assumed only to be of type $\FP_n(R)$, as is the case when $n=1$ as in \cite{QBNS} or when $\phi$ is a homomorphism. This extra assumption is used to make our proof more efficient in two places. The first is in Lemma~\ref{Rips complexes compute group cohomology}, which we later use to recover Novikov homology from the homology of Rips complexes; this is analogous to using a classifying space with finite $(n+1)$-skeleton to compute group homology up to degree~$n$, and simplifies some arguments. The other place is in the proof of Theorem \ref{BNSR invariant and finiteness properties}, where a Mayer--Vietoris argument is used to prove a generalisation of the BNSR theorem, but at the cost of assuming~$G$ to be of type~$\FP_{n+1}(R)$.
\end{remark}

Recall that a metric space is coarsely $(n-1)$-acyclic over $R$ if for every $r\in\R$ there is $s\geq r$ such that the map on reduced simplicial homology \[\tilde{H}_i(P_r(X);R)\to\tilde{H}_i(P_s(X);R)\] is trivial for all $i\leq n-1$. A finitely generated group $G$ is of type $\FP_n(R)$ if and only if it is coarsely $(n-1)$-acyclic as a metric space. 
Observe that, in the following, it is sufficient to confirm the above condition for $r\in \N$, as we are considering groups and their subspaces where distances are integer-valued.
We now state a well-known fact that will allow us to recover the vanishing of Novikov homology from the equivariant homology of Rips complexes, a proof can be given following a similar approach to the proof of \cite[Theorem $2.2$]{MR885095}.

\begin{lemma}\label{Rips complexes compute group cohomology}
    Let $G$ be a group of type $\FP_{n+1}(R)$ and $M$ a right $R[G]$-module. Then for every $ i\leq n$ we have $\Homology_i(G;M)=0$ if and only if for every $r\in\N$ there exists $s\geq r$ such that we obtain triviality of the map on equivariant homology\[\tilde{\Homology}_i^G(P_r(G);M)\to \tilde{\Homology}^G_i(P_s(G);M).\]
\end{lemma}

     The following result will allow us some more flexibility in relating coarse geometry and the topology of Rips complexes.
    \begin{lemma}\label{Neighbourhoods in Rips complexes}
    Let~$Y$ be a discrete bounded-geometry metric space and  $X\subset Y$ a subspace. If for every $r\in\N$ there are $K_r\geq 0$ and $s\geq r$ such that the map on reduced homology\[\tilde{H}_i(P_r(X);R)\to \tilde{H}_i(P_s(N_{K_r}(X));R)\] is trivial for all $i\leq n$, then $X$ is coarsely $(n-1)$-acyclic.
\end{lemma}

\begin{proof}
    Given an $i$-cycle $\sigma$ in $P_r(X)$ for $i\leq n-1$, there exists an $(i+1)$-chain~$\rho$ in $P_s(N_{K_r}(X))$ such that $\partial \rho=\sigma$. 
   By Lemma \ref{Projecting is a coarse equivalence}, any closest-point projection $\pi_r\colon N_{K_r}(X)\to X $ is a quasiisometry that restricts to identity map on $X$. By \cite[Lemma $3$]{MR1293049}, $\pi_r$ extends to a simplicial map \[\bar{\pi}_{r}\colon P_s(N_{K_r}(X))\to P_{s'}(X) , \quad \{x_0,\dots,x_n\}\mapsto \{\pi_r(x_0),\dots, \pi_r(x_n)\},\] where $s'\geq s$ is chosen large enough relative to the quasiisometry constants of $\pi_r$. By construction we have $\bar{\pi}_r|_{P_r(X)}=\id$, so the induced map on simplicial chains satisfies 
   \[\partial\bar{\pi}_r(\rho)=\bar{\pi}_r(\partial \rho)=\sigma.\qedhere\]
   \end{proof}

   We now set up some notation concerning the equivariant homology of Rips complexes with coefficients in the Novikov ring. Let $G$ be a finitely generated group and denote by $C_\bullet(P_r(G))$ the left-equivariant simplicial $R$-chain complex of $P_r(G)$. As $P_r(G)$ is a finite $G$-simplicial complex, this is a finite and finitely generated $R[G]$-chain complex.

   \begin{definition}
    Let $\widehat{R[G]}^\phi$ be the Novikov ring of a  a nontrivial homogeneous quasimorphism $\phi\colon G\to\R$. Observe that an element of $\widehat{R[G]}^\phi\otimes_{R[G]}C_i(P_r(G))$  can be written as a (possibly infinite) formal sum of $i$-simplices in $P_r(G)$ with coefficients in $R$. We make the following definitions.
    \begin{enumerate}
        \item We call elements $x\in \widehat{R[G]}^\phi\otimes_{R[G]} C_i(P_r(G))$ \textit{Novikov $i$-chains}; Novikov cycles are defined analagously.
        \item For a Novikov $i$-chain $x$, we define the \textit{support} of $x$, denoted $\supp(x)$ to be the set of simplices in $P_r(G)$ with nonzero coefficients in the formal sum representing $x$. 

        \item If $x$ is a Novikov $i$-chain and $A\subset \supp(x)$ is a subset of simplices, we form a Novikov $i$-chain $x\vert_A$ with $\supp(x\vert_A)=A$ which consists of the summands of $x$ with simplices in $A$. We call $x\vert_A$ the \textit{restriction} of $x$ to $A$. 
    \end{enumerate}
    \end{definition}

We now introduce the quasi-BNSR invariant of a group, following \cite{MR914846}, \cite{MR960770} and \cite{QBNS}. Let $\mathcal{Q}(G)$ denote the $\R$-vector space of homogeneous quasimorphisms $\phi\colon G\to\R$ under pointwise addition and scalar multiplication and set 
\[G_\phi\coloneq \{g\in G : \phi(g)\geq -2D(\phi)\}.\]

\begin{definition}
    Let $R$ be a commuative ring and $G$ be a group of type $\FP_n(R)$. For $k\leq n$ we define the $k$th \textit{quasi-BNSR} invariant by \[\mathcal{Q}\Sigma^k_R(G)\coloneq\{\phi\in\mathcal{Q}(G): G_\phi \text{ is coarsely } (k-1)\text{-acyclic over } R\}.\] 
\end{definition}

\begin{remark}
    When $\phi$ is a homomorphism, the BNSR invariant is often defined in terms of the homological finiteness properties of $G_{\phi}$ as a monoid instead of its coarse finiteness properties as a metric subspace $G_{\phi}\subseteq G$. That these notions coincide follows from an argument analogous to that of \cite[Theorem $2.2$]{MR885095}.
\end{remark}

The following technical lemma allows us to assign Novikov chains to simplices in $P_r(G)$, which we will use in the proof of Proposition \ref{Novikov homology and sublevel sets} to `push' filling chains into the Rips complex of the half-space $G_\phi\subseteq G$.

\begin{lemma}\label{Novikov pushing chains}
Let $G$ be a group of type $\FP_{n+1}(R)$ and $\phi\colon G\to\R$ be a nontrivial homogeneous quasimorphism. For every $r\in\N$ there is $s\geq r$ such that for every $k\leq n$ and all $k$-simplices $\Delta$ in $P_r(G)$, there are choices of Novikov $(k+1)$-chains $\zeta_\Delta$ in $P_{s}(G)$ such that for all $g\in G$ we have
\[\partial\zeta_\Delta=\Delta+\sum_{\Delta'\in\supp(\partial\Delta)}\zeta_{\Delta'}\quad\text{and}\quad \zeta_{g\Delta}=g\zeta_\Delta.\]
\end{lemma}

\begin{proof}
Lemma \ref{Rips complexes compute group cohomology} implies that for every $r\in\N$ there exists $s\geq r$ such that the map on equivariant homology\[\tilde{\Homology}_i^G(P_r(G);\widehat{R[G]}^\phi)\to \tilde{\Homology}^G_i(P_s(G);\widehat{R[G]}^\phi)\] is trivial for $ i\leq n$. 
    
  We now proceed by induction on $k$. The case of $k=0$ is a consequence of the fact that~$\phi$ is unbounded, so that  to any vertex $v$ we may attach a suitable path on whose vertices the value of $\phi$ tends to $\infty$, corresponding to a Novikov $1$-chain with boundary $v$. The inductive step follows from the vanishing of Novikov homology.
\end{proof}

\begin{proposition}\label{Novikov homology and sublevel sets}

    Let $G$ be a group of type $\FP_{n+1}(R)$ and $\phi\colon G\to\R$ be a nontrivial homogeneous quasimorphism. If for all $i\leq n$ we have
    \[\Homology_i(G;\widehat{R[G]}^{\phi})=0,\] 
    then $\phi\in\mathcal{Q}\Sigma^n_R(G)$.
\end{proposition}

\begin{proof}

Fix $k\leq n-1$ and consider a $k$-cycle $\sigma\in\tilde{H}_k(P_r(G_\phi);R)$, which we identify with a cycle in $P_r(G)$ by inclusion. As $G$ is of type $\FP_{n}(R)$, it is coarsely $(n-1)$-acyclic, hence there exists $s\geq r$ and a $(k+1)$-chain $\rho$ in $P_s(G)$ such that $\partial\rho=\sigma$. It now remains to argue that we may replace $\rho$ by a chain inside $P_s(G_\phi)$. Applying Lemma \ref{Novikov pushing chains} to all simplices in $\supp(\rho)$, we may assume, possibly after increasing~$s$, that there is a Novikov $(k+2)$-chain $\zeta_\rho$ in $P_s(G)$ such that \[\partial\zeta_\rho= \rho+\sum_{\Delta\in\supp(\sigma)}\zeta_{\Delta}.\] 

Let $\kappa$ be the restriction of $\zeta_\rho$ to all simplices having some vertex on which $\phi$ is less than $-2D(\phi)$. We claim that $\rho-\partial\kappa$ is a $(k+1)$-chain filling $\sigma$ whose support is contained in $P_s(N_K(G_\phi))$ for some $K>0$ depending only on $r\in\N$. Together with Lemma~\ref{Neighbourhoods in Rips complexes}, this concludes the proof. 

To prove the claim, observe that $\partial\kappa=\partial(\kappa-\zeta_\rho)+\partial\zeta_\rho$, so that every simplex in $\supp(\rho-\partial\kappa)$ satisfies $\phi (v)\geq -2D(\phi)$ for all its vertices~$v$ or is contained in $\bigcup_{\Delta\in\supp(\sigma)}\supp(\zeta_{\Delta})$. Equivariance of the Novikov chains $\zeta_\Delta$  and the fact that $P_r(G)$ is a finite $G$-simplicial complex imply that~$\phi$ is uniformly bounded below on $\supp(\zeta_\Delta)$ for $\Delta\in\supp(\sigma)$, where the lower bound depends only on $r\in\N$, from which we conclude by Proposition \ref{filtration vs quasimorphism-moved}.
\end{proof}

We now prove some lemmata that will later allow us to apply the Mayer--Vietoris sequence to Rips complexes to obtain the main result of this section. The following is proven using a slight modification of a standard result in combinatorial topology concerning the extension of a subdivision of a subcomplex of a simplicial complex..

\begin{lemma}\label{subdivision of simplicial complex}
    For any $s\in\N$, there is a subdivision $\mathrm{sd}(P_s(G))$ of $P_s(G)$ with the following properties, where $\Delta$ is a simplex in $P_s(G)$.

    \begin{enumerate}
        \item If $\Delta\subseteq P_s(G_\phi)\cup P_s(G_{-\phi})$, then it is not subdivided.
        \item If $\Delta\cap \qker(\phi)\ne \emptyset$, then the interior of $\Delta$ does not contain new vertices in $\mathrm{sd}(P_s(G))$.
        \item Otherwise, a vertex $y_\Delta$ is added at the barycentre of $\Delta$ in $\mathrm{sd}(P_s(G))$.
        
    \end{enumerate}
\end{lemma}

\begin{proof}  We proceed by induction on $k$-skeleta to build a subdivision of $P_s(G)$, where the $k=0$ case is straightforward. 

Let $\Delta$ be a $k$-simplex in $P_r(G)$, and assume that we have subdivisions of the simplices in the support of $\partial \Delta$. If $\Delta\cap X=\emptyset$ and $\Delta\not\subseteq P_s(G_\phi)\cup P_s(G_{-\phi})$, then we subdivide $\Delta$ by taking the join of a vertex at its barycentre $y_\Delta$ with each simplex in the subdivisions of those simplices in the boundary of $\Delta$. If $\Delta\cap \qker(\phi)\ne \emptyset$, then we may replace the vertex $y_\Delta$ with one of its vertices $x_\Delta\in X$ and take the join as before. That this is well defined follows from the fact that any simplex in the boundary of $\Delta$ that contains $x_\Delta$ also contains no new vertices in its interior. Finally, if $\Delta\subseteq P_s(G_\phi)\cup P_s(G_{-\phi})$, we see by induction that the simplices in its boundary are not subdivided, so that we do not need to subdivide $\Delta$ either. \end{proof}

\begin{lemma}\label{Union of Rips complexes and group}

Let $G$ be of type $\FP_n(R)$ and $\phi\colon G\to\R$ be a nontrivial homogeneous quasimorphism.  For every $r\in\N$ there exists $s\geq r$ such that the map \[\tilde{H}_i( P_r(G_\phi)\cup P_r(G_{-\phi});R)\to \tilde{H}_i(P_s(G_\phi)\cup P_s(G_{-\phi});R)\] is trivial for all $i\leq n-1$.

\end{lemma}

\begin{proof}

Let $X=\qker(\phi)$. From Proposition \ref{filtration vs quasimorphism-moved} we know that for every $s\in\N$ there is $s'\geq s$ such that for every simplex \[\Delta\subseteq P_s(G)\setminus (P_s(G_\phi)\cup P_s(G_{-\phi}))\] there exists $x_\Delta\in X$ with $d(v,x_\Delta)\leq s'$ for every vertex $v$ of $\Delta$. For each $\Delta$ we fix such a choice of $x_\Delta$, requiring that~$x_\Delta$ be a vertex of~$\Delta$ where possible.
We may further assume that there is $t\geq s'$ such that for every simplex $\Delta$ in $P_s(G)$, there is a simplex $U(\Delta)$ in $P_t(G)$ containing both $\Delta$ and the vertices $x_\sigma$ for each subsimplex~$\sigma$ of~$\Delta$.

 We now recall the subdivisions $\mathrm{sd}(P_s(G))$ and $\mathrm{sd}(P_{s'}(G))$ from Lemma \ref{subdivision of simplicial complex}, where by construction we may assume that there is an inclusion of subcomplexes $\mathrm{sd}(P_s(G))\subseteq\mathrm{sd}(P_{s'}(G))$.
 
 We will define a simplicial map \[\tau\colon \text{sd}(P_s(G))\to P_t(G_\phi)\cup P_t(G_{-\phi})\]
 such that the restriction of $\tau$ to $P_s(G_\phi)\cup P_s(G_{-\phi})$ is the natural inclusion. We achieve this by inductively proving the existence of compatible simplicial maps \[\tau_\Delta\colon\Delta\to (P_t(G_\phi)\cup P_t(G_{-\phi}))\cap U(\Delta)\] for $\Delta$ being either the image of a $k$-simplex in $P_s(G)$ under the subdivision $\text{sd}(P_{s}(G))$, or the image of a $k$-simplex in $P_{s'}(G)$ in $\text{sd}(P_{s'}(G))$ having a vertex inside $X$. For the latter simplices we extend the definition of $x_\Delta$ to be a choice of one of its vertices lying in $X$. These maps will also have the property that they restrict to the natural inclusion on $P_s(G_\phi)\cup P_s(G_{-\phi})$.
The base case $k=0$ is the identity map.

Let $\Delta$ be a $k$-simplex as above for $k\geq 1$. If $\Delta\subseteq P_{s'}(G_\phi)\cup P_{s'}(G_{-\phi})$, then $\tau_\Delta$ is the natural inclusion, which is simplicial as $\Delta$ is not subdivided in $\mathrm{sd}(P_{s'}(G_\phi))$. We now consider the case that $\Delta$ does not contain a vertex in $X$. Recall that then there exists a vertex $x_\Delta\in X$ such that the join $\Delta\star x_\Delta$ is a $(k+1)$-simplex in $P_{s^\prime}(G)$. Note also that there is vertex $y_\Delta$ at the barycentre of $\Delta$, and define a simplicial map on subdivisions $\Delta\to \Delta\star x_\Delta$ sending $y_\Delta$ to $x_\Delta$ whilst fixing all other vertices. This reduces to the case that $\Delta$ contains a vertex in~$X$.

Assume that $\Delta$ is a $k$-simplex containing $x_\Delta$ as a vertex, and let $\sigma$ be the face of $\Delta$ opposite to $x_\Delta$. By induction we may assume that there is a map \[\tau_\sigma\colon\sigma\to (P_t(G_\phi)\cup P_t(G_{-\phi}))\cap U(\sigma).\] But then, as $U(\sigma)\subseteq U(\Delta)$, this extends to the desired simplicial map  
\[\tau_\Delta\colon \Delta\to (P_t(G_\phi\cup P_t(G_{-\phi}))\cap U(\Delta)\] by fixing $x_\Delta$. We can now define $\tau(x)\coloneq \tau|_\Delta(x)$ for any $x\in \Delta\subseteq \text{sd}(P_s(G))$, which is well defined by compatibility of the maps $\tau_\Delta$.

  We denote also by $\tau$ the induced continuous map on $P_s(G)$. There is then the following commutative square for any $r,s\in\N$,  
\[\begin{tikzcd}
	{P_r(G_\phi)\cup P_r(G_{-\phi})} & {P_{t}(G_\phi)\cup P_{t}(G_{-\phi})} \\
	{P_r(G)} & {P_s(G)}
	\arrow[hook, from=1-1, to=1-2]
	\arrow[hook, from=1-1, to=2-1]
	\arrow[hook, from=2-1, to=2-2]
	\arrow["\tau", from=2-2, to=1-2],
\end{tikzcd}\]
where commutativity follows from the fact that $\tau$ restricts to the natural inclusion on $P_s(G_\phi)\cup P_s(G_{-\phi})$.
As $G$ is of type $\FP_n(R)$, we may assume that the map on reduced homology induced by $P_r(G)\hookrightarrow P_s(G)$ is trivial in degrees $i\leq n-1$. But then the map on reduced homology induced by the inclusion in the top row is trivial in the same degrees, from which we conclude.
\end{proof}

\begin{theorem}\label{BNSR invariant and finiteness properties}
    If $G$ is a group of type $\FP_{n+1}(R)$ for $R$ a commutative ring and $\phi\colon G\to\R$ is a nontrivial homogeneous quasimorphism, then $\qker(\phi)$ is coarsely $(n-1)$-acyclic over $R$ if and only if $\{\phi,-\phi\}\subseteq\mathcal{Q}\Sigma^n_R(G)$.
\end{theorem}

\begin{proof} Let $X=\qker(\phi)$. Consider the following truncated portion of the Mayer--Vietoris sequence over $R$ induced by the intersection $P_r(G_\phi)\cap P_r(G_{-\phi})=P_r(X)$, where $i\geq 0$ and $r\in\N$,

  \[
\adjustbox{scale=1,center}{
     \begin{tikzcd}[column sep=tiny] 
         & \phantom{31}\cdots
          \arrow[r] \arrow[d, phantom, ""{coordinate, name=ZZ}]
         & {\tilde{H}_{i+1}(P_r(G_\phi)\cup P_r(G_{-\phi}))} 
           \arrow[r]
         & {\tilde{H}_i(P_r(X))}
          \arrow[dll,  rounded corners, to path={ --
([xshift=2ex]\tikztostart.east)|- (ZZ) [near end]\tikztonodes-|
([xshift=-2ex]\tikztotarget.west)-- (\tikztotarget)}] \\
         & {\tilde{H}_i(P_r(G_\phi))\oplus \tilde{H}_i( P_r(G_{-\phi}))}
          \arrow[r]
         & \makebox[0pt][l]{$\cdots$} 
     \end{tikzcd}
}
\]

Note that direct products and direct limits are exact in the category of $R$-modules, so that we may apply the functor $\varinjlim_{r\in\N}(\prod_{J}(-)) $ for any index set $J$ to obtain a new long exact sequence. But now by \cite[Lemma $2.1$]{MR885095}, a metric space $Y$ is coarsely $(n-1)$-acyclic over $R$ if and only if for all index sets $J$ and $i\leq n-1$ we have \[\varinjlim_{r\in\N}\left(\prod_J\tilde{H}_i(P_r(Y);R)\right)=0.\] By combining this with Lemma \ref{Union of Rips complexes and group} and the fact that $G$ is of type $\FP_{n+1}(R)$, we obtain for every $i\leq n-1$ the vanishing of the first term of the truncated Mayer--Vietoris sequence under the functor $\varinjlim_{r\in\N}(\prod_{J}(-))$. This implies that \[{\varinjlim_{r}(\prod_J \tilde{H}_i(P_r(X);R))}\cong{\varinjlim_{r}(\prod_J \tilde{H}_i(P_r(G_\phi);R)\oplus \tilde{H}_i( P_r(G_{-\phi});R))} \] for every $i\leq n-1$. It now follows by definition of $\mathcal{Q}\Sigma^k(G)$ that $\{\phi,-\phi\}\subseteq\mathcal{Q}\Sigma^k(G)_R$ if and only if $X$ is coarsely $(n-1)$-acyclic over $R$. 
\end{proof}

\section{Coarse Poincar\'{e} duality in dimension two}\label{section:corase_pr2_spaces}

In this section we classify coarse Poincar\'{e}-duality spaces in dimension two as being either amenable or quasiisometric to~$\HH^2$. 
This will later be of use as the quasikernel appearing in the statement of Theorem \ref{thm:manifold} will turn out to be such a space. We will assume that our metric spaces are discrete and of bounded geometry, as this is required for the following definition of amenability.
\begin{definition}
    Let $X$ be a metric space and $A\subset X$ be a finite subset. For  $L\geq 0$ we define the \textit{$L$-boundary of $A$} as  
\[\partial_{L}A\coloneq \{x\in X: d(x, A)\leq L \text{ and } d(x, X\setminus A)\leq L\}.\]
\end{definition}

\begin{definition}
    A discrete bounded-geometry metric space $X$ is \textit{amenable} if for every $L\geq 0$ and $\epsilon>0$ there is a finite subset $A\subseteq X$ such that \[\frac{|\partial_LA|}{|A|}<\epsilon.\] We call such $A\subseteq X$ \textit{F\o lner sets}, and a sequence of F\o lner sets for which $\epsilon\to 0$ with $L$ fixed a \textit{F\o lner sequence}.
\end{definition}

\begin{remark}\label{basic facts about amenability}
    
Amenability of a discrete bounded-geometry metric space is invariant under coarse equivalence, see for instance \cite[Proposition 3.C.29]{MR3561300}. Further, amenability of a finitely generated group as such a metric space is equivalent to amenability as a group \cite[Proposition 3.1.7]{MR4646531}.

\end{remark}

The main result of this section is the following partial classification of coarse $\PD^2(R)$ spaces. We remark that there are no coarse homotopic finiteness assumptions on our spaces (only the coarse homological ones coming from the definition of coarse Poincar\'e duality), but that $\HH^2$ is coarsely simply-connected.

\numberedtheorem{Theorem~\ref{intro-characterising PD^2}}{}{ If $X$ is a discrete quasigeodesic bounded-geometry coarse $\PD^2(R)$ space over a commutative ring $R$, then it is either amenable or quasiisometric~to~$\HH^2$.}

\begin{remark}
    We note here that,  without additional conditions such as coarse homogeneity,  amenability of a metric space is quite a weak condition. In this generality it should be thought of as a sort of `local amenability', as it only requires the existence of a F\o lner sequence. 
\end{remark}

Our approach follows closely that of Kielak--Kropholler \cite{kielak2021isoperimetric}, who derive a linear homological isoperimetric inequality from the converse of amenability in the case of $\PD^2(R)$ groups. This choice of method requires us to set up a sufficiently general framework for such inequalities. It turns out that projective resolutions over metric spaces provide a natural setting for this purpose, since finite displacement maps allow one to control the supports of module elements in a precise manner. Our definitions are partly inspired by those in \cite{MR32141849}, where implicitly a chain complex very similar to the standard resolution is used to define a notion of linear isoperimetric inequality for a metric graph.

\begin{definition}
    Given a projective $R$-resolution $\textbf{C}_\bullet$ over a metric space $X$ and an increasing function $\Theta\colon \N\to\N$, we define the \emph{$\Theta$-filling norm}~$|\cdot|_{\Theta}^j$ of an element $x\in \textbf{C}_{n}(j)\cap \partial\textbf{C}_{n+1}$ by \[|x|^j_\Theta\coloneq \inf\{|\supp_{\textbf{C}_{n+1}}(y)| :  y\in \textbf{C}_{n+1}(\Theta(j)) \text{ and } \partial y=x\},\]
where we set $|x|_\Theta^j:=\infty$ whenever the set we take the infimum of is empty.
\end{definition}

\begin{definition}
    Given a projective $R$-resolution $\textbf{C}_\bullet$ over a metric space $X$ and a set of increasing functions $\F=\{\eta_j\colon\R_{\geq 0}\to\R_{\geq 0}\}_{j\in\N}$, we say that $\textbf{C}_\bullet$ satisfies an \textit{$\F$-homological isoperimetric inequality in dimension $n$} if there is an increasing function $\Theta\colon \N\to\N$ such that for every $j\geq 0$ and $x\in \textbf{C}_{n}(j)\cap \partial\textbf{C}_{n+1}$ we have \[|x|^j_\Theta\leq \eta_j(|\supp_{\textbf{C}_{n}}(x)|).\]
    If the set of functions $\F$ consists only of linear (resp. quadratic, polynomial, exponential etc.) functions, we say that $X$ satisfies a linear (resp. quadratic, polynomial, exponential etc.) homological isoperimetric inequality in dimension $n$.
\end{definition}

\begin{remark}
    If $\textbf{C}_{n+1}$ has finite height, we may replace the set of functions $\F$ with a single function in the obvious way. This can be applied for instance when~$X$ is a group of type $\FP_{n+1}(R)$, see \cite[Proposition $9.3$]{MR123}.
\end{remark}

Recall that for increasing functions $\eta,\mu\colon\R_{\geq 0}\to\R_{\geq 0}$, there is a partial order defined by setting $\eta\preccurlyeq \mu$ if and only if there is a constant $K>0$ such that $\eta(x)\leq K\mu(Kx+k)+Kx+K$ for all $x\in\R_{\geq 0}$. We say that $\eta\approx \mu$ if $\eta\preccurlyeq \mu$ and $\mu\preccurlyeq \eta$. For sets of functions $\F=\{\eta_j\colon\R_{\geq 0}\to\R_{\geq 0}\}_{j\in\N}$ and $\F'=\{\eta'_j\colon\R_{\geq 0}\to\R_{\geq 0}\}_{j\in\N}$, we define $\F\preccurlyeq\F'$ and $\F\approx \F'$ by comparing functions at each $j\in\N$, where we do not require uniformity of the constant $K_j$. One can confirm that our notion of $\F$-homological isoperimetric inequality is invariant under canonical isomorphisms of chain modules in $\textbf{C}_{\bullet}$, up to~$\approx$. We now show that for a discrete bounded-geometry metric space this does not depend on the choice of  $R$-resolution.

\begin{lemma}\label{iso inequality and different resolutions}

        Let $\textbf{C}_{\bullet}$, $\textbf{D}_\bullet$ be projective $R$-resolutions over a discrete metric space~$X$ of bounded geometry. If $\textbf{D}_\bullet$ satisfies a $\F$-homological isoperimetric inequality in dimension $n\geq 1$, then $\textbf{C}_\bullet$ satisfies a $\F'$-homological isoperimetric inequality in dimension $n\geq 1$ for some $\F'\approx\F$.
 \end{lemma}

    \begin{proof}
        By \cite[Corollary $6.14$]{MR123}, for any two projective $R$-resolutions over~$X$ denoted $\textbf{C}_\bullet$ and $\textbf{D}_\bullet$  
        there exist augmentation-preserving finite-displacement chain maps of $R$-chain complexes $f_\bullet\colon \textbf{C}_\bullet\to\textbf{D}_\bullet$ and $g_\bullet\colon \textbf{D}_\bullet\to\textbf{C}_\bullet$ such that $g_\bullet \circ f_\bullet\simeq \id_{\textbf{C}_\bullet}$ and $f_\bullet\circ g_\bullet\simeq \id_{\textbf{D}_\bullet}$ by finite-displacement chain homotopies, which we call~$s$ and~$s^\prime$ respectively. We denote all chain maps by~$\partial$.   
        Remark \ref{staggering} ensures that, up to canonical isomorphism of chain modules, all boundary maps, chain homotopies and chain maps have a uniform displacement function $\Phi\colon \N\to\N$, which can be chosen such that $\Phi(j)\geq j$ for every $j\in \N$.
        We assume that~$\mathbf{D}_\bullet$ satisfies an $\F$-homological isoperimetric inequality with $\F=\{\eta_j\colon\R_{\geq 0}\to\R_{\geq 0}\}_{j\in\N}$ over a function $\Theta\colon \N\to\N$, which we may also choose to have $\Theta(j)\geq j$ for every $j\in\N$.

        Let $\gamma\in \textbf{C}_{n}(j)\cap \partial\textbf{C}_{n+1} $ for some $j\in\N$. We prove that there exists an increasing function $\Theta'\colon \N\to\N$ and $\eta_j'\colon\R_{\geq 0}\to\R_{\geq 0}$ not depending on~$\gamma$ such that $|\gamma|_{\Theta^\prime}^j\leq \eta_j^\prime(|\supp_{\textbf{C}_n}(\gamma)|)$. This will be achieved by pulling back the corresponding $\F$-homological isoperimetric inequality over $\textbf{D}_\bullet$. 
        
        Let $C= C(\Phi(j))$ be an upper bound on the cardinality of balls of radius~$|\Phi(j)|$; such a bound exists as~$X$ is discrete and of bounded geometry by assumption. As  $\textbf{D}_\bullet$ satisfies an $\F$-homological isoperimetric inequality in dimension~$n$, there is $d\in\textbf{D}_{n+1}(\Theta(\Phi(j)) $ such that $\partial d=f_{n}(\gamma)$ and
        \begin{equation}\label{eq:homological_iso_ineq}
            |\supp_{\textbf{D}_{n+1}}(d)|\leq \eta_j(|\supp_{\textbf{D}_n}(f_{n}(\gamma))|)\leq \eta_j(C|\supp_{\textbf{C}_n}(\gamma)|),
        \end{equation}
       where the last inequality follows from the fact that~$f_{n}$ is of finite displacement. We set $\Theta' \coloneq \Phi\circ\Theta\circ\Phi$ and $d'\coloneq g_{n+1}(d)\in\textbf{C}_{n+1}(\Theta'(j))$. It is easy to see that \[\partial d'=g_{n}(\partial d)=g_{n}(f_{n}(\gamma))=\gamma+\partial s(\gamma),\]
        where $s\colon \textbf{C}_\bullet\to \textbf{C}_{\bullet+1}$ is the chain homotopy introduced above. We show that 
        $d'-s(\gamma)$
        is the desired filling of $\gamma$. Indeed, by Lemma~\ref{lem:Margolis-Lemma3_4} 
        %and the triangle inequality 
        we have \[|\supp_{\textbf{C}_{n+1}}(d'-s(\gamma))|\leq |\supp_{\textbf{C}_{n+1}}(d')|+|\supp_{\textbf{C}_{n+1}}(s(\gamma))|.\] 
       
        Further, as $s$ and $g$ are finite displacement maps, we have  
        \[|\supp_{\textbf{C}_{n+1}}(d')|\leq C|\supp_{\textbf{D}_{n+1}}(d)|\quad \text{ and }\quad |\supp_{\textbf{C}_{n+1}}(s(\gamma))|\leq C|\supp_{\textbf{C}_n}(\gamma)|.\] Finally, from~\eqref{eq:homological_iso_ineq}, we obtain $|\supp_{\textbf{D}_{n+1}}(d)|\leq \eta_j(C|\supp_{\textbf{C}_n}(\gamma)|)$. 
        We conclude by combining these inequalities to get that \[|\supp_{\textbf{C}_{n+1}}(d'-s(\gamma))|\leq C \eta_j(C|\supp_{\textbf{C}_n}(\gamma)|)+C|\supp_{\textbf{C}_n}(\gamma)|.\]
        Setting $\eta_j^\prime(r):=C\eta_j(r)+Cr$ concludes the proof.
    \end{proof}

    We can now define homological isoperimetric inequalities for discrete bounded-geometry metric spaces in the obvious way.

\begin{corollary}\label{iso inequality and qi}

If $X$ and $Y$ are quasiisometric discrete bounded-geometry metric spaces, then $X$ satisfies a $\F$-homological isoperimetric inequality in dimension $n$ over~$R$ if and only if $Y$ does.

\end{corollary}

\begin{proof}
  
    By \cite[Proposition $6.16$]{MR123},
    the pullback of a projective $R$-resolution over~$X$ with respect to a quasiisometry $f\colon Y\to X$ is a projective $R$-resolution over~$Y$. Together with the observation that a quasiisometry of discrete bounded-geometry metric spaces preserves the cardinality of finite subsets up to multiplication by a constant depending only on its additive constant, this yields the result.
\end{proof}

\begin{remark}
  Bader--Kropholler--Vankov \cite{MR4932386} conjectured that quasiisometric groups of type $\FP_{n+1}(R)$ satisfy the same group theoretic homological isoperimetric inequalities up to dimension $n$. This has recently been confirmed by work of Weis \cite{MR3840928094}, but we expect that an alternative proof (in the case of the discrete norm on $R$) could be given using Corollary \ref{iso inequality and qi}.
\end{remark}

We now turn to the study of homological isoperimetric inequalities in coarse $\PD^2(R)$ spaces. We first prove an elementary lemma in coarse homological algebra.

\begin{lemma}\label{ccd=2}
    Let $X$ be a coarsely uniformly $1$-acyclic metric space of coarse cohomological dimension $2$ over $R$. Then there exists $i_0\in\N$ such that for every $i\geq i_0$ there is a finite-height projective $R$-resolution over $X$ of the form \[0\to \textbf{P}\to \textbf{C}_1(i)\xrightarrow{\partial} R[X]\to R\to 0,\] where $\textbf{C}_1(i)\coloneq R \{[x_0,x_1]\colon   d(x_0,x_1)\leq i\}.$
\end{lemma}

\begin{proof}
     The map of $R$-modules $\partial\colon \textbf{C}_1(i)\to R[X]$ is given by $[x_0,x_1]\mapsto x_0-x_1$. We first show that, for $i$ large enough, $\textbf{C}_1(i)\to R[X]\to R\to 0$ satisfies the assumptions of \cite[Proposition $6.26$]{MR123},  which shows that it admits an extension to a full resolution. It is clear from Example~\ref{free module over X} that $\textbf{C}_1(i)$ and $R[X]$ are free $R$-modules over $X$, so it suffices to verify that $\partial$ is a finite-displacement map and that the conditions $(1)-(4)$ of Definition \ref{projective res over $X$} hold. Note however that our complex will need to be extended to give a full resolution.
     
     As $\supp_{R[X]}(\partial[x_0,x_1])=\{x_0,x_1\}\subseteq N_i(\supp_{\textbf{C}_1(i)}([x_0,x_1]))=N_i(x_0)$, the map~$\partial$ has finite displacement. As $X$ is coarsely uniformly $0$-acyclic, there exists some $i_0\in\N$ so that for $i\geq i_0$ the sequence is exact at $R[X]$ and $\partial$ has uniform preimages, from which we conclude a chain complex satisfying the desired assumptions. We may now extend this to a full resolution  \[\dots\to \textbf{C}_2\to \textbf{C}_1(i)\to R[X]\to R\to 0.\] As $X$ is coarsely uniformly $1$-acyclic, the proof of \cite[Proposition $9.3$, $(3)\implies (4)$]{MR123} allows us to assume that $\mathbf{C}_2$ has finite height. But then by \cite[Proposition~$8.2$~$(1)$]{MR123} we may truncate at length~$2$ as $\ccd_R(X)=2$, by which we replace $\textbf{C}_2$ with a projective summand $\textbf{P}$. The result now follows as~$\textbf{P}$ has finite height by \cite[Proposition~$3.27$~(1)]{MR123}.\end{proof}

\begin{proposition}\label{amenable or linear iso inequality}
    Let $X$ be a discrete bounded-geometry coarse $\PD^2(R)$ space. If $X$ is not amenable, it satisfies a linear homological isoperimetric inequality in dimension~$1$ over $R$.
\end{proposition}
\begin{proof}
    First note that as $X$ is a coarse $\PD^2(R)$ space, Proposition~$12.4$ of \cite{MR123} implies that $\ccd_R(X)=2$ and that $X$ is coarsely uniformly $1$-acyclic, so from Lemma~\ref{ccd=2} we obtain a projective $R$-resolution $\textbf{P}_\bullet\to R\to 0$ over $X$ of the form
    \begin{equation*}
        0\to \textbf{P} \overset{}{\to} \textbf{C}_1(i) \xrightarrow{\partial} R[X]\to R \to 0.
    \end{equation*}

    By coarse Poincar\'{e} duality 
    , the dual complex $\textbf{P}_c^{2-\bullet}$ is finite-displacement chain homotopy equivalent to $\textbf{P}_\bullet$, hence it gives a projective $R$-resolution over $X$ of the form
    \begin{equation*}
        0\to  \Hom_{\fd}(R[X],R) \overset{\delta}{\to} \Hom_{\fd}(\textbf{C}_1(i),R) \overset{}{\to} \Hom_{\fd}(\textbf{P},R) \overset{}{\to} R \to 0,
    \end{equation*}
    with  
    \(\delta(f)= f\circ \partial\) 
    given by \[(f\circ \partial)([x_0, x_1])=f(x_0)-f(x_1).\] 

    By \cite[Lemma $5.6$ and Proposition $5.9$]{MR123} we know that $f\in \Hom_{\fd}(R[X],R)$ if and only if $f$ is finitely supported. One confirms easily that $\Hom_{\fd}(R[X],R)\cong R[X]$ is a geometric isomorphism of $R$-modules over $X$, where $\Hom_{\fd}(R[X],R)$ has the structure of an $R$-module over~$X$ defined in~\cite[Proposition $5.20$]{MR123}. 
    The same proposition implies that  $\Hom_{\fd}(\textbf{C}_1(i),R)\cong \textbf{C}_1(i)$ by a geometric isomorphism of $R$-modules over $X$.
    
    Consequently, there exists a map
    \(\delta^\prime \colon R[X]\rightarrow \textbf{C}_1(i)\) induced by~$\delta$ and the corresponding isomorphisms, with \[\delta^\prime\left(\sum\lambda_j x_j\right)=\sum_{d(x_j,x_k)\leq i}(\lambda_j-\lambda_k)[x_j,x_k].\] 
    
    We now assume that~$X$ is non-amenable. By definition there exist $L, \varepsilon>0$ such that for all finite nonempty sets $A\subseteq X$ we have
    \begin{equation*}
        \frac{|\partial_{L}A|}{|A|}\geq \varepsilon.
    \end{equation*}

    Up to changing the choice of projective resolution we may assume that $i\geq 2L$. Choose $\alpha\in R[X] $ and denote by~$A$ its support in~$X$. We recall that the $R$-module structure over $X$ on $\textbf{C}_1(i)$ is given by 
    \[\supp_{\textbf{C}_1(i)}\left(\sum\lambda_{j,k}[x_j,x_k]\right)= \{x_j\in X : \exists x_k \text{ with } d(x_j,x_k)\leq i \text{ and }\lambda_{j,k} \ne 0 \}.\] 
    For $x\in \partial_L A$ there exists $w\in A$ and $z\in X\setminus A$ such that $d(x,w),d(x,z)\leq L$ and therefore $d(w,z)\leq 2L$. It follows that the coefficient of $[w,z]$ in $\delta^\prime(\alpha)$ is nonzero, 
    so that $w\in \supp_{\textbf{C}_1(i)}(\delta^\prime(\alpha))$. 
    
    For each \(x \in \partial_L A\) we thus find an element \([w,z] \in \supp_{\textbf{C}_1(i)}(\delta^\prime(\alpha))\) as above, and we show that this assignment is almost injective. 
    By the assumption that~$X$ is discrete and has bounded geometry there exists a constant $C>0$ such that $|B_{2L+1}(x)|\leq C$. Further, if the above assignment produces the same element in $\supp_{\textbf{C}_1(i)}(\delta^\prime(\alpha))$ for  distinct points $x,y\in \partial_LA$, it must be that $d(x,y)\leq 2L$. We thus conclude that 
    \[|\supp_{\textbf{C}_1(i)}(\delta^\prime(\alpha))|\geq\frac{|\partial_LA|}{C}\geq \frac{\varepsilon|A|}{C}=\frac{\varepsilon|\supp_{R[X]}(\alpha)|}{C}.\qedhere\] 
   
   \end{proof}

    \begin{proposition}\label{Linear iso implies hyperbolic}
        If $X$ is a  discrete quasigeodesic bounded-geometry metric space satisfying a linear homological isoperimetric inequality in dimension~$1$ over $R$, 
        then it is Gromov hyperbolic.
    \end{proposition}

    \begin{proof}
        We provide only a sketch, as the argument follows closely those already present in the literature. By Lemma \ref{iso inequality and different resolutions}, the simplicial $R$-chain complex of $P_\infty(X)$ satisfies a linear homological isoperimetric inequality in dimension $1$. We may readily assume, using for instance \cite[Proposition $2.18$]{MR3816385}, that up to quasiisometry $X$ is the vertex set of a metric graph with edge lengths $1$. 
        By Corollary~\ref{iso inequality and qi}, the simplicial chain complex of $P_\infty(X)$ still satisfies a linear isoperimetric inequality in dimension $1$.

        If $X$ is not Gromov hyperbolic, then the metric graph of which it is the vertex set is not $\delta$-hyperbolic for any $\delta>0$. Under this assumption, the argument of \cite[Proposition~2.13]{MR32141849}, following \cite[Chapter III.H Theorem $2.9$]{MR1744486} and \cite[Proposition $4.2$]{kielak2021isoperimetric}, can be adapted to show that arbitrarily thick triangles in~$P_1(X)$ induce $1$-cycles that cannot be filled linearly inside $P_{\Theta(1)}(X)$ for any fixed $\Theta(1) \geq 1$, yielding the desired conclusion.\end{proof}

    \begin{remark}
        In \cite[Proposition $2.13$]{MR32141849} the filling norm is defined using the absolute value on $\R$ instead of the discrete norm on a general commutative ring; the argument in fact becomes simpler in the latter case. 
    \end{remark}

    \begin{proof}[Proof of Theorem \ref{intro-characterising PD^2}]
        
    We note first that an amenable discrete bounded-geometry metric space cannot be quasiisometric to $\HH^2$, as this would imply that surface groups are amenable. Now, if~$X$ is not amenable, by Proposition~\ref{amenable or linear iso inequality} it satisfies a linear homological isoperimetric inequality in dimension~$1$. In this case, we know from Proposition \ref{Linear iso implies hyperbolic} that~$X$ is Gromov hyperbolic. 
    
    A result attributed to Rips then says that the Rips complex $P_k(X)$ is contractible for any~$k$ sufficiently large (a reference is \cite[Proposition $\mathrm{III}.3.23$]{MR1744486}). By \cite[Lemma $12.2$]{MR123}, the $R$-simplicial chain complex of $P_k(X)$ gives a proper finite-height projective $R$-resolution over $P_k^1(X)$ satisfying our definition of coarse Poincar\'e duality. 
    
    Finally, \cite[Example $12.3$]{MR123} implies that $P_k(X)$ is a coarse Poincar\'e-duality space in the sense of \cite{MR2168506} and \cite{HyperbolicPD^n}, so we conclude by \cite[Corollary $3$]{HyperbolicPD^n} that $P^1_k(X)$, being Gromov hyperbolic, is quasiisometric to $\HH^2$.   \end{proof}

\section{Group theory and the quasikernel}\label{section:group_theory_and_the_quasikernel}
In this section we extend two well-known group theoretic results concerning kernels of homomorphisms to $\Z$ to the setting of quasimorphisms. The first statement which we generalise is that an extension of an amenable group by~$\Z$ is amenable, and the second is that the kernel of a nontrivial homomorphism from a non-elementary hyperbolic group to~$\Z$ is not quasiconvex. These results will be crucial in the proofs of Theorems \ref{thm:manifold} and \ref{introthm:circle_action} respectively.

\begin{proposition}\label{amenable quasikernel implies amenable group}
        Let $G$ be a finitely generated group and $\phi\colon G\to\R$ be a nontrivial homogeneous quasimorphism. If $\mathcal{Q}ker(\phi)$  is amenable as a discrete bounded-geometry metric space, then $G$ is amenable.
    \end{proposition}

    \begin{proof}

   By Lemma \ref{G as a disjoint union of level sets} combined with Lemma \ref{Normal subgroupness}, there is $c\in G$ such that with setting $Q=\qker(\phi)$, $G$ is quasiisometric to the disjoint union 
\[Y\coloneq\bigsqcup_{i\in \Z}Q.c^i.\] Observe that it now suffices to show that~$Y$ is amenable by Remark \ref{basic facts about amenability}.

We fix $x\in B_L(1)\subseteq G$ for some $L>0$ and $i\in\Z$, and consider the possibly empty set of elements $j\in\Z$ with $c^ixc^{-j}\in Q$, where it is clear by Proposition \ref{filtration vs quasimorphism-moved} that $|j-i|$ is bounded above depending only on~$L$. Let the finite set of such $j\in\Z$ be denoted~$J_{x,i}$, and fixing $T>0$ define \[X_{L,T}\coloneq\bigcup_{i\in[-T,T]}\{c^ixc^{-j}: x\in B_L(1), j\in J_{x,i}\}\subseteq Q.\]

    We now show that $Y$ is amenable, that is, for every $\epsilon>0$ and $L>0$, there is a F\o lner set $F\subset Y$ such that\[\frac{|\partial_LF|}{|F|}<\epsilon.\] Fixing $L>0$, set $K_T\coloneq \max_{x\in X_{L,T}}d(1,x)$. Take a F\o lner set $E_{\epsilon,K_T}\subseteq Q$ such that \[\frac{|\partial_{K_T}E_{\epsilon,{K_T}}|}{|E_{\epsilon,{K_T}}|}<\epsilon,\] which exists by amenability of~$Q$. Let \[F_T\coloneq\bigsqcup_{i\in[-T,T]}E_{\epsilon,{K_T}}.c^i.\] We will show that this is the desired F\o lner set for $Y$ with respect to $L,2\epsilon>0$, after possibly choosing $T$ large. 
    
    For $y\in Y$ such that $d(y,F_T)\leq L$ we know that 
    $d(y,zc^i)\leq L$ for some $z\in E_{\epsilon,K_T}$ and $i\in [-T, T]$. We can then write $y=zc^ix$ for some $x\in B_L(1)$. It follows from the fact that $y\in Y$  and the definition of $X_{L,T}$ that $y=(zx')c^j$ for some $j\in \N$ with $zx'\in Q$ and $x'\in X_{L,T}$, where $|j-i|$ is bounded depending only on $L$.

    In particular $y\in (E_{\epsilon,K_T}.x').c^j$ as $z\in E_{\epsilon, K_T}$, so $y\in N_{K_T}(E_{\epsilon,{K_T}}).c^j\cap Q.c^j$ by definition of $K_T$ as $x'\in X_{L,T}$. By the same argument, up to choosing $z\in Q\setminus E_{\epsilon, K_T}$ instead of $z\in E_{\epsilon, K_T}$, if $y\in F_T$ and $d(y,Y\setminus F_{T})\leq L$, then $y\in N_{K_T}(Q\setminus E_{\epsilon,{K_T}}).c^j\cap Q.c^j$. 
    We conclude that \[\partial_LF_T\subset \bigsqcup_{j\in [-T',T']}(\partial_{K_T}E_{\epsilon,{K_T}}).c^j , \] where $T'- T\geq 0$ depends only on $L$ and not $T$ as the same holds for all $|j-i|$.
    Now, 
    \begin{align*}
        \frac{|\partial_LF_{T}|}{|F_T|}&\leq \frac{\sum_{j\in[-T,T]}|\partial_{{K_T}}E_{\epsilon,{K_T}}|}{\sum_{j\in[-T,T]}|E_{\epsilon,{K_T}}|}+\frac{\sum_{j\in[-T',T']\setminus [-T,T]} |\partial_{K_T}E_{\epsilon,{K_T}}|}{\sum_{j\in [-T,T]}|E_{\epsilon,{K_T}}|}\\
        &=\left(1+\frac{T^\prime-T}{T}\right)  \frac{|\partial_{K_T}E_{\epsilon,{K_T}}|}{|E_{\epsilon,{K_T}}|}
    \end{align*}
     which is strictly bounded above by $2\epsilon$ when $T$ is sufficiently large as $E_{\epsilon,K_T}$ is a F\o lner set with respect to $\epsilon,K_T>0$.\end{proof}

\begin{corollary}\label{amenable quasikernel implies a homomorphism}
    If $G$ is a finitely generated group and $\phi\colon G\to\R$ is a nontrivial homogeneous quasimorphism such that $\qker(\phi)$ is amenable, then $\phi\colon G\to\R$ is a group homomorphism.
\end{corollary}

\begin{proof}
    By Proposition \ref{amenable quasikernel implies amenable group}, our assumptions imply that $G$ must be an amenable group. But now it is a consequence of a result of Johnson, Trauber and Gromov, see for instance \cite[Theorem $2.47$]{calegari2009scl}, that every homogeneous quasimorphism on an amenable group is a homomorphism, from which we conclude.
\end{proof}

We now move on to proving non-quasiconvexity of the quasikernel in the setting of hyperbolic groups. In order to do this efficiently, we invoke a rigidity result coming from the field of \textit{geometric approximate group theory}, as developed by Cordes--Hartnick--Toni\'c in \cite{cordes2020foundations}. The following is the key to our argument.

\begin{lemma}\label{lem:finite-hausdorff-distance-of-lambda_infty-in-quasiconvex-case_new}
    Let $G$ be a hyperbolic group and $\phi\colon G\to\R$ be a nontrivial homogeneous quasimorphism. If~$\qker(\phi)$ is quasiconvex, then it is at finite Hausdorff distance from the subgroup $\langle\qker(\phi)\rangle\leqslant G$.
\end{lemma}
\begin{proof}
    We first point out that~$(\qker(\phi), \langle\qker(\phi)\rangle)$ is an approximate group in the sense of~\cite{cordes2020foundations}, see for instance~\cite[Proposition $2.8$]{QBNS}. Moreover, quasiconvexity readily implies being geometrically finitely-generated \cite[Definition~3.32]{cordes2020foundations}. 
    Then, by \cite[Corollary~6.36]{cordes2020foundations}, there is a finite set $F\subset G$ such that \[\qker(\phi)\subset \langle\qker(\phi)\rangle\subset \qker(\phi).F.\] We conclude that
    \begin{align*}
        d_{\Hd}(\qker(\phi), \langle\qker(\phi)\rangle)&\leq d_{\Hd}(\qker(\phi), \qker(\phi).F)
    \end{align*}
    which is finite by Proposition~\ref{filtration vs quasimorphism-moved} and finiteness of~$F$. 
\end{proof}

\begin{proposition}\label{quasiconvex}
    If $G$ is a non-elementary hyperbolic group and $\phi\colon G\to\R$ is a nontrivial homogeneous quasimorphism, then $\qker(\phi)$ is not quasiconvex in $G$.
\end{proposition}

\begin{proof}
    Let us assume that $\qker(\phi)$ is quasiconvex, and denote $H\coloneq \langle\qker(\phi)\rangle$. By Lemma~\ref{lem:finite-hausdorff-distance-of-lambda_infty-in-quasiconvex-case_new}, we know that $d_{\Hd}(\qker(\phi), H)$ is finite. Using this, it follows from
   Lemma~\ref{Normal subgroupness} and Corollary \ref{cor:filtration_vs_quasimorphism-quantitative}  that $d_{\Hd}(g^{-1}Hg, H)$ is uniformly bounded from above for every $g\in G$. 
   
   Observe now that by left-invariance of the metric, $d_{\Hd}(H, Hg)\leq |g|$, so that the triangle inequality yields
    \begin{equation*}
        d_\Hd(gH, H)\leq d_{\Hd}(gH,Hg) + |g|<\infty.
    \end{equation*}
  This implies that if $\Lambda(H)\subseteq\partial G$ is the limit set of $H$, $g.\Lambda(H)\subseteq \Lambda(H)$ for every $g\in G$. By minimality of the action of a non-elementary hyperbolic group on its boundary we obtain that either $\Lambda(H)=\emptyset$ or $\Lambda(H)=\partial G$. The first case is in contradiction with the fact that $H$ is infinite and quasiconvex, and the second is in contradiction with unboundedness of~$\phi$ by Corollary~\ref{cor:filtration_vs_quasimorphism-quantitative} and quasiconvexity. We conclude that $\qker(\phi)$ cannot be quasiconvex.
\end{proof}

    \section{Quasimorphisms on $\PD^3$ groups}\label{section:quasimorphisms_on_pd3_groups}

In this section we combine the results of the previous sections to obtain the characterisation of $\qker(\phi)$ stated in Theorem \ref{thm:manifold}. We in fact obtain this part of the result for Poincar\'e-duality groups over any~PID. We first recall a homological result that will allow us to reduce our finiteness hypotheses on $\qker(\phi)$. The proof is identical to that of~\cite[Corollary $3.2$]{PD^nfibring}, where Novikov rings of homomorphisms to $\Z$ are considered, and is likewise implicit in \cite{MR2282258}.

    \begin{lemma}\label{Novikov homology and cohomology}
    Let~$R$ be a commutative ring,~$G$ be a $\PD^3(R)$ group and $\phi\colon G\to\R$ a nontrivial homogeneous quasimorphism such that\[\Homology_1(G; \widehat{R[G]}^{\pm\phi})=0.\] Then  for all $j\in \{0,1,2,3\}$, we have
    \[\Homology_j(G;\widehat{R[G]}^{\pm\phi})=\Homology ^j(G;\widehat{R[G]}^{\pm\phi})=0.\] 
\end{lemma}

We now prove the key dichotomy concerning the coarse geometry of the quasikernel in the setting of $3$-dimensional Poincar\'e-duality groups.

\begin{proposition}\label{thm: approx_kernel_qi_to_H2}
   Let $G$ be a $\PD^3(R)$ group for $R$ a PID and $\phi\colon G\to\R$ be a nontrivial homogeneous quasimorphism. If $\qker(\phi)$ is coarsely connected, the following hold.

    \begin{enumerate}
        \item If $G$ is amenable, then $\phi$ is a homomorphism and $\ker(\phi)$ is virtually $\Z^2$. Moreover, in the case that $R=\Z$, we have that $\ker(\phi)\cong\Z\rtimes \Z$.
        \item If $G$ is non-amenable, then $\qker(\phi)$ is coarsely equivalent to $\HH^2$.
    \end{enumerate}

\end{proposition}

\begin{proof}

As $G$ is a $\PD^3(R)$ group, it has cohomology at infinity \[\Homology^i(G;R[G])=
\begin{cases}
    R&\text{  if } i=3,\\
    0  &\text{  if } i\in\{0,1,2\}.
\end{cases}
\]The quasikernel $\qker(\phi)$ is coarsely connected by assumption, so by \cite[Theorem $1.2$]{QBNS}, the first Novikov homology of~$G$ vanishes,
\[\Homology_1(G;\widehat{R[G]}^{\pm\phi})=0.\] 
Note that the result of Heuer--Kielak is stated for $R=\Z$, but works for any commutative ring. From this it follows by  Lemma \ref{Novikov homology and cohomology} that\[\Homology_i(G;\widehat{R[G]}^{\pm\phi})=\Homology^i(G;\widehat{R[G]}^{\pm\phi})=0\] for all $i\in\{0,1,2,3\}$. 
The hypotheses of Theorem \ref{Novikov homology new} are satisfied, so $\qker(\phi)$ is coarsely uniformly $1$-acyclic by Lemma \ref{Uniform acyclicity vs acyclicity} as it is coarsely homogeneous by Lemma \ref{coarselyhomogeneous}.
From Proposition \ref{cohomology at infinity} we obtain that that 
\[\Homology_{\coarse}^{i-1}(\qker(\phi),R)\cong \Homology^i(G;R[G])\quad\text{ for } i\in\{1,2,3\}.\]

The quasikernel now satisfies the assumptions of Theorem \ref{charactersing PD^n}, and is hence is a coarse $\PD^2(R)$ space.

In the case that $G$ is amenable, the proof of Corollary \ref{amenable quasikernel implies a homomorphism} says that $\phi\colon G\to\R$ must be a group homomorphism. It then follows from \cite[Theorem $1.19$]{MR1943724} that $\ker(\phi)$ is a $\PD^2(R)$ group, from which an application of \cite[Theorem $1.2$]{Quasiplanes} yields that it is in fact virtually $\Z^2$. When $R=\Z$,  \cite[Proposition $4.6$]{kielak2021isoperimetric} implies that $\ker(\phi)\cong\Z\rtimes\Z$.

In the case that~$G$ is non-amenable, Proposition \ref{amenable quasikernel implies amenable group} implies that~$\qker(\phi)$ is non-amenable. As $\qker(\phi)$ is coarsely connected as a subspace of $G$, it is coarsely equivalent to the Rips graph~$P_r^1(\qker(\phi))$ for some $r\geq 0$ by Proposition~\ref{X and Rips graph}. Observe that~$P_r^1(\qker(\phi))$ is a geodesic metric space, and that it is also non-amenable by Remark~\ref{basic facts about amenability}. 
By~\cite[Lemma~12.2]{MR123}, the property of being a coarse $\PD^2(R)$ space is invariant under coarse equivalence, so  Theorem~\ref{intro-characterising PD^2} implies that the vertex set of~$P_r^1(\qker(\phi))$ is quasiisometric to~$\HH^2$, and thus $\qker(\phi)$ is coarsely equivalent to~$\HH^2$.\end{proof}

\section{A faithful action on~$S^1$ for hyperbolic $\PD^3$  groups}\label{section:circle_action}

Let $R$ be a PID.
In preparation of the proofs of Theorems~\ref{thm:manifold} and~\ref{introthm:circle_action}, for a $\PD^3(R)$ group~$G$ with homogeneous quasimorphism~$\phi\colon G\to \R$ whose quasikernel is coarsely connected and coarsely equivalent to~$\HH^2$, we associate to every~$g\in G$ a quasiisometry~$f_g$ of the hyperbolic plane. Furthermore, our assignment commutes with group multiplication up to bounded distance, so that the induced boundary maps define an action of~$G$ on~$S^1$, as introduced in \cite{calegari2010boundedcochains3manifolds},  and in the case that~$G$ is hyperbolic, we prove that this action is nontrivial and faithful.

By Proposition~\ref{thm: approx_kernel_qi_to_H2}, there is a coarse equivalence $\theta_1\colon \qker(\phi)\to \HH^2$.
Since left-translation is an isometry of~\(G\), the subspace metrics on all $G$-translates of~\(\qker(\phi)\) are fully determined by the metric on~\(\qker(\phi)\).
This allows us to obtain coarse equivalences \(\theta_g\colon g.\qker(\phi)\to\mathbb{H}^2\) for each~\(g\in G\) by setting \(\theta_g := \theta_1 \circ L_{g^{-1}}\), where $L_{g^{-1}}$ denotes left-multiplication by~$g^{-1}$.
Choose a quasiinverse~\(\theta^{-1}_g\) and let~\(\pi_g\colon g.\qker(\phi)\to \qker(\phi)\) be a closest-point projection. We define
\begin{equation}\label{eq:define_f_g_new}
    \begin{aligned}
        f_g\colon \mathbb{H}^2&\to \mathbb{H}^2, \\
        x&\mapsto \theta_1\circ \pi_g\circ \theta^{-1}_{g}(x).
    \end{aligned}
\end{equation}

\begin{lemma}
    For every~\(g\in G\) the map~\(f_g\) defined in~\eqref{eq:define_f_g_new} is a quasiisometry.
\end{lemma}
\begin{proof}
    First recall that by Lemma~\ref{Projecting is a coarse equivalence} and Proposition~\ref{filtration vs quasimorphism-moved} the maps~$\pi_g$ are quasiisometries.
As a composition of coarse equivalences, \(f_g\) is a coarse equivalence, and since~\(\HH^2\) is a geodesic metric space, \(f_g\) is a quasiisometry.
\end{proof}

By classical results in hyperbolic geometry, every quasiisometry \(f_g\) defines a quasisymmetric homeomorphism~\(\varphi_g\) of~\(S^1\). This yields a map 
\begin{align*}
    \Phi\colon G&\to \mathrm{Homeo}_{\mathrm{qs}}(S^1)\\
    g&\mapsto \varphi_g, 
\end{align*}
and to prove that~$\Phi$ is a group homomorphism, we need to show that 
\(d_\infty(f_{g}\circ f_{h}, f_{gh})\) is finite, which will be achieved by choosing the closest-point projections in~\eqref{eq:define_f_g_new} suitably.
As any two closest-point projections between spaces at bounded Hausdorff distance are at bounded distance from each other, the map~$\Phi$ is independent of this choice.
\begin{lemma}\label{lem:phi_for_circle_action_is_a_group_homomorphism}
    The map~$\Phi\colon G\to \mathrm{Homeo}_{\mathrm{qs}}(S^1)$ is a group homomorphism.
\end{lemma}
\begin{proof}
Fix $g, h\in G$ and choose nearest point projections 
\(\pi_g, \pi_h\) 
and~\(\pi_{gh}\) from the respective translates of~$\qker(\phi)$ onto~$\qker(\phi)$.

Further, we set
\[\pi_{h,g} (x) := L_g\circ \pi_h\circ L_{\inv{g}},\] 
 which 
is a nearest point projection from $gh.\qker(\phi)$ to $g.\qker(\phi)$ by left-invariance. 
Following~\eqref{eq:define_f_g_new}, we define quasiisometries $f_g, f_h$ and $f_{gh}$ of $\HH^2$ using the projections~$\pi_g, \pi_h$ and~$\pi_{gh}$.

We then have
\begin{align*}
    f_{g}\circ f_h 
    &=\theta_1\circ \pi_g\circ L_g\circ \pi_h\circ L_h\circ \inv{\theta_1},
\end{align*}
and
\begin{align*}
    f_{gh} 
    &= \theta_1\circ \pi_{gh}\circ L_{gh}\circ \inv{\theta}_1.
\end{align*}

Proposition~\ref{filtration vs quasimorphism-moved} implies that the translates of~$\qker(\phi)$ by $g, h$ and $gh$ are at pairwise bounded Hausdorff distance from each other.
We invoke Lemma~\ref{lem:composition_of_nearest_point_projections_is_close_to_nearest_point_projection} to conclude that the composition of nearest point projections \(\pi_g\circ \pi_{h,g}\) is bounded distance from~\(\pi_{gh}\). 
It follows that
\begin{equation*}
    d_\infty(\pi_{gh}\circ L_{gh}, \pi_g\circ L_g\circ \pi_h \circ L_{\inv{g}}\circ L_g\circ L_h)<\infty,
\end{equation*}
so that $f_g\circ f_h$ is at bounded distance from~$f_{gh}$.
\end{proof}

In the case that the group $G$ is hyperbolic, we will show that $\Phi$ is injective, and in doing so prove the following theorem.

\numberedtheorem{Theorem~\ref{introthm:circle_action}}{}{ Let $G$ be a hyperbolic $\PD^3$ group and $\phi\colon G\to\R$ be a nontrivial homogeneous quasimorphism. If $\qker(\phi)$ is coarsely connected, then $G$ acts faithfully on $S^1$ by quasisymmetric homeomorphisms.}

The remainder of this section is dedicated to the proof of this result, so that we fix the following notation.
\begin{convention}\label{convenctions_for_faithfulness}
From now on, in this section we let~$G$ always be a hyperbolic $\PD^3$ group with a left-invariant word metric~$d$ and with a nontrivial homogeneous quasimorphism~$\phi$ with coarsely connected quasikernel.
\end{convention}

We first observe a basic property of self-quasiisometries of $\HH^2$.

\begin{lemma}\label{lem:quasiisometries_with_identical_boundary_maps_are_uniformly_close}
    If $f, g\colon \HH^2\to \HH^2$ are 
    quasiisometries such that the induced boundary maps $\partial f, \partial g\colon S^1\to S^1$ coincide, then there exists a constant  $C>0$  
    depending only on the quasiisometry constants
    with 
    \begin{equation*}
        d_{\HH^2}(f(x),g(x))\leq C \qquad \text{ for all }x\in \HH^2.
    \end{equation*}
\end{lemma}
\begin{proof}
    We compose $f, g$ with the quasiinverse of~$g$, so that, up to a change of the quasiisometry constants, we may assume that~$g$ is the identity map. Let $x\in X$ and $\gamma$ be a geodesic ray containing $x$ that is orthogonal to the geodesic segment $[x,f(x)]$. 
    By the Morse lemma~\cite[Theorem 11.72 and Lemma 11.105]{MR3753580}, there exists a constant $C$ depending only on the quasiisometry constants such that~$f(\gamma)$ is at Hausdorff distance at most~$C$ from~$\gamma$. It follows that 
    \begin{equation*}
        d_{\HH^2}(x,f(x))=d(\gamma, f(x))\leq C.\qedhere
    \end{equation*}
\end{proof}

 We now exhibit faithfulness of the circle action, first establishing a criterion for a group element~$g$ to be in the kernel of~$\Phi$.

\begin{lemma}\label{lemma:characterise_trivial_action_on_circle_new}
    \(\Phi(g)=\mathrm{id}_{S^1}\) if and only if \(d_\infty(\pi_g \circ L_g, \id_{\qker(\phi)})<\infty\).
\end{lemma}
\begin{proof}
    One implication is clear, so that it remains to show that $\Phi(g)=\id_{S^1}$ implies that $\pi_g\circ L_g$ is at bounded distance from the identity.

    Let \(g\in \ker(\Phi)\), then by Lemma~\ref{lem:quasiisometries_with_identical_boundary_maps_are_uniformly_close} there exists \(C\geq 0\) with \(d_{\HH^2}(f_g(x), x)\leq C\) for all \(x\in \HH^2\). Then 
    \begin{align*}
        d_{\HH^2}(f_g(x), x)
        =d_{\HH^2}(\theta_1\circ\pi_g\circ L_g\circ \theta_1^{-1}(x),x)\leq C.
    \end{align*}
    Writing \(y=\inv{\theta_1}(x)\) choosing \(\tilde{C}\geq 0\) such that \(d_\infty(\theta_1^{-1}\circ \theta_1, \id_{\qker(\phi)})\leq \tilde{C}\), we reformulate the above condition as  
    \begin{equation*}
        d(\pi_g\circ L_g(y),y)\leq \rho_+(C)+\tilde{C}
    \end{equation*}
    for all \(y\in \qker(\phi)\), where $\rho_+$ is the upper control function of the coarse equivalence $\HH^2\to \qker(\phi)$. 
\end{proof}

We recall here that as $\qker(\phi)$ is coarsely connected as a subspace of $G$, it is coarsely equivalent to the Rips graph $P_r^1(\qker(\phi))$ for some $r\geq 0$ by Proposition~\ref{X and Rips graph}. Under our assumptions on~$G$ and~$\phi$, the Rips graph is then quasiisometric to~$\HH^2$.  We now complete the proof of Theorem~\ref{introthm:circle_action}.

\begin{proposition}\label{prop:injectivity_of_Phi}
    Under the assumptions of Convention~\ref{convenctions_for_faithfulness}, the action of~$G$ on~$S^1$ defined by~$\Phi$ is faithful.
\end{proposition}

\begin{proof}
    Let $X=\qker(\phi)$ and consider first the case that $\phi(g)=0$. By homogeneity of $\phi$, we then have $\phi(g^n)=0$ for every $n\in\mathbb{Z}$, so that $g^n \in X$ for every $n\in\mathbb{Z}$. Since~$G$ as a $\PD^3$  group is torsion free, the map $n\mapsto g^n$ describes a bi-infinite quasigeodesic in $X$; this is a standard fact for hyperbolic groups, see for instance \cite[Theorem III.$\Gamma$.4.10]{MR1744486}. The metric on $P_r^1(X)$ satisfies 
    \[d_{P_r^1(X)}(x,y)\leq \frac{1}{r}d(x,y)\] 
    for all $x,y\in G$,
    where~$r$ is as described above.
    It follows that $n\mapsto g^n$ is also a quasigeodesic in $P_r^1(X)$, which  we denote by~$\gamma$.

We claim that there is a quasi-action of $\langle g\rangle$ on $\HH^2$ by the quasiisometries $f_{g^n}$ for $n\in\Z$.  Indeed, the quasiisometry constants can be chosen uniformly for all~$n$, as they only depend on those of $\pi_{g^n}$, which are uniformly bounded due to homogeneity of~$\phi$ and Corollary~\ref{cor:filtration_vs_quasimorphism-quantitative}.    The image of $\gamma$ under the quasiisometry $P_r^1(X)\to\HH^2$ further defines a coarse axis for the quasi-action.

By~\cite{hinkkanen1985uniformly}, this quasi-action of $\langle g\rangle$ is quasiconjugate to an action of $\langle g\rangle$ on $\HH^2$ by isometries $\tilde{f}_{g^n}$, and the image of $\gamma$ serves as a coarse axis. 
It is clear that $\tilde{f}_g$ is a hyperbolic isometry of $\mathbb{H}^2$, hence fixes exactly two points on the boundary.
Since quasiconjugate actions on $\mathbb{H}^2$ induce conjugate actions on the boundary, it follows that $f_g$ also fixes exactly two boundary points. This completes the first case.

We now consider the case $\phi(g)\ne 0$. 
We will show that in this case the assumption $g\in \ker\Phi$ implies that $\qker(\phi)$ is a quasiconvex subspace of $G$, contradicting Proposition \ref{quasiconvex}. Suppose therefore that $\qker(\phi)$ is not quasiconvex. 
Since $\phi(g)\neq 0$, analogously to Lemma~\ref{G as a disjoint union of level sets}~(ii) one obtains the existence of a constant $K\ge 0$ such that\[
d_{\mathrm{Hd}} \Bigl(G, \bigcup_{n\in\mathbb{Z}} g^n. \qker(\phi)\Bigr)=K<\infty.
\]

Let $n\in\mathbb{Z}$ be minimal in absolute value such that $[x,y]$ does not intersect $g^m. \qker(\phi)$ for any $|m|>|n|$, such an~$n$ exists by Corollary~\ref{cor:filtration_vs_quasimorphism-quantitative}. 
Choose $z\in g^n. \qker(\phi)$ with $d(z,[x,y])\le K$, and note that there exists  $r=r(n)$ such that\[
d(z,\qker(\phi))\ge r-K.
\] By non-quasiconvexity of $\qker(\phi)$, we may choose the configuration so that $r$ is arbitrarily large, where Corollary~\ref{cor:filtration_vs_quasimorphism-quantitative} forces $|n|\to\infty$ as $r\to\infty$. We also note that there exists a point $z'\in [x,y]$ with $d(z,z')\le K$,
and the same corollary yields a constant $C_0=C_0(K)$ such that \[
|\phi(z)-\phi(z')|\le C_0.
\]

The key observation is that the map $k\mapsto g^k.x$ defines a quasigeodesic whose constants depend on $g$ and $\phi$, but crucially not on $x$. This follows from the fact that  $L_{g}$ is at finite distance from a nearest-point  projection from $ \qker(\phi)$ onto its $g$-translate by Lemma~\ref{lemma:characterise_trivial_action_on_circle_new}.  

By the Morse lemma,
there exists $C_1=C_1(C_0,g,\phi)$ such that this quasigeodesic lies at Hausdorff distance at most $\tau$ from any geodesic segment $[x,g^n.x]$.
It then follows that $[x,g^n.x]$ does not intersect $g^m\cdot \qker(\phi)$ for $|m|\ge |n|+C_2$, where $C_2=C_2(C_1)$.
Corollary~\ref{cor:filtration_vs_quasimorphism-quantitative} then implies that there exists $C_3=C_3(C_2)$ such that for all $p\in [x,g^n.x]$ we have
\begin{equation*}\label{eq:estimate_phi}
    |\phi(p)|\le |\phi(g^n.x)|+C_3\le n|\phi(g)|+3D(\phi)+C_3,
\end{equation*}
and the same estimate holds for all $p\in [y,g^n.y]$. 

Consider the geodesic quadrilateral with vertices $g^n.x, x, y, g^n.y$. Observe that a geodesic with endpoints $g^n.x,g^n.y$ is given by $g^n.[x,y]$. By hyperbolicity, each side is contained in the $2\delta$-neighbourhood
of the union of the other three sides. In particular, there exists a point 
\[
q\in [g^n.x,x]\cup [x,y]\cup [y,g^n.y]
\]
such that $d(q,g^n.z')\le 2\delta$. If $q\in[x,y]$, Corollary~\ref{cor:filtration_vs_quasimorphism-quantitative} implies that there exists $C_4=C_4(\phi, \delta)$ with $|\phi(q)|\leq |n||\phi(g)|+C_4$, otherwise the bound~\eqref{eq:estimate_phi} applies. It follows that \begin{equation*}
    |\phi(q)|\leq |n||\phi(g)|+3D(\phi)+C_3+C_4.
\end{equation*} 
Writing $z=g^n.u$ for some $u\in \qker(\phi)$, we obtain
\begin{align*}
|\phi(g^n.z')|
&\ge |\phi(g^n.z)|-2D(\phi)-C_0\\
&\ge |\phi(g^{2n})|-|\phi(u)|-3D(\phi)-C_0\\
&\ge 2|n|\,|\phi(g)|-5D(\phi)-C_0.
\end{align*} 
Comparing with the upper bound on~$|\phi(q)|$ and letting~$|n|$ tend to infinity yields a contradiction with Corollary~\ref{cor:filtration_vs_quasimorphism-quantitative}.
\end{proof}

\section{Constructing Riemannian manifolds quasiisometric to $\PD^3$  groups}\label{section:constructing_Riemannian_manifolds}

In this final section we complete the proof of Theorem \ref{thm:manifold}. Given Proposition \ref{thm: approx_kernel_qi_to_H2}, it remains to prove the following.

\begin{theorem}\label{Riemannian manifold}
    Let $G$ be a $\PD^3(R)$ group for $R$ a PID and $\phi:G\to\R$ be a nontrivial homogeneous quasimorphism. If $\qker(\phi)$ is coarsely equivalent to $\HH^2$, then $G$ is quasiisometric to a complete Riemannian manifold $(\R^3,g)$. 
\end{theorem}

Our approach is to show that the geometry of such a $\PD^3(R)$ group can be characterised by a certain self-quasiisometry of $\HH^2$, which should be thought as the analogue of the monodromy of a closed fibred $3$-manifold in the case that the fibre is a hyperbolic surface. This approach takes inspiration from unpublished work of Danny Calegari \cite{calegari2010boundedcochains3manifolds}.

\begin{remark}
    
The analogy with fibred $3$-manifolds can be made precise in the following sense. If the manifold $M$ is equipped with a Riemannian metric, and the monodromy is lifted to a homeomorphism of $\R^2$, then, when $\R^2$ is equipped with the Riemannian metric induced by that of the universal cover $\tilde{M}$, it is quasiisometric to $\HH^2$, and the lift of the monodromy becomes a self-quasiisometry. 

\end{remark}

In the spirit of this, we construct Riemannian metrics on~$\R^3$ in a way that can best be described as building \textit{iterated mapping cylinders} of self-diffeomorphisms of~$\R^2$. If we choose the self-diffeomorphism to be an appropriate smoothening of the self-quasiisometry of $\HH^2$ characterising the geometry of a $\PD^3(R)$ group as above, such a manifold will naturally yield a quasiisometric model for the group.

For the remainder of this section, let~$G$ be a $\PD^3(R)$  group with a nontrivial homogeneous quasimorphism $\phi\colon G\to \R$ such that~$\qker(\phi)$ is coarsely equivalent to~$\HH^2$.
Recall from Lemma~\ref{G as a disjoint union of level sets} that there exists an element $c\in G$ such that the disjoint union of the left $c^i$-translates of~$\qker(\phi)$ is coarsely dense in~$G$. This observation serves as inspiration for the construction of our metric on~$\R^3$. The topological planes $\R^2\times\{i\}$ for $i\in \Z$ will serve as analogues of the $c^i$-translates of the quasikernel, and the change of metric between them will be determined by the quasiisometry induced by left multiplication and closest-point projection.

More precisely, let $\pi_c\colon c.\qker(\phi)\to \qker(\phi)$ be a closest-point projection and let~$f_c$ be the corresponding quasiisometry of~$\HH^2$ associated to the group element~$c$ defined by~\eqref{eq:define_f_g_new} in Section~\ref{section:circle_action}.
It is necessary to smoothen~$f_c$ into a bi-Lipschitz diffeomorphism of~$\HH^2$ which we denote~$f$; that such a map exists at bounded distance from~$f_c$ can be seen for example using the work of Douady--Earle, who prove that every quasisymmetric circle homeomorphism is the extension to the visual boundary of such a diffeomorphism \cite{MR857678}. Extending~$f_c$ to the boundary and defining~$f$ to be the thus obtained bi-Lipschitz diffeomorphism with identical boundary extension together with Lemma~\ref{lem:quasiisometries_with_identical_boundary_maps_are_uniformly_close} yields the result. The map $f\colon\HH^2\to\HH^2$ will be the `monodromy' of $G$ alluded to previously.

We will construct a Riemannian metric~$g$ on~$\R^3$ such that $(\R^3, g)$ has the properties described in Theorem~\ref{Riemannian manifold}.
Fix a global parametrisation of~$\HH^2$ by~$\R^2$; by abuse of notation we will often conflate the two. In particular, we will treat~$f$ also as a diffeomorphism of~$\R^2$.   
Let~$g_{\HH^2}$ denote the pullback metric on~$\R^2$ under this parametrisation 
and set
\[g_i\coloneq (f^i)^\ast g_{\HH^2}\] for $i\in\Z$.
 The geometry of~$G$ will be reflected in our manifold by choosing the metric on the leaves $\R^2\times \{i\}$ to be equal to $g_i+dt^2$, where $t$ is the third direction, and interpolating smoothly between these.

Consider therefore the open cover of~$\R$ given by
\begin{equation}\label{eq:open_cover_for_partition_of_unity}
    \mathcal{U}=\{ (i-2/3, i+2/3)\}_{i\in \Z}
\end{equation}
and let $\{\psi_i\}_{i\in \Z}$ be a partition of unity subordinate to~$\mathcal{U}$.
For every $t\in\R$, we define a metric on $\R^2$ given by \[(g_t)_x\coloneq\underset{i\in\Z}{\sum}\psi_i(t).(g_i)_x\] for every $x\in\R^2$, which will be the induced metric on the leaves $\R^2\times\{t\}\subseteq \R^3$. For $(x,t)\in\R^2\times\R$, we set
\begin{equation}\label{eq:definition_metric_iterated_mapping_cylinder}
    g_{(x,t)}=(g_t)_x+dt^2,
\end{equation}
which defines the desired Riemannian metric on~$\R^3$.

We now prove some key properties of the manifold $(\R^3,g)$. For a smooth curve~$\gamma$ on~$(\R^3, g)$, denote by~$l_g(\gamma)$ its length and by~$d_g$ the induced length metric.

\begin{lemma}\label{vertical geodesics in manifold}
    The Riemannian manifold $(\R^3,g)$ has the following properties.
    \begin{enumerate}
        \item For every $p\in \R^3$ there exist $i\in\Z$ and $q\in \R^2\times\{i\}$ with $d_g(p,q)\leq 1$.
        \item For all $x, y\in\R^2$ and $i,j\in\Z$, we have that \begin{align*}
            &d_g((x,i),(x,i+1))=1,\\
            &d_g((x,i),(y,j))\geq |j-i|.
        \end{align*}
    \end{enumerate}
\end{lemma}

\begin{proof} 
 Write $p=(x,i+r)\in \R^2\times\R$ for some $i\in\Z$ and $r\in[0,1)$.  Consider the smooth curve $\gamma\colon[0,1-r]\to \R^3$ from $(x,i+r)$ to $(x,i+1)$ given by $s\mapsto (x,r+s)$. 
In coordinates $(a,b,t)$ on $\R^3$, we have $\gamma'(s)=dt$, and so \[l_{g}(\gamma)=\int_0^{1-r}\sqrt{g_{\gamma(s)}(dt,dt)}ds.\] 
 
We have $g_{(a,b,t)}(dt,dt)=1$, so that $l_{g}(\gamma)=1-r\leq 1$, proving $(1)$.

For the proof of~$(2)$, we first observe that it is immediate from the above that there is a smooth curve $\gamma$ of length $1$ with endpoints $(x,i)$ and $(x,i-1)$.  
Suppose there exists another smooth curve $\eta\colon I\to \R^3$ between these two points of length $l_{\tilde{g}}(\eta)<1$. 
We may write $\eta(s)=(x(s),t(s))$ for smooth functions $x\colon I\to\R^2$ and $t\colon I\to \R$, and compute that \[l_g(\eta)\geq\int_0^{1}|t'(s)|ds,\] which is bounded from below by $1$ using the fundamental theorem of calculus. The latter part of $(2)$ follows from an analogous argument.\end{proof}

In the following proposition we consider the induced Riemannian metric on 
\begin{equation*}
    N_r(\R^2\times\{i\}))=\R^2\times[i-r,i+r]\subseteq \R^3
\end{equation*} 
for $i\in\Z$ and $r\in\N$. We will denote this metric, which is simply the restriction of~$g$ to $\R^2\times[i-r,i+r]$, by~$g_{i,r}$.

 \begin{lemma}\label{Leaves are properly embedded}
 For every $i\in \Z$ and $r\in \N$, the inclusion 
 \[(\R^2\times\{i\},g_i)\to(\R^2\times [i-r,i+r],g_{i,r})\] 
 induces a quasiisometric embedding whose constants depend only on~$r$. In particular, for every~$i\in\Z$ the inclusion
 \[(\R^2\times\{i\},g_i)\to (\R^3,g)\] 
 induces a coarse embedding with control functions independent of~$i$.

 \end{lemma}
\begin{proof}
 From Lemma \ref{vertical geodesics in manifold} we know that for each $r\in \N$ and $i\in \Z$ we have \[N_r(\R^2\times\{i\})=\R^2\times[i-r,i+r].\] 
  Observe that for any $t\in\R$ the induced metric on the leaf $\R^2\times\{t\}$ can be written as 
\[(g_{t})_x=\psi_j(t)(g_j)_{x}+\psi_{j+1}(t)(g_{j+1})_{x}.\]

   For any $j\in\Z$, denote by~$\tau_j$ the projection onto the integer leaf~$\R^2\times\{j\}$ given by
   \begin{equation*}
       \tau_j(x,t)=(x,j) \quad \text{ for all }(x,t)\in\R^3.
   \end{equation*}
    Recall that~$f$ is a bi-Lipschitz diffeomorphism, whose bi-Lipschitz constant we denote by~$L$.
    Let $k\in\Z$ and $v\in T_{(x,j)}(\R^2\times\{j\})$, and set $w\coloneq D_{(x,j)}\tau_k(v)$.
   By definition of the metric on integer leaves and the fact that $f$ is $L$-bi-Lipschitz, we have that
    \begin{equation}\label{eq:biLipschitz_comparison_Riemannian_metrics_on_slices}
        L^{-|j-k|}(g_j)_x(w, w) \leq (g_k)_x(v, v)\leq L^{|j-k|} (g_j)_x(w, w) 
    \end{equation} for any $j\in\Z$.

    Let now $u\in T_{(x,t)}  (\R^2\times\{t\}) $ for any $t\in\R$ and denote  $u_0\coloneq D_{(x,j)}\tau_j(v)$ and $u_1\coloneq D_{(x,j+1)}\tau_{j+1}(v)$. We then have that
    \begin{equation}\label{Min-max inequality}
        \begin{split}
            \min\{(g_j)_x(u_0,u_0), (g_{j+1})_x(u_0,u_0)\} & \leq g_{(x,t)}(u,u)\\
            &\leq \max \{(g_j)_x(u_1,u_1), (g_{j+1})_x(u_1,u_1)\}
        \end{split} 
    \end{equation} 
    as~$g_t$ interpolates between $g_j$ and $g_{j+1}$ as above.

    Consider a pair of points $(x,i)$, $(y,i)$ in $\R^2\times \{i\}$, let $\varepsilon>0$ and choose a smooth curve~$\gamma$ in $\R^2\times[i-r,i+r]$ between them with $l_{g_{i,r}}(\gamma)=d_{g_{i,r}}((x,i),(y,i))+\varepsilon$.
    It is clear that $\gamma_i\coloneq \tau_i\circ\gamma$ is a smooth curve in $\R^2\times\{i\}$ from $(x,i)$ to $(y,i)$, and~\eqref{eq:biLipschitz_comparison_Riemannian_metrics_on_slices} combined with \eqref{Min-max inequality} implies that
    \begin{equation*}
        l_{g_i}(\gamma_i)\leq L^\sigma l_{g_{i,r}}(\gamma),
    \end{equation*}
    where $\sigma = \lceil r \rceil$.
    The fact that~$\gamma$ is almost length-minimising further yields 
    \begin{equation*}
        l_{g_{i,r}}(\gamma)-\varepsilon\leq l_{g_{i,r}}(\gamma_i) = l_{g_i}(\gamma_i).
    \end{equation*}
    Taking a limit as~$\varepsilon$ tends to zero, we obtain that the induced length metrics satisfy
    \begin{equation*}
        L^{-\sigma} d_{g_i}(x,y) \leq d_{g_{i,r}}((x,i),(y,i))\leq d_{g_i}(x,y),
    \end{equation*}
    which proves the first statement.

The fact that the leaves $\R^2\times\{i\}$ are coarsely embedded follows from Lemma~\ref{vertical geodesics in manifold}, and, as $L>1$ depended only on $r>0$, the control functions do not depend on $i\in\Z$, as desired.
 \end{proof}

\begin{corollary}
    $(\R^3,g)$ is a complete Riemannian manifold.
\end{corollary}
\begin{proof}
    By the Hopf--Rinow Theorem, it suffices to prove that every closed and bounded subset of $(\R^3,g)$ is compact. 
    From Lemma~\ref{vertical geodesics in manifold} we know that any bounded subset of $(\R^3,g)$ is contained in a subset $\R^2\times J$ for some bounded interval $J\subset \R$. By the proof of Lemma \ref{Leaves are properly embedded}, the length metrics on all of the leaves in $\R^2\times J$ are uniformly bi-Lipschitz to $\HH^2$. In particular, the metric space $(\R^2\times J, d_{g}|_{\R^2\times J})$ is bi-Lipschitz to $\HH^2\times J$ equipped with the product metric, which is readily complete.
\end{proof}
Observe that, again by the Hopf--Rinow theorem, $(\R^3, g)$ also is geodesically complete.
We now complete the proof of Theorem~\ref{Riemannian manifold} by proving that the manifolds constructed above are quasiisometric to the corresponding $\PD^3(R)$ groups.

 \begin{proof}[Proof of Theorem \ref{Riemannian manifold}]
     Let $G$ be a $\PD^3(R)$ group and $\phi\colon G\to\R$ be a nontrivial homogeneous quasimorphism such that there is a coarse equivalence $\theta\colon \qker(\phi)\to\HH^2$ with control functions $\eta_\pm\colon\R_{\geq 0}\to\R_{\geq 0}$. Let~$(\R^3, g)$ be the Riemannian manifold described above. We define a map
     \begin{align*}
         F\colon \bigsqcup_{i\in\Z}c^i.\qker(\phi) & \to (\R^3,g)\\
         c^i.x & \mapsto (f^i(\theta(x)),i),
    \end{align*}
     which we show to be a quasiisometry.
     Coarse surjectivity follows readily from Lemma~\ref{vertical geodesics in manifold}, so that it suffices to prove that it is a coarse embedding as both domain and target are quasigeodesic metric spaces.

     Observe that there is an action of~$\langle c\rangle$ on~$(\R^3, g)$ by isometries, where $c^j$ acts by $ f^{-j}\times(t\mapsto t+j)$, as well as a natural one on~$G$ by left-multiplication. One easily confirms that~$F$ is equivariant with respect to this group action, hence it suffices to prove that~$F$ has controlled distance distortion for pairs of points $x, c^i y$ with $x,y\in \qker(\phi)$ and~ $i\in \Z $.    
  
    Denote $x_\theta=\theta(x)$ and $y_\theta=\theta(y)$, then the images of $x,c^i.y$ under $F$ are $(x_\theta,0),(f^i(y_\theta),i)$ respectively.

     Observe that~$\theta$ was called~$\theta_1$ in the previous section. By Lemma~\ref{Leaves are properly embedded}, there are proper functions $\rho_{\pm}\colon\R_{\geq0}\to\R_{\geq 0}$ such that 
     \[\rho_-(d_{{\HH^2}}(x_\theta,f^i(y_\theta)))\leq d_g((x_\theta,0),(f^{i}(y_\theta),0))\leq\rho_+(d_{{\HH^2}}(x_\theta,f^i(y_\theta))).\]
     An application of the triangle inequality together with Lemma \ref{vertical geodesics in manifold} then implies
     \[d_g((x_\theta,0),(f^i(y_\theta),i))\leq \rho_+(d_{\HH^2}(x_\theta,f^{i}(y_\theta))) + i .\]

    For each $i\in \Z$, let~$\pi_i$ be a closest-point projection from~$c^i.\qker(\phi)$ onto~$\qker(\phi)$.    
    Proposition \ref{cor:filtration_vs_quasimorphism-quantitative} yields a proper function $\Lambda\colon \R_{\geq 0}\to\R_{\geq 0}$ such that  
     \[d_G(c^i.y,\pi_{i}(c^i.y))\leq \Lambda(i).\] 
     
     By Lemma~\ref{lem:phi_for_circle_action_is_a_group_homomorphism}, there exists~$C(i)\geq 0$ such that~$f^i$ is $C(i)$-close to $f_{c^i}=\theta\circ \pi_{i}\circ c^{i}\circ \theta^{-1}$, where we choose the closest-point projection~$\pi_{c^i}$ in~\eqref{eq:define_f_g_new} to be~$\pi_i$. Precomposing with~$\theta$, we see that $f^{i}\circ \theta$ is $C(i)$-close to $\theta\circ\pi_i\circ c^{i}$, so that 
     \begin{align*}
         d_{\HH^2}(x_\theta,f^{i}(y_\theta))) & \leq d_{\HH^2}(x_\theta,\theta(\pi_i(c^{i}(y))))+C(i)\\
         & \leq \eta_+(d_G(x,\pi_i(c^{i}.y)))+C(i).
     \end{align*} 
     Note now that $ d_G(x,\pi_i(c^{i}.y))\leq d_G(x,c^i.y)+ \Lambda(i)$ by the triangle inequality, from which we conclude that 
     \begin{align*}
         d_g((x_\theta,0),(f^i(y_\theta),i)) & \leq \rho_+(d_{\HH^2}(x_\theta,f^{i}(y_\theta))) + i \\
         &\leq \rho_+(\eta_+(d_G(x,\pi_i(c^{i}.y)))+C(i))+i\\
         &\leq \rho_+(\eta_+(d_G(x,c^{i}.y)+\Lambda(i))+C(i))+i . 
     \end{align*}

    Analogously, one derives a lower bound, 
    \begin{align}\label{eq:lower_bound_on_control_function_coarse_equivalence_group_and_manifold}
         d_g((x_\theta,0),(f^i(y_\theta),i)) 
         & \geq  \rho_-(\eta_-(d_G(x,c^{i}.y)-\Lambda(i))-C(i)) - i .
    \end{align}
    In addition, observe that Lemma~\ref{vertical geodesics in manifold} implies that 
    \begin{equation*}
        d_g((x_\theta,0),(f^i(y_\theta),i))\geq i,
    \end{equation*}
    so that one can set the lower control function of~$F$ to be the maximum of~$i$ and the bound derived above. 
    
    It remains to verify that these bounds depend only on $d_G(x,c^i.y)$, and that the lower control function is proper. By Proposition \ref{cor:filtration_vs_quasimorphism-quantitative}, $i$ can be bounded from above by a function of $d_G(x,c^i.y)$. Further, if the maximum in the lower control function takes the value $i$, it must be that
     \[i\geq \rho_-(\eta_-(d_G(x,c^{i}.y)-\Lambda(i))-C(i)) - i.\] 
     Observe that $\rho_-, \eta_-\colon\R_{\geq 0}\to\R_{\geq 0}$ are proper and may be chosen to be increasing, and that $\Lambda(i),C(i)\geq 0$, hence we can rearrange and pass these proper functions over the inequality, so that the new left-hand side is a proper function of $i$ bounded below by~$d_G(x,c^i.y)$. It follows that~$i$ can be bounded from below by a proper function of $d_G(x,c^i.y)$. We conclude that both control functions can be made to depend only on $d_G(x,c^i.y)$. To see that the lower control function is proper, note that if the term given by the right-hand side of~\eqref{eq:lower_bound_on_control_function_coarse_equivalence_group_and_manifold} in the maximum is bounded from above on some sequence with $d_G(x_j,c^{i_j}.y_j)\to\infty$, it must be that $i_j\to\infty$ as subtracting nonnegative functions of~$i$ constitute the only possible source of non-properness. In particular, taking the maximum of the right-hand side of~\eqref{eq:lower_bound_on_control_function_coarse_equivalence_group_and_manifold} and~$i$ gives a proper function.
 \end{proof}

\section{Riemannian geometry of iterated
mapping cylinders}\label{section:bounded_geometry}

In this section, we discuss the geometry of the manifolds described in Section~\ref{section:constructing_Riemannian_manifolds}. For convenience, we introduce the following terminology.
\begin{definition}
    A Riemannian manifold~$(\R^3, g)$ is an \textit{iterated mapping cylinder} if its metric is defined by~\eqref{eq:definition_metric_iterated_mapping_cylinder} with respect to some diffeomorphism~$f$ of~$\HH^2$.
\end{definition}

In particular, under suitable assumptions we show that 
the iterated mapping cylinder~$M=(\R^3, g)$ has bounded geometry in the following sense.
\begin{definition}
    We say that a Riemannian manifold~$M$ has \textit{bounded geometry} if it is complete and there exist constants $a, r_0\in \R_+$ such that the sectional curvatures~$\kappa$ and the injectivity radius~$\mathrm{inj}_M$ satisfy
    \begin{equation}\label{eq:bounds_for_bounded_geometry}
        |\kappa|\leq a, \qquad \mathrm{inj}_M \geq r_0.
    \end{equation}
\end{definition}

We conclude this section by proving that the suitable assumption, which is uniform boundedness of derivatives of~$f$ up to third order, is satisfied when the diffeomorphism comes from the Douady--Earle extension. This allows us to strengthen the results of Section \ref{section:constructing_Riemannian_manifolds} by showing that the $\PD^3$~groups considered in this article are in fact quasiisometric to Riemannian manifolds of bounded geometry.

\begin{convention}\label{conv:symmetric_partition_of_unity}
    In this section, let us assume that the partition of unity subordinate to the open cover~$\mathcal{U}$ from~\eqref{eq:open_cover_for_partition_of_unity} is invariant under the obvious translation action of~$\Z$, that is, there exists a function~$\psi$ such that for every~$i\in\Z$ we have $\psi_i(t) = \psi(t-i)$ for some smooth function $\psi\colon \R\to [0,1]$ with support contained in~$(-2/3, 2/3)$. 
\end{convention}
It is easy to see that such a partition of unity exists, for completeness we provide an explicit construction.
Choose a smooth function $\tilde{\psi}\colon \R\to [0,1]$ with support contained in~$(-2/3, 2/3)$ which is strictly positive on $[-1/2, 1/2]$ and define
\begin{equation*}
    \psi(t):=\frac{\tilde{\psi}(t)}{\sum_{i\in\mathbb Z}\tilde{\psi}(t-i)}.
\end{equation*}
Setting~$\psi_i(t):=\psi(t-i)$ we have that $\sum_{i\in\mathbb Z}\psi_i(t)=1$ for all~$t\in\R$, so the~$\psi_i$ form the desired $\Z$-invariant partition of unity.

The main result of this section is the following.
\begin{proposition}\label{prop:bounded_geometry}
    Let~$M=(\R^3, g)$ be an iterated mapping cylinder, where we assume that the partition of unity satisfies Convention~\ref{conv:symmetric_partition_of_unity}. Let us further assume that~$f$ is a bi-Lipschitz diffeomorphism of~$\HH^2$ with bi-Lipschitz constant~$L$ and that there exists a constant~$C\geq 0$ such that all covariant derivatives~$\nabla^k f$ with respect to~$g_{\HH^2}$ up to third order are uniformly bounded, that is,
    \begin{equation*}
        \sup_{x\in\HH^2}\|\nabla^k f(x)\|\leq C,
    \end{equation*}
    for all $1\leq k\leq 3$, where $\|\cdot\|$ denotes the norm induced by the metric~$g_{\HH^2}$.
    Then~$M$ has bounded geometry, where the constants in~\eqref{eq:bounds_for_bounded_geometry} depend only on~$L, C$ and the function~$\psi$ determining the partition of unity. 
\end{proposition}

From now on, let us fix a diffeomorphism~$f\colon \HH^2\to \HH^2$ and a partition of unity satisfying the assumptions of Proposition~\ref{prop:bounded_geometry} and denote for the remainder of this section the associated iterated mapping cylinder by~$M=(\R^3, g)$.
We continue to parametrise~$M$ by~$(x,t)$, with $x\in \R^2$ a coordinate on the leaf~$\Sigma_t$ and~$t$ parametrising the leaves.
We begin by observing that for all local Riemannian geometric considerations it is sufficient to consider points $p=(x,t)$ with $0\leq t\leq 1$; this is an immediate consequence of translation-invariance of the partition of unity.
\begin{lemma}\label{lem:isometries_between_integer_leaves}
    For every $i\in\Z$ there is an isometry~$F_i$ of~$M$ taking integer leaves to integer leaves, given by
    \begin{equation*}
        F_i((x,t))=(f^{-i}(x),t+i).
    \end{equation*}
\end{lemma}
    \begin{proof}
    By definition of the metric, we have for $(x,t)\in M$,
    \begin{align*}
        F_i^{\ast} g_{(x,t)} &= F_i^{\ast} \left(dt^2+\underset{n\in \Z}{\sum}\psi_n(t)({f^{n}})^\ast (g_{\HH^2})_{x}\right) \\
        &= \left(dt^2+\underset{n\in \Z}{\sum}\psi_{n-i}(t)({f^{n-i}})^\ast (g_{\HH^2})_{x}\right)\\
        &= g_{F_i(x,t)}.\qedhere
    \end{align*}
\end{proof}

Recall that we can write
\begin{equation*}
    g_{(x,t)}=(g_t)_x+dt^2,
\end{equation*}
where for $t\in (-1/3, 4/3)$ we have 
\begin{equation*}
     (g_t)_x=\psi(t)(g_{\HH^2})_x +\psi(t-1)f^{\ast}(g_{\HH^2})_x .
\end{equation*}
Observe that, as~$\psi$ and~$f$ are smooth, $(g_t)_{t\in \R}$ is a smooth family of Riemannian metrics on~$\HH^2$ as we vary~$t$.
The bi-Lipschitz property of~$f$ shows that for all $t\in (-1/3, 4/3)$ the metrics $(g_t+dt^2)$ are bi-Lipschitz equivalent.  
We further obtain uniform control on the determinants of the metrics~$g_t$.

\begin{lemma}\label{lem:determinant_estimates_between_g_ts}
For all $x\in\HH^2$ and all~$t\in\R$ we have 
\begin{equation*}
    L^{-4}\det(g_{\HH^2})_x\leq \det g_{(x,t)}\leq L^4\det(g_{\HH^2})_x.
\end{equation*}
\end{lemma}

\begin{proof}
Fix~$x\in\HH^2$ and let us first consider the case of $t\in (-1/3, 4/3)$. As $\det g_{(x,t)}=\det (g_t)_x$, it is sufficient to prove the two-sided bounds for $\det (g_t)_x$. Since $(g_{\HH^2})_x$ is nondegenerate, we can define a linear map $(A_t)_x\colon T_x\R^2\to T_x\R^2$, defined by
\begin{equation*}
    (g_{\HH^2})_x((A_t)_x u,v)=(g_t)_x(u,v)
\end{equation*}
for $u,v\in T_x\R^2$.
In particular, by the bi-Lipschitz property of~$f$, one has
\begin{equation}\label{eq:biLipschitz_comparision_for_eigenvalues_of_A_x}
    L^{-2} (g_{\HH^2})_x(u,u) \leq (g_{\HH^2})_x((A_t)_x u,u) \leq L^2 (g_{\HH^2})_x(u,u)  
\end{equation}
for all $u\in T_x\R^2$.
Since each $(A_t)_x$ is $(g_{\HH^2})_x$-self-adjoint, there exists a basis of eigenvectors whose eigenvalues lie in $[L^{-2},L^2]$ by~\eqref{eq:biLipschitz_comparision_for_eigenvalues_of_A_x}. It follows that
\begin{equation*}
    L^{-4}\leq \det(A_t)_x\leq L^4. 
\end{equation*} 
Observing that we have $\det(g_t)_x=\det(A_t)_x\det(g_{\HH^2})_x$ yields the desired
estimate for $t\in (-1/3, 4/3)$, and the case of general $t\in\R$ then follows with Lemma~\ref{lem:isometries_between_integer_leaves}.
\end{proof}

\begin{lemma}\label{lem:curvature_bounds}
    There exist constants $\kappa_-, \kappa_+\in \R$ depending only on~$L, C$ such that for any $p\in M$ and $X, Y\in T_p M$ we have
    \begin{equation*}
        \kappa_-\leq \kappa(X, Y)\leq \kappa_+. 
    \end{equation*}
\end{lemma}
We point out that by a classical result (see for instance~\cite[Corollary~6.3.2]{petersen2006riemannian}), a complete Riemannian manifold whose sectional curvature is bounded below by a positive constant is compact, so that it is clear that $\kappa_-\leq 0$.
\begin{proof}
    By Lemma~\ref{lem:isometries_between_integer_leaves}, it is sufficient to prove the lemma for points~$p=(x,t)$ with $t\in [0,1]$. Let us fix therefore such a point and make some observations that allow us to derive the desired curvature bounds.
    For orthonormal tangent vectors~$X, Y\in T_p M$, the sectional curvature of the 2-dimensional subspace spanned by~$X$ and~$Y$ is given by
    \begin{equation*}
        \kappa(X, Y)=g_p(R(X, Y)Y, X),
    \end{equation*}
    where~$R$ is the Riemann curvature tensor. 
    In order to obtain curvature bounds, it is therefore sufficient to bound the terms of the Riemann curvature tensor as well as the metric~$g_p$ uniformly in any coordinate system, where for convenience we choose coordinates in the domain and target which restrict to normal coordinates on the leaves. 

   Writing $f^\alpha, f^\beta$ when expressing~$f$ in coordinates, the coefficients of the pullback metric $f^\ast g_{\mathbb H^2}$ are obtained as
    \begin{equation*}
        (f^\ast g_{\mathbb H^2})_{ij} = \sum_{\alpha, \beta} ((g_{\HH^2})_{\alpha\beta}\circ f)\partial_i f^\alpha \partial_j f^\beta.
    \end{equation*}
    Hence every derivative of $f^\ast g_{\mathbb H^2}$ can be expressed as a finite sum of products involving derivatives of $g_{\mathbb H^2}$ and derivatives of~$f$ up to second order.

    Since $\mathbb H^2$ is a symmetric space, the covariant derivative of its Riemann curvature tensor~$R_{\HH^2}$ satisfies $\nabla R_{\HH^2}\equiv 0$. It then follows from~\cite{eichhorn1991boundedness} that the coefficients of the hyperbolic metric $g_{\mathbb H^2}$, as well as their Christoffel symbols and all their derivatives, are uniformly bounded in normal coordinates of a fixed radius. By assumption, the partial derivatives of~$f$ in normal coordinates are also uniformly bounded, so that it follows that the coefficients of $f^\ast g_{\mathbb H^2}$ and their first and second derivatives are uniformly bounded. 
    Further, as~$\psi$ has compact support, all of its derivatives are bounded. This implies uniform boundedness of the coefficients of 
    \begin{equation*}
        g_t = \psi(t)g_{\mathbb H^2} + \psi(t-1) f^\ast g_{\mathbb H^2},
    \end{equation*}
    and hence of the coefficients and first and second derivatives of the metric
    \begin{equation*}
        g = g_t+dt^2.
    \end{equation*}

    Since $f$ is bi-Lipschitz, the metrics $g_{\mathbb H^2}$, $f^\ast g_{\mathbb H^2}$, and $g_t$ are uniformly bi-Lipschitz equivalent for any $t\in (-1/3, 4/3)$. 
    In addition, \(df\) is invertible with $\|(df)^{-1}\|\leq L$. Differentiating the identity \(f\circ f^{-1}=\mathrm{id}\) shows inductively that all derivatives of \(f^{-1}\) are uniformly bounded whenever the corresponding derivatives of \(f\) are. Together with Lemma~\ref{lem:determinant_estimates_between_g_ts}, it follows that the coefficients of~$g^{-1}$ are also uniformly bounded.
    
    The coefficients of the Riemann curvature tensor are polynomials in the coefficients of~$g^{-1}$, the first derivatives of~$g$, and the second derivatives of~$g$. It now follows that the components of the curvature tensor, and thus sectional curvature, are uniformly bounded on~$M$.
\end{proof}

In the following lemma, we obtain uniform two-sided bounds on volumes of balls. We denote by $B_M(p,r)$ the ball of radius~$r$ at~$p\in M$ with respect to the length metric~$d_{g}$. The estimates will be derived from uniform comparisons with balls in $(\R^2\times  I, g_{\HH^2}+dt^2)$ for a finite interval $I\subset \R$, whose balls we denote analogously by~$B_{\R^2\times I}(p,r)$.

\begin{lemma}\label{lem:uniform_bounds_on_volumes_of_balls}
    There exist functions $V_-, V_+\colon \R_+\to \R_+$ such that for every $(x,t)\in M$ and for every $r\in \R_+$ we have
    \begin{equation*}
        V_-(r) \leq \mathrm{vol}_M (B_M((x,t),r)) \leq V_+(r) . 
    \end{equation*}
\end{lemma}
\begin{proof}
    We will only prove the lemma for $r=1$; by
    a theorem of Bishop and Gromov (see for example~\cite[Lemma~7.1.4]{petersen2006riemannian}) this already implies the general case. 
    Note that, in order for the theorem to apply, one needs a uniform lower bound on Ricci curvature, which follows from the uniform lower bound on sectional curvature of~$M$.
    Moreover, using the isometries~$F_i$ from Lemma~\ref{lem:isometries_between_integer_leaves}, it is sufficient to restrict to the case  $t\in [0,1]$.
    
    Lemma~\ref{vertical geodesics in manifold} implies that there exists a finite interval~$I$ such that for every point  $p=(x,t)\in \R^2\times [0,1] \subset M$, the ball $B_M(p,1)$ as well as all almost length-minimising curves between pairs of points in~$B_M(p,1)$ are contained in $ \R^2\times I$. Denote the embedding of $(\R^2\times I, g_{\HH^2}+dt^2)$ into~$M$ given by the identity map by~$\imath$. 
    The metrics~$g_t$ and~$g_{\HH^2}$ are uniformly bi-Lipschitz equivalent as~$t$ ranges over~$I$,
     so that arguments analogous to those in the proof of Lemma~\ref{Leaves are properly embedded} imply that for every curve~$\gamma$ in $\R^2\times I$, there exists a constant~$L^\prime$ %depending only on~$I$ 
    such that one has
    \begin{equation*}\label{eq:bi_lipschitz_equivalence_of_lengths_of_curves_in_R2_times_I}
        {L^\prime}^{-1}l_g(\imath(\gamma))\leq l_{g_{\HH^2}+dt^2}(\gamma) \leq L^\prime l_g(\imath(\gamma)).
    \end{equation*}
    This implies that there exist constants $0 < r_1 \leq r_2$, not depending on the choice of~$p$, such that 
    \begin{equation*}
        \imath(B_{\R^2\times I}((x,t), r_1))\subset B_M((x,t),1) \subset \imath(B_{\R^2\times I}((x,t), r_2)).
    \end{equation*} 
    
    Recall that the volume of~$B_M(p,r)$ is computed as
    \begin{equation*}
        \mathrm{vol}_M(B_M((x,t),r))=\int_{B_M((x,t),r)} d V = \int_{B_M((x,t),r)} \sqrt{|\det (g)|}dx_1 dx_2 dt.
    \end{equation*}
    Then, using Lemma~\ref{lem:determinant_estimates_between_g_ts}, we obtain
    \begin{align*}
        \mathrm{vol}_M(B_M((x,t),1))&\leq L^2\int_{B_{\R^2\times I}((x,t),r_2)} \sqrt{|\det (g_{\HH^2}+dt^2)|}dx_1 dx_2 dt\\
        &=L^2\mathrm{vol}_{\R^2\times I}(B_{\R^2\times I}((x,t),r_2)).
    \end{align*}
    As $(\R^2\times I, g_{\HH^2}+dt^2)$ is homogeneous, the volume on the right-hand side only depends on~$r_2$.
    The lower bound follows analogously.
\end{proof}

\begin{proof}[Proof of Proposition~\ref{prop:bounded_geometry}]
    We have established upper and lower curvature bounds of an iterated mapping cylinder~$M$ in Lemma~\ref{lem:curvature_bounds}, so that it only remains to prove the existence of a constant $r_0>0$ such that $\mathrm{inj}_M\geq r_0$.
    This, however is an immediate consequence of a result of Cheeger--Gromov--Taylor \cite[Theorem~4.7]{cheeger1982finite}, once one has uniform two-sided sectional curvature bounds and the uniform two-sided bounds on volumes of balls that we established in Lemma~\ref{lem:uniform_bounds_on_volumes_of_balls}.
\end{proof}

We finally prove that, after choosing a suitable partition of unity, the manifolds from Theorem~\ref{Riemannian manifold} satisfy the assumptions of Proposition~\ref{prop:bounded_geometry}. We will phrase this as  a property of Douady--Earle extensions of quasisymmetric homeomorphisms.
\begin{definition}
    A homeomorphism $f:X\to Y$ between metric spaces $(X, d_X)$ and $(Y, d_Y)$ is said to be \emph{$\eta$-quasisymmetric} if there exists a homeomorphism
    $\eta:[0,\infty)\to[0,\infty)$ such that
    \begin{equation*}
            \frac{d_Y(f(x),f(a))}{d_Y(f(x),f(b))} \leq \eta\left(\frac{d_X(x,a)}{d_X(x,b)}\right)
    \end{equation*}
    whenever $a,b,x\in X$ are distinct.
\end{definition}

\begin{lemma}\label{lem:boundedness_of_derivatives_of_f}
Let $\varphi:S^1\to S^1$ be an $\eta$-quasisymmetric homeomorphism, and let
\begin{equation*}
    f=E(\varphi):\HH^2\to\HH^2
\end{equation*}
be its Douady--Earle extension. Then for every integer $k\geq 1$ there exists
a constant $C_k=C_k(\varphi)$ such that
\begin{equation*}
    \sup_{x\in\HH^2}\|\nabla^k f(x)\|\leq C_k,
\end{equation*}
where $\nabla^k f$ denotes the $k$-th covariant derivative with respect to~$g_{\HH^2}$, and $\|\cdot\|$ is the corresponding norm.
\end{lemma}
\begin{proof}
    Fix three distinct points $z_1,z_2,z_3\in S^1$, and let $x\in\HH^2$ be arbitrary. Choose a hyperbolic isometry  $A_x\in\mathrm{Isom}(\HH^2)$ such that
    \begin{equation*}
            A_x(0)=x.
    \end{equation*}
    Since the action of $\mathrm{Isom}(\HH^2)$ on $S^1$ is triply transitive, there exists a hyperbolic isometry    $B_x\in\mathrm{Isom}(\HH^2)$ such that the boundary map
    \begin{equation*}
        \xi_x:=B_x\circ\varphi\circ A_x|_{S^1}
    \end{equation*}
    satisfies
    \begin{equation*}
        \xi_x(z_j)=z_j,\qquad j=1,2,3.
    \end{equation*}   
    We set
    \begin{align*}
        \mathcal{F}_{\varphi}= \{ \xi:S^1\to S^1:\xi=B\circ \varphi\circ A \text{ for $A, B$ extensions of isometries of~$\HH^2$ }&
        \\
         \text{ and }\xi(z_j)=z_j \text{ for } j=1,2,3\}.&
    \end{align*}
    Every map in~$\mathcal{F}$ is $\eta_0$-quasisymmetric for some~$\eta_0=\eta_0(\eta)$.
    By a classical compactness result for normalised families of quasisymmetric homeomorphisms, $\mathcal{F}_{\varphi}$ is compact in the uniform topology, that is,  in the topology induced by the $\|\cdot\|_\infty$-norm, see for instance \cite[Chapter~10]{heinonen2001lectures}.

    Let us parametrise~$\HH^2$ in with one complex coordinate~$w$ in the disk model. 
    By Proposition~2 of~\cite{MR857678}, the map
    \begin{align*}
        h\colon \HH^2\times\mathrm{Homeo}(S^1)&\to\HH^2,\\
        (w,\xi)&\mapsto E(\xi)(w),
    \end{align*}
    is continuous with respect to the product topology, where $\mathrm{Homeo}(S^1)$ is equipped with the uniform topology. Moreover, they prove that and all partial derivatives of~$h$ with respect to~$w$ and~$\overline{w}$ are continuous. In particular, partial derivatives of~$h$ with respect to any coordinates on~$\HH^2$ are continuous maps of $\HH^2\times\mathrm{Homeo}(S^1)$ into~$\C$. 
    It follows that for $f=E(\varphi)$, the covariant derivatives~$\nabla^k f$ exist for all~$k\geq 1$ and, in coordinates, depend continuously on the corresponding partial derivatives of~$h(\cdot, \varphi)$ and the Christoffel symbols of~$\HH^2$ and their derivatives. 
  
    We then observe that the Douady--Earle extension satisfies
    \begin{equation*}
        E(\xi_x)=B_x\circ f\circ A_x
    \end{equation*}
    for isometries~$A_x$ and~$B_x$ of~$\HH^2$~\cite[Section~3]{MR857678}. As isometries preserve the Levi--Civita connection and therefore the norms of all covariant derivatives with respect to the hyperbolic metric,  we have
    \begin{equation*}
        \|\nabla^k f(x)\|=\|\nabla^k E(\xi_x)(0)\|
    \end{equation*}
    for all~$k\geq 1$, where~$\|\cdot\|$ denotes the corresponding tensor norm.
    Moreover, from the continuity of partial derivatives of~$h$, by the discussion above relating coordinate derivatives to covariant derivatives it follows that for every integer $k\geq 1$ the map
    \begin{equation*}
        \xi\longmapsto \nabla^k E(\xi)(0)
    \end{equation*}
    is continuous on $\mathrm{Homeo}(S^1)$. Since $\mathcal{F}_{\varphi}$ is a compact subset, there exists a constant~$C_k\geq 0$ depending only on~$\varphi$, such that
    \begin{equation*}
        \|\nabla^k E(\xi)(0)\|\leq C_k
    \end{equation*}
    for all $\xi\in\mathcal{F}_{\varphi}$.
    Applying this estimate to $\xi=\xi_x$ yields
    \begin{equation*}
        \|\nabla^k f(x)\|\leq C_k.
    \end{equation*}
    Since $x\in\HH^2$ was arbitrary, this completes the proof.
\end{proof}

\begin{remark}
    Lemma~\ref{lem:boundedness_of_derivatives_of_f} implies that in the case of~$f$ coming from the Douady--Earle extension, the associated iterated mapping cylinder satisfies in fact a stronger notion of bounded geometry than the one introduced above, which replaces uniform bounds on sectional curvature with uniform bounds on all derivatives of the Riemann curvature tensor~$R$.
\end{remark}

We are now ready to conclude the proof of our main result.

\numberedtheorem{Theorem~\ref{thm:manifold}}{}{Let $G$ be a $\PD^3$ group and $\phi\colon G\to\R$ be a nontrivial homogeneous quasimorphism. If $\qker(\phi)$ is coarsely connected, then either~$G$ is the fundamental group of a torus or Klein-bottle bundle over~$S^1$, or~$\qker(\phi)$ is coarsely equivalent to $\HH^2$. In the latter case, $G$ is quasiisometric to a Riemannian manifold $(\R^3,g)$ of bounded geometry, and is hence finitely presented.}

\begin{proof}
    Proposition~\ref{thm: approx_kernel_qi_to_H2} combined with the well-known fact that any outer automorphism of $\Z\rtimes\Z$ can be realised by a self-homeomorphism of the torus or Klein bottle yields the dichotomy between $G$ being the fundamental group of a torus or Klein-bottle bundle over $S^1$ and the quasikernel being coarsely equivalent to $\HH^2$. Theorem~\ref{Riemannian manifold} shows that in the second case~$G$ is quasiisometric to a complete Riemannian manifold $(\R^3,g)$, and Proposition~\ref{prop:bounded_geometry} combined with Lemma~\ref{lem:boundedness_of_derivatives_of_f} shows that, when choosing the partition of unity suitably, $(\R^3, g)$ is of bounded geometry. It now follows readily from \cite[Theorem~$9.52$ and Corollary~$9.55$]{MR3753580} that $G$ is coarsely simply-connected and hence finitely presented.
\end{proof}

\begin{remark}
    Theorem \ref{thm:manifold} also holds for $G$ a $\PD^3(R)$ group when $R$ is any PID, but at the cost of obtaining only that $G$ is \textit{virtually} the fundamental group of a torus bundle over $S^1$. To see this, note that by Proposition \ref{thm: approx_kernel_qi_to_H2}, if $\qker(\phi)$ is not coarsely equivalent to $\HH^2$, then $\phi$ is a homomorphism and $\ker(\phi)$ is virtually $\Z^2$. By a standard argument, such a group has a characteristic $\Z^2$-subgroup of finite index, so the result follows.
    \end{remark}

\bibliographystyle{plain}
\bibliography{bibliography.bib}

@unpublished{HyperbolicPD^n,
    author = {Kapovich, Michael and Kleiner, Bruce},
    title = {Boundaries of {G}romov-Hyperbolic spaces satisfying coarse {P}oincaré Duality},
    note = {},
    year= {2004}
}

@unpublished{Quasiplanes,
    author = {Kapovich, Michael and Kleiner, Bruce},    
    title = {Geometry of quasi-planes},
    note = {https://www.math.ucdavis.edu/$\sim$kapovich/ EPR/pd3.pdf} ,
    year = {2004}
}

@article{MR123,
    author = {Margolis, Alexander},    
title = {Coarse homological invariants of metric spaces},
    journal={arXiv preprint},
    pages={arXiv:2411.04745},
    year = {2024}
}

@article{kielak2021isoperimetric,
  title={Isoperimetric inequalities for {P}oincar{\'e} duality groups},
  author={Kielak, Dawid and Kropholler, Peter},
  journal={Proceedings of the American Mathematical Society},
  volume={149},
  pages={4685--4698},
  year={2021}
}

@article{calegari2010boundedcochains3manifolds,
title={Bounded cochains on 3-manifolds}, 
      author={Calegari, Danny},
      year={2001},
      journal={arXiv preprint},
    pages={arXiv:0111270},
}

@article{calegari2009scl,
  title={scl, volume 20 of MSJ Memoirs},
  author={Calegari, Danny},
  journal={Mathematical Society of Japan, Tokyo},
  volume={9},
  year={2009}
}

@article {MapstoZ,
    AUTHOR = {Fisher, Sam P.},
     TITLE = {On the cohomological dimension of kernels of maps to {Z}},
   JOURNAL = {Geom. Topol.},
  FJOURNAL = {Geometry \& Topology},
    VOLUME = {30},
      YEAR = {2026},
    NUMBER = {1},
     PAGES = {373--388},
      ISSN = {1465-3060,1364-0380},
   MRCLASS = {20F65 (16S34 20J05)},
  MRNUMBER = {5020388},
       DOI = {10.2140/gt.2026.30.373},
       URL = {https://doi.org/10.2140/gt.2026.30.373},
}

@book {Brown,
    AUTHOR = {Brown, Kenneth S.},
     TITLE = {Cohomology of groups},
    SERIES = {Graduate Texts in Mathematics},
    VOLUME = {87},
      NOTE = {Corrected reprint of the 1982 original},
 PUBLISHER = {Springer-Verlag, New York},
      YEAR = {1994},
     PAGES = {x+306},
      ISBN = {0-387-90688-6},
   MRCLASS = {20J05 (20-02)},
  MRNUMBER = {1324339},
}

@article {QBNS,
    AUTHOR = {Heuer, Nicolaus and Kielak, Dawid},
     TITLE = {Quasi-{BNS} invariants},
   JOURNAL = {Groups Geom. Dyn.},
  FJOURNAL = {Groups, Geometry, and Dynamics},
    VOLUME = {20},
      YEAR = {2026},
    NUMBER = {1},
     PAGES = {93--105},
      ISSN = {1661-7207,1661-7215},
   MRCLASS = {20F65},
  MRNUMBER = {5032289},
       DOI = {10.4171/ggd/939},
       URL = {https://doi.org/10.4171/ggd/939},
}

@article{PD^nfibring,
    author ={Fisher, Sam P. and Italiano, Giovanni and Kielak, Dawid} ,
    title = {Virtual fibring of {P}oincaré-duality groups} ,
       journal={arXiv preprint},
    pages={arXiv:2506.14666},
    year = {2025}
}

@article {MR2231471,
    AUTHOR = {Sauer, R.},
     TITLE = {Homological invariants and quasi-isometry},
   JOURNAL = {Geom. Funct. Anal.},
  FJOURNAL = {Geometric and Functional Analysis},
    VOLUME = {16},
      YEAR = {2006},
    NUMBER = {2},
     PAGES = {476--515},
      ISSN = {1016-443X},
   MRCLASS = {20F65 (20F16 20J06)},
  MRNUMBER = {2231471},
MRREVIEWER = {Wolfgang Pitsch},
       DOI = {10.1007/s00039-006-0562-y},
       URL = {https://doi.org/10.1007/s00039-006-0562-y},
}

@article {MR3868227,
    AUTHOR = {Li, Xin},
     TITLE = {Dynamic characterizations of quasi-isometry and applications
              to cohomology},
   JOURNAL = {Algebr. Geom. Topol.},
  FJOURNAL = {Algebraic \& Geometric Topology},
    VOLUME = {18},
      YEAR = {2018},
    NUMBER = {6},
     PAGES = {3477--3535},
      ISSN = {1472-2747},
   MRCLASS = {20F65 (20J06 37A20 37B99)},
  MRNUMBER = {3868227},
MRREVIEWER = {Tushar Das},
       DOI = {10.2140/agt.2018.18.3477},
       URL = {https://doi.org/10.2140/agt.2018.18.3477},
}

@inproceedings {MR885095,
    AUTHOR = {Brown, Kenneth S.},
     TITLE = {Finiteness properties of groups},
 BOOKTITLE = {Proceedings of the {N}orthwestern conference on cohomology of
              groups ({E}vanston, {I}ll., 1985)},
   JOURNAL = {J. Pure Appl. Algebra},
  FJOURNAL = {Journal of Pure and Applied Algebra},
    VOLUME = {44},
      YEAR = {1987},
    NUMBER = {1-3},
     PAGES = {45--75},
      ISSN = {0022-4049},
   MRCLASS = {20J05 (11F75 20F05 22E40)},
  MRNUMBER = {885095},
MRREVIEWER = {Ralph Strebel},
       DOI = {10.1016/0022-4049(87)90015-6},
       URL = {https://doi.org/10.1016/0022-4049(87)90015-6},
}

@article{calegari2024zippers,
  title={Zippers},
  author={Calegari, Danny and Loukidou, Ino},
  journal={arxiv preprint, to appear in Geom. Top.},
  pages={arXiv:2411.15610},
  year={2026}
}

@article {MR1293049,
    AUTHOR = {Alonso, Juan M.},
     TITLE = {Finiteness conditions on groups and quasi-isometries},
   JOURNAL = {J. Pure Appl. Algebra},
  FJOURNAL = {Journal of Pure and Applied Algebra},
    VOLUME = {95},
      YEAR = {1994},
    NUMBER = {2},
     PAGES = {121--129},
      ISSN = {0022-4049},
   MRCLASS = {20J99 (20F32)},
  MRNUMBER = {1293049},
MRREVIEWER = {Thomas Brady},
       DOI = {10.1016/0022-4049(94)90069-8},
       URL = {https://doi.org/10.1016/0022-4049(94)90069-8},
}

@inproceedings {MR158375,
    AUTHOR = {Stallings, John},
     TITLE = {On fibering certain {$3$}-manifolds},
 BOOKTITLE = {Topology of 3-manifolds and related topics ({P}roc. {T}he
              {U}niv. of {G}eorgia {I}nstitute, 1961)},
     PAGES = {95--100},
 PUBLISHER = {Prentice-Hall, Inc., Englewood Cliffs, NJ},
      YEAR = {1961},
   MRCLASS = {54.78},
  MRNUMBER = {158375},
MRREVIEWER = {O. G. Harrold},
}

@article {MR2168506,
    AUTHOR = {Kapovich, Michael and Kleiner, Bruce},
     TITLE = {Coarse {A}lexander duality and duality groups},
   JOURNAL = {J. Differential Geom.},
  FJOURNAL = {Journal of Differential Geometry},
    VOLUME = {69},
      YEAR = {2005},
    NUMBER = {2},
     PAGES = {279--352},
      ISSN = {0022-040X},
   MRCLASS = {57P10 (20F65 55M05)},
  MRNUMBER = {2168506},
MRREVIEWER = {John Roe},
       URL = {http://projecteuclid.org/euclid.jdg/1121449108},
}

@article {MR32141849,
    AUTHOR = {Milizia, Francesco and Petrosyan, Nansen and Sisto, Alessandro
              and Vankov, Vladimir},
     TITLE = {Cohomological characterisation of hyperbolicity},
   JOURNAL = {Int. Math. Res. Not. IMRN},
  FJOURNAL = {International Mathematics Research Notices. IMRN},
      YEAR = {2025},
    NUMBER = {24},
     PAGES = {rnaf353, 23},
      ISSN = {1073-7928,1687-0247},
   MRCLASS = {53C23 (20F67)},
  MRNUMBER = {5000799},
       DOI = {10.1093/imrn/rnaf353},
       URL = {https://doi.org/10.1093/imrn/rnaf353},
}

@book {MR1744486,
    AUTHOR = {Bridson, Martin R. and Haefliger, Andr\'{e}},
     TITLE = {Metric spaces of non-positive curvature},
    VOLUME = {319},
 PUBLISHER = {Springer-Verlag, Berlin},
      YEAR = {1999},
     PAGES = {xxii+643},
      ISBN = {3-540-64324-9},
   MRCLASS = {53C23 (20F65 53C70 57M07)},
  MRNUMBER = {1744486},
MRREVIEWER = {Athanase Papadopoulos},
       DOI = {10.1007/978-3-662-12494-9},
       URL = {https://doi.org/10.1007/978-3-662-12494-9},
}

@article {MR3816385,
    AUTHOR = {Margolis, Alexander J.},
     TITLE = {Quasi-isometry invariance of group splittings over coarse
              {P}oincar\'{e} duality groups},
   JOURNAL = {Proc. Lond. Math. Soc. (3)},
  FJOURNAL = {Proceedings of the London Mathematical Society. Third Series},
    VOLUME = {116},
      YEAR = {2018},
    NUMBER = {6},
     PAGES = {1406--1456},
      ISSN = {0024-6115},
   MRCLASS = {20F65 (20E08 20J05 57P10)},
  MRNUMBER = {3816385},
MRREVIEWER = {Vassilis Metaftsis},
       DOI = {10.1112/plms.12117},
       URL = {https://doi.org/10.1112/plms.12117},
}

@article {MR1147350,
    AUTHOR = {Roe, John},
     TITLE = {Coarse cohomology and index theory on complete {R}iemannian
              manifolds},
   JOURNAL = {Mem. Amer. Math. Soc.},
  FJOURNAL = {Memoirs of the American Mathematical Society},
    VOLUME = {104},
      YEAR = {1993},
    NUMBER = {497},
     PAGES = {x+90},
      ISSN = {0065-9266},
   MRCLASS = {58G12 (19K56 55N35)},
  MRNUMBER = {1147350},
MRREVIEWER = {John G. Miller},
       DOI = {10.1090/memo/0497},
       URL = {https://doi.org/10.1090/memo/0497},
}

@incollection {MR1253544,
    AUTHOR = {Gromov, M.},
     TITLE = {Asymptotic invariants of infinite groups},
 BOOKTITLE = {Geometric group theory, {V}ol. 2 ({S}ussex, 1991)},
    SERIES = {London Math. Soc. Lecture Note Ser.},
    VOLUME = {182},
     PAGES = {1--295},
 PUBLISHER = {Cambridge Univ. Press, Cambridge},
      YEAR = {1993},
   MRCLASS = {20F32 (57M07)},
  MRNUMBER = {1253544},
}

@unpublished{MR49321,
    author = {Mess, Geoffrey } ,
    title = {The {S}eifert conjecture and groups which are coarse quasi–isometric to planes},
    note = {https://lamington.wordpress.com/wp-content/uploads/2014/08/mess\_seifert\_conjecture.pdf},
    year = {1990}
}

@article {MR2282258,
    AUTHOR = {Hillman, J. A. and Kochloukova, D. H.},
     TITLE = {Finiteness conditions and {${\rm PD}_r$}-group covers of
              {${\rm PD}_n$}-complexes},
   JOURNAL = {Math. Z.},
  FJOURNAL = {Mathematische Zeitschrift},
    VOLUME = {256},
      YEAR = {2007},
    NUMBER = {1},
     PAGES = {45--56},
      ISSN = {0025-5874},
}

@book {MR3561300,
    AUTHOR = {Cornulier, Yves and de la Harpe, Pierre},
     TITLE = {Metric geometry of locally compact groups},
    SERIES = {EMS Tracts in Mathematics},
    VOLUME = {25},
 PUBLISHER = {European Mathematical Society (EMS), Z\"{u}rich},
      YEAR = {2016},
     PAGES = {viii+235}
}

@book {MR4646531,
    AUTHOR = {Nowak, Piotr W. and Yu, Guoliang},
     TITLE = {Large scale geometry},
    SERIES = {EMS Textbooks in Mathematics},
      NOTE = {Second edition [of  2986138]},
 PUBLISHER = {EMS Press, Berlin},
      YEAR = {2023},
     PAGES = {xvi+197},
      ISBN = {978-3-98547-018-1; 978-3-98547-518-6},
   MRCLASS = {51F30 (19K56 20F69 46L87 53C23 54C20 54E40 58B34)},
  MRNUMBER = {4646531},
}

@article{cordes2020foundations,
  title={Foundations of geometric approximate group theory},
  author={Cordes, Matthew and Hartnick, Tobias and Toni{\'c}, Vera},
  journal={arXiv preprint, accepted for publication in Mem. Am. Math. Soc. },
  pages={arXiv:2012.15303},
  year={2024}
}

@article{hinkkanen1985uniformly,
  title={Uniformly quasisymmetric groups},
  author={Hinkkanen, Aimo},
  journal={Proceedings of the London Mathematical Society},
  volume={3},
  number={2},
  pages={318--318},
  year={1985},
  publisher={Oxford University Press}
}

@article {MR857678,
    AUTHOR = {Douady, Adrien and Earle, Clifford J.},
     TITLE = {Conformally natural extension of homeomorphisms of the circle},
   JOURNAL = {Acta Math.},
  FJOURNAL = {Acta Mathematica},
    VOLUME = {157},
      YEAR = {1986},
    NUMBER = {1-2},
     PAGES = {23--48},
      ISSN = {0001-5962,1871-2509},
   MRCLASS = {30C60 (32G15 58F08)},
  MRNUMBER = {857678},
MRREVIEWER = {William\ Abikoff},
       DOI = {10.1007/BF02392590},
       URL = {https://doi.org/10.1007/BF02392590},
}

@article{MR473827927,
    author = {Thurston, W. P.} ,
    title = {Three-manifolds, Foliations and Circles, {I}},
    journal={arXiv preprint}, pages={arXiv:9712268},
    year = {1997}
}

@inproceedings {MR624833,
    AUTHOR = {Sullivan, Dennis},
     TITLE = {On the ergodic theory at infinity of an arbitrary discrete
              group of hyperbolic motions},
 BOOKTITLE = {Riemann surfaces and related topics: {P}roceedings of the 1978
              {S}tony {B}rook {C}onference ({S}tate {U}niv. {N}ew {Y}ork,
              {S}tony {B}rook, {N}.{Y}., 1978)},
    SERIES = {Ann. of Math. Stud., No. 97},
     PAGES = {465--496},
 PUBLISHER = {Princeton Univ. Press, Princeton, NJ},
      YEAR = {1981},
      ISBN = {0-691-08264-2},
   MRCLASS = {58F11 (22E40 30C60 53C30 57R99)},
  MRNUMBER = {624833},
MRREVIEWER = {M.\ Rees},
}

@article {MR1096169,
    AUTHOR = {Bestvina, Mladen and Mess, Geoffrey},
     TITLE = {The boundary of negatively curved groups},
   JOURNAL = {J. Amer. Math. Soc.},
  FJOURNAL = {Journal of the American Mathematical Society},
    VOLUME = {4},
      YEAR = {1991},
    NUMBER = {3},
     PAGES = {469--481},
      ISSN = {0894-0347,1088-6834},
   MRCLASS = {20F32 (57M40)},
  MRNUMBER = {1096169},
MRREVIEWER = {Jerzy\ Dydak},
       DOI = {10.2307/2939264},
       URL = {https://doi.org/10.2307/2939264},
}

@article {MR4932386,
    AUTHOR = {Bader, Shaked and Kropholler, Robert and Vankov, Vladimir},
     TITLE = {Subgroups of word hyperbolic groups in dimension 2 over
              arbitrary rings},
   JOURNAL = {J. Lond. Math. Soc. (2)},
  FJOURNAL = {Journal of the London Mathematical Society. Second Series},
    VOLUME = {112},
      YEAR = {2025},
    NUMBER = {1},
     PAGES = {Paper No. e70230, 23},
      ISSN = {0024-6107,1469-7750},
   MRCLASS = {20F67 (20F65 57M07)},
  MRNUMBER = {4932386},
       DOI = {10.1112/jlms.70230},
       URL = {https://doi.org/10.1112/jlms.70230},
}

@article{MR3840928094,
    author = {Weis, Jannis} ,
    title = {Quasi-Isometry Invariance of discrete Higher Filling Functions},
    journal={arXiv preprint}, pages={arXiv:2601.15140},
    year = {2026}
}

@article {MR1465330,
    AUTHOR = {Bestvina, Mladen and Brady, Noel},
     TITLE = {Morse theory and finiteness properties of groups},
   JOURNAL = {Invent. Math.},
  FJOURNAL = {Inventiones Mathematicae},
    VOLUME = {129},
      YEAR = {1997},
    NUMBER = {3},
     PAGES = {445--470},
      ISSN = {0020-9910,1432-1297},
   MRCLASS = {20F36 (20J05 57M07)},
  MRNUMBER = {1465330},
MRREVIEWER = {John\ Meier},
       DOI = {10.1007/s002220050168},
       URL = {https://doi.org/10.1007/s002220050168},
}

@phdthesis{Sikorav,
    author = {J.-C. Sikorav} ,
    title = {Homologie de Novikov associ\'ee \'a une classe de cohomologie r\'eelle de degr\'e
un },
    school = {Thes\'e Orsay} ,
    year = {1987}
}

@article {MR4797112,
    AUTHOR = {Fisher, Sam P.},
     TITLE = {Improved algebraic fibrings},
   JOURNAL = {Compos. Math.},
  FJOURNAL = {Compositio Mathematica},
    VOLUME = {160},
      YEAR = {2024},
    NUMBER = {9},
     PAGES = {2203--2227},
      ISSN = {0010-437X,1570-5846},
   MRCLASS = {20F65},
  MRNUMBER = {4797112},
       DOI = {10.1112/S0010437X24007309},
       URL = {https://doi.org/10.1112/S0010437X24007309},
}

@article {MR2283437,
    AUTHOR = {Bieri, Robert},
     TITLE = {Deficiency and the geometric invariants of a group},
      NOTE = {With an appendix by Pascal Schweitzer},
   JOURNAL = {J. Pure Appl. Algebra},
  FJOURNAL = {Journal of Pure and Applied Algebra},
    VOLUME = {208},
      YEAR = {2007},
    NUMBER = {3},
     PAGES = {951--959},
      ISSN = {0022-4049,1873-1376},
   MRCLASS = {20J05 (20F65)},
  MRNUMBER = {2283437},
MRREVIEWER = {J.\ R. J. Groves},
       DOI = {10.1016/j.jpaa.2006.02.003},
       URL = {https://doi.org/10.1016/j.jpaa.2006.02.003},
}

@article {MR914846,
    AUTHOR = {Bieri, Robert and Neumann, Walter D. and Strebel, Ralph},
     TITLE = {A geometric invariant of discrete groups},
   JOURNAL = {Invent. Math.},
  FJOURNAL = {Inventiones Mathematicae},
    VOLUME = {90},
      YEAR = {1987},
    NUMBER = {3},
     PAGES = {451--477},
      ISSN = {0020-9910,1432-1297},
   MRCLASS = {20J05 (22E40 57R19)},
  MRNUMBER = {914846},
MRREVIEWER = {Peter\ H.\ Kropholler},
       DOI = {10.1007/BF01389175},
       URL = {https://doi.org/10.1007/BF01389175},
}

@article {MR960770,
    AUTHOR = {Bieri, Robert and Renz, Burkhardt},
     TITLE = {Valuations on free resolutions and higher geometric invariants
              of groups},
   JOURNAL = {Comment. Math. Helv.},
  FJOURNAL = {Commentarii Mathematici Helvetici},
    VOLUME = {63},
      YEAR = {1988},
    NUMBER = {3},
     PAGES = {464--497},
      ISSN = {0010-2571,1420-8946},
   MRCLASS = {20J05 (20F32 57M99)},
  MRNUMBER = {960770},
MRREVIEWER = {Ralph\ Strebel},
       DOI = {10.1007/BF02566775},
       URL = {https://doi.org/10.1007/BF02566775},
}

@book {MR3753580,
    AUTHOR = {Dru\c{t}u, Cornelia and Kapovich, Michael},
     TITLE = {Geometric group theory},
    SERIES = {American Mathematical Society Colloquium Publications},
    VOLUME = {63},
      NOTE = {With an appendix by Bogdan Nica},
 PUBLISHER = {American Mathematical Society, Providence, RI},
      YEAR = {2018},
     PAGES = {xx+819},
      ISBN = {978-1-4704-1104-6},
   MRCLASS = {20F65 (20E08 20F05 20F16 20F18 20F67 20F69 57M07)},
  MRNUMBER = {3753580},
MRREVIEWER = {Igor\ Belegradek},
       DOI = {10.1090/coll/063},
       URL = {https://doi.org/10.1090/coll/063},
}

@article {MR1452851,
    AUTHOR = {Epstein, David B. A. and Fujiwara, Koji},
     TITLE = {The second bounded cohomology of word-hyperbolic groups},
   JOURNAL = {Topology},
  FJOURNAL = {Topology. An International Journal of Mathematics},
    VOLUME = {36},
      YEAR = {1997},
    NUMBER = {6},
     PAGES = {1275--1289},
      ISSN = {0040-9383},
   MRCLASS = {20J05 (20F32)},
  MRNUMBER = {1452851},
MRREVIEWER = {Thomas\ Delzant},
       DOI = {10.1016/S0040-9383(96)00046-8},
       URL = {https://doi.org/10.1016/S0040-9383(96)00046-8},
}

@unpublished{MR3281809382,
     author = {Margolis, Alexander J.},
     title = {Groups of cohomological codimension one},
     journal = {Annales de l'Institut Fourier},
     year = {2026},
     publisher = {Association des Annales de l{\textquoteright}institut Fourier},
     doi = {10.5802/aif.3766},
     language = {en},
     note = {Online first},
}

@article{calegari2026catherinewheels,
   author={Danny Calegari and Ino Loukidou} ,
    title={CaTherine wheels},
    journal={arXiv preprint}, pages={	arXiv:2604.24619 },
    year = {2026}
}

@article{MR1381603,
    AUTHOR = {Bestvina, Mladen},
     TITLE = {Local homology properties of boundaries of groups},
   JOURNAL = {Michigan Math. J.},
  FJOURNAL = {Michigan Mathematical Journal},
    VOLUME = {43},
      YEAR = {1996},
    NUMBER = {1},
     PAGES = {123--139},
      ISSN = {0026-2285,1945-2365},
   MRCLASS = {57P10 (20J05 54C56 57N25)},
  MRNUMBER = {1381603},
MRREVIEWER = {Fr\'{e}d\'{e}ric\ Paulin},
       DOI = {10.1307/mmj/1029005393},
       URL = {https://doi.org/10.1307/mmj/1029005393},
}

@incollection {MR1130181,
    AUTHOR = {Cannon, James W.},
     TITLE = {The theory of negatively curved spaces and groups},
 BOOKTITLE = {Ergodic theory, symbolic dynamics, and hyperbolic spaces
              ({T}rieste, 1989)},
    SERIES = {Oxford Sci. Publ.},
     PAGES = {315--369},
 PUBLISHER = {Oxford Univ. Press, New York},
      YEAR = {1991},
      ISBN = {0-19-853390-X; 0-19-859685-5},
   MRCLASS = {57S30 (20F32 53C23 57M07)},
  MRNUMBER = {1130181},
}

@book{wall1979homological,
  title={Homological Group Theory},
  author={Wall, C. T. C.},
  year={1979},
  publisher={Cambridge University Press},
  series={London Mathematical Society Lecture Note Series},
  volume={36},
  address={Cambridge},
  isbn={9780521227292}
}

@article{2648284748,
    author = {Calegari, Danny and Zung, Jonathan} ,
    title = {Quasimorphisms and pseudo-{A}nosov flows} ,
        journal={arXiv preprint},
    pages={arXiv:2606.14889},
    year = {2026}
}

@article{eichhorn1991boundedness,
  title={The boundedness of connection coefficients and their derivatives},
  author={Eichhorn, J{\"u}rgen},
  journal={Mathematische Nachrichten},
  volume={152},
  number={1},
  pages={145--158},
  year={1991},
  publisher={Wiley Online Library}
}

@article{cheeger1982finite,
  title={Finite propagation speed, kernel estimates for functions of the {L}aplace operator, and the geometry of complete {R}iemannian manifolds},
  author={Cheeger, Jeff and Gromov, Mikhail and Taylor, Michael},
  journal={Journal of Differential Geometry},
  volume={17},
  number={1},
  pages={15--53},
  year={1982},
  publisher={Lehigh University}
}

@book{petersen2006riemannian,
  title={Riemannian geometry},
  author={Petersen, Peter},
  edition={Third edition (2006)},
  year={1998},
  publisher={Springer}
}

@book{heinonen2001lectures,
  title={Lectures on analysis on metric spaces},
  author={Heinonen, Juha},
  year={2001},
  publisher={Springer Science \& Business Media}
}

@book {MR1943724,
    AUTHOR = {Hillman, J. A.},
     TITLE = {Four-manifolds, geometries and knots},
    SERIES = {Geometry \& Topology Monographs},
    VOLUME = {5},
 PUBLISHER = {Geometry \& Topology Publications, Coventry},
      YEAR = {2002},
     PAGES = {xiv+379},
   MRCLASS = {57N16 (57N13 57Q45)},
  MRNUMBER = {1943724},
MRREVIEWER = {Alberto\ Cavicchioli},
}

@article {MR699010,
    AUTHOR = {Eckmann, Beno and Linnell, Peter},
     TITLE = {Poincar\'{e} duality groups of dimension two. {II}},
   JOURNAL = {Comment. Math. Helv.},
  FJOURNAL = {Commentarii Mathematici Helvetici},
    VOLUME = {58},
      YEAR = {1983},
    NUMBER = {1},
     PAGES = {111--114},
      ISSN = {0010-2571,1420-8946},
   MRCLASS = {57M20 (20F38 57P10)},
  MRNUMBER = {699010},
MRREVIEWER = {R.\ M. F. Moss},
       DOI = {10.1007/BF02564628},
       URL = {https://doi.org/10.1007/BF02564628},
}

@article {MR604709,
    AUTHOR = {Eckmann, Beno and M\"{u}ller, Heinz},
     TITLE = {Poincar\'{e} duality groups of dimension two},
   JOURNAL = {Comment. Math. Helv.},
  FJOURNAL = {Commentarii Mathematici Helvetici},
    VOLUME = {55},
      YEAR = {1980},
    NUMBER = {4},
     PAGES = {510--520},
      ISSN = {0010-2571,1420-8946},
   MRCLASS = {57M20 (57P10)},
  MRNUMBER = {604709},
MRREVIEWER = {R.\ M. F. Moss},
       DOI = {10.1007/BF02566702},
       URL = {https://doi.org/10.1007/BF02566702},
}

@incollection {MR4861506,
    AUTHOR = {Bridson, Martin R. and Kielak, Dawid and Kudlinska, Monika},
     TITLE = {Stallings's fibring theorem and {$\rm PD^3$}-pairs},
 BOOKTITLE = {Topology at infinity of discrete groups},
    SERIES = {Contemp. Math.},
    VOLUME = {812},
     PAGES = {125--132},
 PUBLISHER = {Amer. Math. Soc., [Providence], RI},
      YEAR = {2025},
      ISBN = {978-1-4704-7534-5; [9781470478636]},
   MRCLASS = {20J05 (57K30)},
  MRNUMBER = {4861506},
MRREVIEWER = {Sam\ Hughes},
       DOI = {10.1090/conm/812/16263},
       URL = {https://doi.org/10.1090/conm/812/16263},
}

@article {MR1555345,
    AUTHOR = {Nielsen, Jakob},
     TITLE = {Untersuchungen zur {T}opologie der geschlossenen zweiseitigen
              {F}l\"{a}chen. {III}},
   JOURNAL = {Acta Math.},
  FJOURNAL = {Acta Mathematica},
    VOLUME = {58},
      YEAR = {1932},
    NUMBER = {1},
     PAGES = {87--167},
      ISSN = {0001-5962,1871-2509},
   MRCLASS = {DML},
  MRNUMBER = {1555345},
       DOI = {10.1007/BF02547775},
       URL = {https://doi.org/10.1007/BF02547775},
}

@article {MR1545786,
    AUTHOR = {Mangler, W.},
     TITLE = {Die {K}lassen von topologischen {A}bbildungen einer
              geschlossenen {F}l\"{a}che auf sich},
   JOURNAL = {Math. Z.},
  FJOURNAL = {Mathematische Zeitschrift},
    VOLUME = {44},
      YEAR = {1939},
    NUMBER = {1},
     PAGES = {541--554},
      ISSN = {0025-5874,1432-1823},
   MRCLASS = {DML},
  MRNUMBER = {1545786},
       DOI = {10.1007/BF01210672},
       URL = {https://doi.org/10.1007/BF01210672},
}

\end{document}